\documentclass[12pt,a4paper]{amsart}
\usepackage{amsmath,amsfonts,amssymb,amsthm,amscd}
\usepackage{dsfont,mathrsfs}
\usepackage{curves,epic,eepic}
\usepackage{graphicx}
\usepackage{rotating}

\newtheorem{thm}{Theorem}[section]
\newtheorem{lemma}[thm]{Lemma}
\newtheorem{prop}[thm]{Proposition}
\newtheorem{cor}[thm]{Corollary}
\theoremstyle{definition}
\newtheorem{defi}[thm]{Definition}
\newtheorem{rem}[thm]{Remark}
\newtheorem{ex}[thm]{Example}
\newtheorem{con}[thm]{Construction}

\def\Z{\mathds Z}
\def\Q{\mathds Q}
\def\phi{\varphi} 
\def\rho{\varrho} 
\def\id{{\it id}}
\def\A{\mathds A}
\def\P{\mathds P}
\def\G{\mathds G}
\def\O{\mathcal O}

\DeclareMathSymbol{\minus}{\mathord}{operators}{"2D}
\DeclareMathOperator{\Mor}{Mor}
\DeclareMathOperator{\Hom}{Hom}
\DeclareMathOperator{\Isom}{Isom}

\DeclareMathOperator{\Proj}{Proj}
\DeclareMathOperator{\Div}{Div}
\DeclareMathOperator{\divis}{div}
\DeclareMathOperator{\Sym}{Sym}
\DeclareMathOperator{\Id}{Id}
\DeclareMathOperator{\conv}{conv}
\DeclareMathOperator{\QCoh}{QCoh}
\DeclareMathOperator{\DG}{DG}

\begin{document}

\title[Toric orbifolds associated to Cartan matrices]
{Toric orbifolds associated to Cartan matrices} 
\author[Mark Blume]{Mark Blume}
\thanks{\newline\indent 
Supported by DFG-Schwerpunkt 1388 Darstellungstheorie.
\newline}
\address{Mathematisches Institut, Universit\"at M\"unster,\newline
Einsteinstrasse 62, 48149 M\"unster, Germany}
\email{mark.blume@uni-muenster.de}

\begin{abstract}
We investigate moduli stacks of pointed chains
of projective lines related to the Losev-Manin moduli spaces
and show that these moduli stacks coincide with certain toric 
stacks which can be described in terms of the Cartan matrices 
of root systems of type $A$. We also consider variants of 
these stacks related to root systems of type $B$ and $C$.
\end{abstract}

\maketitle

\setlength{\parskip}{2mm}
\tableofcontents
\setlength{\parskip}{0mm}

\section*{Introduction}
\medskip

The Losev-Manin moduli spaces $\overline{L}_n$, introduced in 
\cite{LM00}, parametrise isomorphism classes of stable $n$-pointed 
chains of $\P^1$. The space $\overline{L}_n$ forms a 
compactification of the torus $(\G_m)^n/\G_m$ that parametrises 
$n$ points $s_1,\ldots,s_n$ in $\P^1\setminus\{0,\infty\}=\G_m$ 
up to automorphisms of $\P^1$ fixing the two points $0,\infty$. 
It is a smooth projective toric variety isomorphic to the 
toric variety $X(A_{n-1})$ associated with the root system 
$A_{n-1}$, see \cite{BB11a}.

\pagebreak

In the present paper we are concerned with a variant of the 
Losev-Manin moduli spaces which arises as a compactification
of the moduli space of $n$ indistinguishable points in 
$\P^1\setminus\{0,\infty\}$, or equivalently, finite 
subschemes of degree $n$ in $\P^1\setminus\{0,\infty\}$.
Isomorphism classes of such subschemes correspond to 
polynomials of the form $y^n+a_{n-1}y^{n-1}+\ldots+a_1y+1$ 
up to multiplication of the variable $y$ by an $n$-th root 
of unity.
The torus $(\G_m)^{n-1}$, parametrising polynomials with 
non-zero coefficients $a_1,\ldots,a_{n-1}$, is compactified 
by the moduli stack of chains of $\P^1$ with finite 
subschemes of degree $n$. 
On the boundary both the coefficients of the polynomials
may become zero and the curve may become a reducible chain 
of $\P^1$.
The category of these pointed curves, which we call 
degree-$n$-pointed chains of $\P^1$, forms an 
orbifold $\overline{\mathcal L}_n$. 

\medskip

The orbifold $\overline{\mathcal L}_n$ is related to the Losev-Manin 
moduli space $\overline{L}_n$ by an $S_n$-equi-variant morphism 
$\overline{L}_n\to\overline{\mathcal L}_n$, $\overline{L}_n$
with the operation of the symmetric group $S_n$ that permutes the 
$n$ sections and $\overline{\mathcal L}_n$ with trivial operation, 
which is given by mapping an $n$-pointed chain of $\P^1$
to the corresponding degree-$n$-pointed chain by forgetting the 
labels of the sections. 
The moduli stack $\overline{\mathcal L}_n$ is defined such that 
the morphism $\overline{L}_n\to\overline{\mathcal L}_n$ is closely 
related to morphisms of the form $C_0^n\to C_0^n/S_n=C_0^{(n)}
=\Div_{C_0/Y}^n$ from the $n$-fold product over $Y$ to the scheme 
of relative effective Cartier divisors of degree $n$ for 
$C_0\to Y$ a relative smooth curve over $Y$, here a chain of 
$\P^1$ over $Y$ without the poles of the components 
of the fibres.
Therefore the morphism $\overline{L}_n\to\overline{\mathcal L}_n$ 
inherits properties like being faithfully flat and finite of degree 
$n!$ and  being ramified exactly in the points corresponding to curves 
with coinciding marked points, see proposition \ref{prop:L_n-L_n}.
The stack $\overline{\mathcal L}_n$ differs from the quotient stack 
$[\overline{L}_n/S_n]$, it has the same points but different 
automorphism groups. The coarse moduli space of 
$\overline{\mathcal L}_n$ coincides with the quotient 
$\overline{L}_n/S_n$.

\medskip

A main result of this paper, theorem \ref{thm:modulistack-toricstack}, 
is an explicit description of the structure of the stacks 
$\overline{\mathcal L}_n$: we show that $\overline{\mathcal L}_n$ 
is a toric orbifold and we determine the associated combinatorial data.

\medskip

Toric Deligne-Mumford stacks over fields of characteristic 
$0$ were introduced in \cite{BCS05} and constructed from 
combinatorial data called (simplicial) stacky fans,   
consisting of a simplicial fan and some extra data, as  
quotient stacks $[U/T]$ of an open subscheme $U$ of some affine space 
by a diagonalisable group scheme $G$, generalising the quotient 
construction of a smooth toric variety described in \cite{Co95a}. 
Over more general base schemes in the same way these data give 
rise to toric stacks which are not necessarily Deligne-Mumford stacks
but tame stacks in the sense of \cite{AOV08}.
As our moduli problem results in stacks which are orbifolds, in this 
paper we are mainly concerned with toric orbifolds, i.e.\ toric tame 
stacks with trivial generic stabiliser. We will work with toric 
orbifolds over the integers, considering the fact that our moduli 
problem is naturally defined over the integers.

\bigskip\bigskip

\pagebreak

It turns out that the moduli stacks $\overline{\mathcal L}_n$ can 
be described in terms of the Cartan matrices of root systems of 
type $A$, more precisely, $\overline{\mathcal L}_n$ is isomorphic 
to the toric orbifold $\mathcal Y(A_{n-1})$ which corresponds to 
the stacky fan $\mathbf\Upsilon(A_{n-1})$ defined in section 
\ref{sec:toricorbifoldsAn} using the Cartan matrix of the root 
system $A_{n-1}$.
For the proof of the isomorphism $\overline{\mathcal L}_n\cong
\mathcal Y(A_{n-1})$ we make use of a generalisation of the description 
of the functor of toric varieties \cite{Co95b} for toric stacks, 
which allows to characterise $\mathcal Y(A_{n-1})$ as a stack 
$\mathcal C_{\mathbf\Upsilon(A_{n-1})}$ of 
$\mathbf\Upsilon(A_{n-1})$-collections, i.e.\ collections of
pairs of a line bundle with a section and additional data.

\medskip

We also characterise the morphism $\overline{L}_n\to\overline{\mathcal L}_n$,
determined by forgetting the labels of the $n$ sections, in terms of the 
combinatorial data by specifying the $\mathbf\Upsilon(A_{n-1})$-collection 
on $X(A_{n-1})\cong\overline{L}_n$ corresponding to this morphism, see 
theorem \ref{thm:datamorphLM}.
In doing this, in section \ref{sec:LM} we compare the description of the
functor of the toric varieties $X(A_{n-1})$ associated with root systems 
of type $A$ after Cox \cite{Co95b} in terms of $\Sigma(A_{n-1})$-collections 
to two other descriptions: the description of \cite{BB11a} in terms of 
$A_{n-1}$-data and a new description involving $S_n$-invariant line 
bundles on $X(A_{n-1})$. Both of these are related to Minkowski sum 
decompositions of the permutohedron: the first is a decomposition into 
line segments and the second corresponds to an embedding 
$X(A_{n-1})\to\prod_{j=1}^{n-1}\P^{\binom{n}{j}-1}$ 
and expresses the permutohedron as sum of $S_n$-symmetric polytopes.

\medskip

Generalisations of the Losev-Manin moduli spaces were investigated
in \cite{BB11b}. We considered $(2n\!+\!1)$-pointed and $2n$-pointed 
chains of $\P^1$ with involution and showed that the 
moduli spaces $\overline{L}_n^{0,\pm}$ and $\overline{L}_n^{\pm}$ 
of these objects coincide with the toric varieties $X(B_n)$ and 
$X(C_n)$ associated with the root systems $B_n$ and $C_n$,
see \cite[Thm.\ 4.1 and 6.15]{BB11b}.

\medskip

In the present setting it makes sense to investigate similar 
generalisations of the moduli stacks $\overline{\mathcal L}_n$
and to relate these to the toric orbifolds $\mathcal Y(R)$
for root systems $R$ belonging to other classical families
as well as to the moduli spaces $\overline{L}_n^{0,\pm}\cong X(B_n)$ 
and $\overline{L}_n^{\pm}\cong X(C_n)$.
In section \ref{sec:BC} we consider moduli stacks of stable 
degree$\,$-$(2n\!+\!1)$-pointed and  
degree-$2n$-pointed chains of $\P^1$ with involution, 
$\overline{\mathcal L}_n^{0,\pm}$ and $\overline{\mathcal L}_n^{\pm}$. 
We show that $\overline{\mathcal L}_n^{\pm}$ has a main
component $\overline{\mathcal L}_{n,+}^{\pm}$ isomorphic to 
$\mathcal Y(C_n)$ and that $\overline{\mathcal L}_n^{0,\pm}$ 
is isomorphic to $\mathcal Y(B_n)^{\textrm{can}}$, the 
canonical stack associated to $\mathcal Y(B_n)$ (see \cite{FMN10}).
We have morphisms
$\overline{L}_n^{0,\pm}\to\overline{\mathcal L}_n^{0,\pm}$
and $\overline{L}_n^{\pm}\to\overline{\mathcal L}_{n,+}^{\pm}$,
defined by forgetting the labels of the sections, which are
equivariant under the Weyl group.

\medskip

\noindent
{\it Acknowledgements.} 
Thanks to Victor Batyrev.

\bigskip\bigskip

\pagebreak
\section{Moduli stacks of degree-$n$-pointed chains}
\label{sec:modulistacks}
\medskip

We define moduli stacks of stable degree-$n$-pointed chains of $\P^1$.
Compared to the Losev-Manin moduli spaces considered in \cite{LM00}, 
\cite{BB11a}, we replace the $n$ marked points $s_1,\ldots,s_n$ of 
an $n$-pointed chain of $\P^1$ by a finite closed subscheme $S$ of 
degree $n$.

\begin{defi}\label{def:L_n}
A stable degree-$n$-pointed chain of $\P^1$ 
over an algebraically closed field $K$ is a tuple
$(C,s_-,s_+,S)$, where $C$ is a chain of $\P^1$ 
over $K$ with two distinct closed points $s_-,s_+$ on 
the outer components such that on each component the number of 
intersection points together with $s_-,s_+$ adds up to $2$
(cf.\ \cite[Def.\ 3.1]{BB11a}), 
and $S\subset C$ a finite closed subscheme of degree $n$ 
that does neither meet the intersection points of components 
nor $s_-,s_+$, but that does meet every component of $C$.
We define the category $\overline{\mathcal L}_n$ of stable 
degree-$n$-pointed chains of $\P^1$ over the category 
of schemes.
The objects over a scheme $Y$ are stable degree-$n$-pointed 
chains of $\P^1$ over $Y$, i.e.\ tupels 
$\mathscr C=(C\!\to\!Y,s_-,s_+,S)$, where $C\to Y$ is a 
locally finitely presented, flat, proper 
morphism of schemes, $s_-,s_+\colon Y\to C$ are sections 
and $S\subset C$ is a subscheme finite flat over $Y$, such that 
the geometric fibres are stable degree-$n$-pointed chains of $\P^1$.
We have the natural notion of isomorphism of degree-$n$-pointed 
chains of $\P^1$ over the same scheme $Y$ and of 
pullback of an object over a scheme $Y$ with respect to a 
morphism $f\colon Y'\to Y$.
A morphism in $\overline{\mathcal L}_n$ over a morphism 
$f\colon Y'\to Y$ is a cartesian diagram of stable degree-$n$-pointed 
chains of $\P^1$ over $f$.
\end{defi}

\begin{rem}
(1) For a chain of $\P^1$ $(C,s_-,s_+)$ over a field $K$
any component is isomorphic to $\P^1_K$ since it contains a point
with residue field $K$.\\
(2) As the morphisms $C\to Y$ are locally finitely presented, 
by \cite[IV, (8.9.1)]{EGA} we can use some results which originally  
require some noetherian hypothesis.
\end{rem}

\begin{rem}\label{rem:autom}
The automorphism group of a chain of $\P^1$ $(C,s_-,s_+)$
of length $l$ over a field $K$ is a torus $(\G_m)_K^l$. 
A stable degree-$n$-pointed chain of $\P^1$
$(C,s_-,s_+,S)$ of length $l$ over $K$ has a finite automorphism 
group scheme which is a subgroup scheme of $(\G_m)_K^l$. 
There are objects $(C,s_-,s_+,S)$ having nontrivial automorphisms:
consider for example $\P^1_K$ with homogeneous
coordinates $z_0,z_1$, two poles $s_-\!=\!(1\!:\!0)$, 
$s_+\!=\!(0\!:\!1)$ and a subscheme $S$ of degree $k$ given by 
the equation $z_0^k-z_1^k=0$; in this example we have an 
automorphism group scheme isomorphic to $\mu_k$.
\end{rem}

\begin{prop}\label{prop:L_n-stack}
The category $\overline{\mathcal L}_n$ is a category fibred in
groupoids over the category of schemes. It forms a stack over the 
fpqc site of schemes with representable, finite diagonal. 
Over fields of characteristic $0$ the diagonal is unramified.
\end{prop}
\begin{proof}
The category $\overline{\mathcal L}_n$ together with the
natural functor to the category of schemes is a fibred category,
the cartesian arrows being cartesian diagrams of degree-$n$-pointed 
chains, and moreover the fibres $\overline{\mathcal L}_n(Y)$ over 
schemes $Y$ form a groupoid.
 
The fibred category $\overline{\mathcal L}_n$ is a prestack in the 
fpqc topology, i.e.\ descent data for morphisms are effective, see 
for example \cite[Prop.\ 4.31]{Vi05}. 
To show that $\overline{\mathcal L}_n$ is a stack, it remains
to show that descent data for objects are effective. 
Let $(\pi\colon C\to Y,s_-,s_+,S)$ be a stable degree-$n$-pointed 
chain of $\P^1$ over a scheme $Y$. 
The subscheme $S\subset C$ is an effective Cartier divisor in $C$ 
because this is true on the fibres, see \cite[Lemma 9.3.4]{Kl05}, 
and so its ideal sheaf $\mathscr I\subset\O_C$ is a line bundle. 
The line bundle $\O_C(S)=\mathscr I^{-1}$ is relatively ample with 
respect to $\pi$ since it is ample on the fibres, see 
\cite[III, (4.7.1)]{EGA}, \cite[IV, (9.6.5)]{EGA}.
In fact, $\O_C(S)$ defines a closed embedding in the
projective bundle $\P_Y(\pi_*\O_C(S))$, see proposition \ref{prop:embY}.
Given a morphism $F\colon(C'\!\to\!Y',s_-',s_+',S')\to
(C\!\to\!Y,s_-,s_+,S)$ of two degree-$n$-pointed chains over 
a morphism $f\colon Y'\to Y$ forming a cartesian diagram, we 
have a natural isomorphism $F^*\O_C(S)\cong\O_{C'}(S')$, and 
further, given morphisms $F$ and $G$ over $f\colon Y'\to Y$ 
and $g\colon Y''\to Y'$, after identifying $(FG)^*\O_C(S)$ with 
$G^*F^*\O_C(S)$ the isomorphisms $(FG)^*\O_C(S)\to\O_{C''}(S'')$ 
and $G^*F^*\O_C(S)\to G^*\O_{C'}(S')\to\O_{C''}(S'')$ coincide.
Then, by descent theory of flat proper morphisms of schemes 
with a relatively ample invertible sheaf, see \cite[Thm.\ 4.38]{Vi05}, 
descent data for objects of $\overline{\mathcal L}_n$ are effective.

We show that the diagonal $\overline{\mathcal L}_n\to
\overline{\mathcal L}_n\times\overline{\mathcal L}_n$ 
is representable and finite. 
For a scheme $Y$ and a morphism $Y\to\overline{\mathcal L}_n\times 
\overline{\mathcal L}_n$ given by two objects $\mathscr C,\mathscr C'
\in\overline{\mathcal L}_n(Y)$, the category 
$Y\times_{\overline{\mathcal L}_n\times\overline{\mathcal L}_n}
\overline{\mathcal L}_n$ fibred over the category of $Y$-schemes
is isomorphic to the functor on $Y$-schemes 
$\Isom(\mathscr C,\mathscr C')(f\colon Z\to Y)
=\Mor_{\overline{\mathcal L}_n(Z)}(f^*\mathscr C,f^*\mathscr C')$. 
Using the embedding via $\O_C(S)$ into $\P_Y(\pi_*\O_C(S))$ 
described in proposition \ref{prop:embY} we see that 
$\Isom(\mathscr C,\mathscr C')$ is a finite closed subgroup scheme 
of the open dense torus of $\P_Y(\pi_*\O_C(S))$. In characteristic $0$ 
it is unramified over $Y$, because then the fibres are reduced.
\end{proof}

\medskip

The stack $\overline{\mathcal L}_n$ is related to the Losev-Manin
moduli space $\overline{L}_n$ by a morphism 
$\overline{L}_n\to\overline{\mathcal L}_n$ that arises by considering 
the $n$ sections of an $n$-pointed chain over $Y$ as a relative 
effective Cartier divisor of degree $n$ over $Y$.

\begin{prop}\label{prop:L_n-L_n}
The morphism $\overline{L}_n\to\overline{\mathcal L}_n$ is 
faithfully flat and finite of degree $n!$.
It is ramified exactly in the points of $\overline{L}_n$
corresponding to $n$-pointed chains with some coinciding
marked points.
\end{prop}
\begin{proof}
Note that the morphism is representable, since $\overline{\mathcal L}_n$
has representable diagonal. We show that for any morphism 
$Y\to\overline{\mathcal L}_n$, $Y$ a scheme, the morphism of schemes 
$Y\times_{\overline{\mathcal L}_n}\overline{L}_n\to Y$ has the properties 
in question. The morphism $Y\to\overline{\mathcal L}_n$ corresponds to an 
object $\mathscr C=(C\!\to\!Y,s_-,s_+,S)$ over $Y$ and the functor
$Y\times_{\overline{\mathcal L}_n}\overline{L}_n$ maps a scheme
$T$ to the set $\{(f\colon T\!\to\!Y,(C'\!\to\!T,s_-,s_+,s_1,\ldots,s_n),
\alpha)\:|\:\alpha\colon f^*\mathscr C\to
(C'\!\to \!T,s_-,s_+,s_1+\ldots+s_n)\}$, 
where $s_1+\ldots+s_n$ denotes the divisor of degree $n$ associated 
to the $n$ sections and $\alpha$ is a morphism in 
$\overline{\mathcal L}_n(T)$.
We denote by $C_0$ the open subscheme of $C$ obtained by excluding 
the poles and intersection points of components on the fibres. 
Then $C_0$ is a quasi-projective curve over $Y$, which is smooth over 
$Y$ since it is flat with smooth fibres (see \cite[IV, (17.5.1)]{EGA}).
We may, for any $T$, identify the chains $C'$ over $T$ occurring 
in the above sets with $C\times_YT$ via the specified isomorphisms.
The additional data given by the subscheme $S\subset C_0$ are 
equivalent to a section $s\colon Y\to\Div_{C_0/Y}^n=C_0^{(n)}$ 
of the scheme of relative effective divisors of degree $n$, 
which coincides with the $n$-fold symmetric product of $C_0$ over $Y$, 
see \cite[Expos\'e XVII, 6.3.9, p.\ 186]{SGA4(3)}. Likewise,
the data given by the sections $s_1,\ldots,s_n$ are equivalent to 
a section  $s'\colon T\to(C_0^n)_T$ of the $n$-fold product
such that its composition with $(C_0^n)_T\to(C_0^{(n)})_T$ is 
the base extension $s_T$ of $s$, or equivalently, to a morphism 
$s'\colon T\to C_0^n$ whose composition with $C_0^n\to C_0^{(n)}$ 
coincides with $s\circ f$.
Thus the functor $Y\times_{\overline{\mathcal L}_n}\overline{L}_n$
is isomorphic to the functor of the scheme $Y\times_{C_0^{(n)}}C_0^n$,
and this concludes the proof because the morphism $C_0^n\to C_0^{(n)}$ 
has the required properties.
\end{proof}

\begin{rem} 
With proposition \ref{prop:L_n-L_n} and some general theory we can 
derive some properties of the stack $\overline{\mathcal L}_n$: 
by \cite[Thm.\ 10.1]{LMB}, making use of the proposition, 
$\overline{\mathcal L}_n$ is an algebraic stack (Artin stack); 
in characteristic $0$, by \cite[Thm.\ 8.1]{LMB} and the fact 
that it has unramified diagonal, it is a Deligne-Mumford stack.
However, the result will follow independently later in section 
\ref{sec:L_n-toricstack} together with a more detailed description 
of the structure of $\overline{\mathcal L}_n$.
\end{rem}

On the Losev-Manin moduli space $\overline{L}_n$ we have an 
operation of the symmetric group $S_n$ permuting the $n$ sections.
Any $S_n$-equivariant morphism $\overline{L}_n\to Z$, 
$Z$ a scheme with trivial $S_n$-action, factors through
$\overline{L}_n\to\overline{\mathcal L}_n$. 
This implies that the quotient morphism $\overline{L}_n\to
\overline{L}_n/S_n$ factors as
\[
\overline{L}_n\;\to\;\overline{\mathcal L}_n\;\to\;\overline{L}_n/S_n,
\]
and moreover, as $\overline{L}_n\to\overline{\mathcal L}_n$ is an 
epimorphism, $\overline{\mathcal L}_n\to\overline{L}_n/S_n$ forms 
the coarse moduli space of $\overline{\mathcal L}_n$.

\begin{rem}
There is the quotient stack $[\overline{L}_n/S_n]$, which has the same 
geometric points as $\overline{\mathcal L}_n$.
However, the automorphism groups of objects of $[\overline{L}_n/S_n]$
differ from those of $\overline{\mathcal L}_n$ which are always
abelian.
\end{rem}

In the case of the Losev-Manin moduli spaces, the boundary divisors 
arise as images of closed embeddings $\overline{L}_m\times\overline{L}_n
\to\overline{L}_{m+n}$.
For the stacks $\overline{\mathcal L}_n$ we also have embeddings
$\overline{\mathcal L}_m\times\overline{\mathcal L}_n\to
\overline{\mathcal L}_{m+n}$, defined as in the Losev-Manin case 
by concatenation of chains, and the diagrams
\begin{equation}\label{eq:concat}
\begin{array}{ccc}
\\[-4mm]
\overline{L}_m\times\overline{L}_n
&\longrightarrow&\overline{L}_{m+n}\\
\downarrow&&\downarrow\\
\overline{\mathcal L}_m\times\overline{\mathcal L}_n
&\longrightarrow&\overline{\mathcal L}_{m+n}\\
\end{array}
\end{equation}
are commutative. 

The morphism $\overline{L}_n\to\overline{\mathcal L}_n$ maps
the open dense torus of $\overline{L}_n=X(A_{n-1})$, the moduli 
space of irreducible $n$-pointed chains, onto the moduli stack of 
irreducible degree-$n$-pointed chains. This open substack of 
$\overline{\mathcal L}_n$ parametrises subschemes $S$ of degree $n$ 
in $\P^1\setminus\{0,\infty\}$ modulo automorphisms of $\P^1$ fixing 
$0$ and $\infty$. 
An object over an algebraically closed field $K$ can described 
by a monic polynomial $\prod_{i=1}^n(y-s_i)$ of degree $n$ with 
$s_1,\ldots,s_n\in K^*$ determined up to scaling by a common factor 
$\lambda\in K^*$ and permutations. We can write this polynomial as 
\[
y^n\;-\;(s_1+\ldots+s_n)y^{n-1}\;+\quad\ldots\ldots\quad+\;
(-)^ns_1\cdots s_n
\]
where the coefficients are the symmetric polynomials in $s_1,\ldots,s_n$. 
Assuming\linebreak $s_1\cdots s_n=(-1)^n$, we have a polynomial of the form
\[
y^n\;+\;a_1y^{n-1}\;+\;a_2y^{n-2}\;+\quad\ldots\ldots\quad+\;
a_{n-1}y\;+\;1
\]
with coefficients $a_1,\ldots,a_{n-1}\in K$. The isomorphism class 
of the object determines these coefficients up to the equivalence 
$(a_1,\ldots,a_{n-1})\sim(\xi^{n-1}a_1,\ldots,\xi a_{n-1})$, 
$\xi$ an $n$-th root of unity. The moduli stack of such objects
is the quotient stack $[\A^{n-1}/\mu_n]$, where the group scheme $\mu_n$
of $n$-th roots of unity acts with weights $(n\!-\!1,\ldots,1)$.

\medskip

It contains an $(n\!-\!1)$-dimensional algebraic torus $T$ parametrising 
classes of polynomials with non-zero coefficients. 
A $K$-valued point of $T$ is given by an $(n\!-\!1)$-tuple 
\[\textstyle
b_1\!=\!\frac{a_2}{a_1^2},\;b_2\!=\!\frac{a_1a_3}{a_2^2},\;\ldots\:,\;
b_{k}\!=\!\frac{a_{k-1}a_{k+1}}{a_k^2},\;\ldots\:,\;
b_{n-2}\!=\!\frac{a_{n-3}a_{n-1}}{a_{n-2}^2},\;
b_{n-1}\!=\!\frac{a_{n-2}}{a_{n-1}^2}
\]
of elements $b_i\in K^*$. These expressions in the $a_i$ 
form a set of generators of $\mu_n$-invariants in the coordinate 
ring of the torus $(\G_m)^{n-1}\subset\A^{n-1}$.
Equivalently, we can express a $K$-valued point of $T$ as a 
collection $(a_1,\ldots,a_{n-1},b_1,\ldots,b_{n-1})\in(K^*)^{2n-2}$ 
up to the equivalence 
\[\textstyle
(a_1,\ldots,a_{n-1},b_1,\ldots,b_{n-1})\sim
(\kappa_1a_1,\ldots,\kappa_{n-1}a_{n-1},\lambda_1b_1,\ldots,
\lambda_{n-1}b_{n-1})
\]
for $\kappa_i\in K^*$ and $\lambda_i=\kappa_i^2/
(\kappa_{i-1}\kappa_{i+1})$, putting $\kappa_0=\kappa_n=1$.
Then $(1,\ldots,1,b_1,\ldots,b_{n-1})\linebreak\sim
(a_1,\ldots,a_{n-1},1,\ldots,1)$ 
if $a_i,b_i$ satisfy the above equations.

\medskip

Allowing certain subsets of the coordinates $a_i,b_i$ to become 
zero, we obtain a toric tame stack which compactifies the 
moduli stack $[\A^{n-1}/\mu_n]$ of irreducible chains. 
Its definition and properties are contained in section
\ref{sec:toricorbifoldsAn} and we will show in section 
\ref{sec:L_n-toricstack} that it coincides with the moduli stack 
$\overline{\mathcal L}_n$. In particular, degenerating some of 
the $b_i$ to zero can be interpreted as degenerating 
$\P^1$ to a reducible chain of $\P^1$.
These additional divisors arise as in diagram (\ref{eq:concat}).

\begin{ex}\label{ex:L_2}
We illustrate some results of this paper in the case $n=2$ 
(see also examples \ref{ex:stackyfanY(A_1)},
\ref{ex:L_2-Y(A_1)}, \ref{ex:L_2-L_2}).
There is a natural embedding of degree-$2$-pointed chains 
$(C,s_-,s_+,S)$ into $\P^2$ determined by the line bundle
$\O_C(S)$ (for arbitrary $n$ see section \ref{sec:L_n-toricstack}): 
in $\P^2=\P(H^0(C,\O_C(S)))$ we can choose homogeneous coordinates 
$y_0,y_1,y_2$ such that $C$ is given by an equation $y_0y_2=b_1y_1^2$, 
the subscheme $S\subset C$ by an additional equation $y_0+a_1y_1+y_2=0$, 
and the two sections $s_-,s_+$ are $(1\!:\!0\!:\!0),(0\!:\!0\!:\!1)$.
Over an algebraically closed field $K$, data 
$(a_1,b_1)\in K^2\setminus\{(0,0)\}$ up to the equivalence 
$(a_1,b_1)\sim(\kappa_1a_1,\lambda_1b_1)$ for 
$\kappa_1^2=\lambda_1\in K^*$ correspond to isomorphism classes 
of degree-$2$-pointed chains over $K$.

The moduli stack $\overline{\mathcal L}_2$ is isomorphic to the 
quotient stack $[(\A^2\setminus\{(0,0)\})/\G_m]$ for the 
operation with weights $(1,2)$, i.e.\ the weighted projective line 
$\P(1,2)$ (which coincides with the toric orbifold $\mathcal Y(A_{n-1})$ 
for $n=2$ defined in section \ref{sec:toricorbifoldsAn}).
The open substack parametrising irreducible
curves, the locus 
where $b_1\neq0$, is the quotient stack $[\A^1/\mu_2]$ with 
coordinate $a_1$ on $\A^1$. The open substack parametrising 
objects without isomorphisms, the locus where $a_1\!\neq\!0$,
is isomorphic to $\A^1$ with coordinate~$b_1$.

\noindent
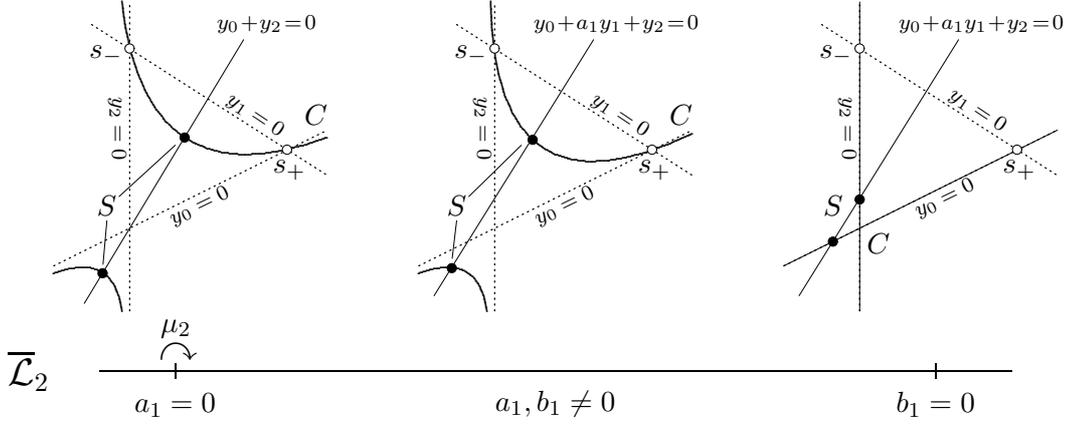
\begin{picture}(150,65)(0,0)
\put(8,6){\makebox(0,0)[l]{\large$\overline{\mathcal L}_2$}}
\put(5,0)
{
\put(15,6){\line(1,0){120}}
\put(25,0){\makebox(0,0)[b]{\small$a_1=0$}}
\put(25,5){\line(0,1){2}}
\put(25.4,7.5){\makebox(0,0)[b]{\large$\curvearrowright$}}
\put(25.1,10.5){\makebox(0,0)[b]{\small$\mu_2$}}
\put(14,20){
\dottedline{0.8}(5,-6)(5,35)
\dottedline{0.8}(-5,0)(31,18)
\dottedline{0.8}(0,32)(31,12)
\put(3.4,14){\rotatebox{-90}{\makebox(0,0)[r]{\tiny$y_2=0$}}}
\put(18,9.8){\rotatebox{26}{\makebox(0,0)[r]{\tiny$y_0=0$}}}
\put(25,17.6){\rotatebox{-34}{\makebox(0,0)[r]{\tiny$y_1=0$}}}
\put(23,32){\makebox(0,0)[c]{\tiny$y_0\!+\!y_2\!=\!0$}}
\qbezier(4,35)(7,8)(31,17)
\qbezier(-5,-1)(3,2)(4,-6)
\put(31,20){\makebox(0,0)[r]{\small$C$}}
\drawline(-1,-5)(20,30)
\filltype{white}
\put(5,28.8){\circle*{1.2}}
\put(4,28){\makebox(0,0)[r]{\small$s_-$}}
\put(25.7,15.4){\circle*{1.2}}
\put(26,14){\makebox(0,0)[t]{\small$s_+$}}
\filltype{black}
\put(1.44,-1){\circle*{1.2}}
\put(12.25,17){\circle*{1.2}}
\put(2,8){\makebox(0,0)[c]{\small$S$}}
\drawline(11,16)(3.6,9.2)
\drawline(1.5,0.2)(2,6)
}
\put(75,0){\makebox(0,0)[b]{\small$a_1,b_1\neq 0$}}
\put(62,20){
\dottedline{0.8}(5,-6)(5,35)
\dottedline{0.8}(-5,0)(31,18)
\dottedline{0.8}(0,32)(31,12)
\put(3.4,14){\rotatebox{-90}{\makebox(0,0)[r]{\tiny$y_2=0$}}}
\put(18,9.8){\rotatebox{26}{\makebox(0,0)[r]{\tiny$y_0=0$}}}
\put(25,17.6){\rotatebox{-34}{\makebox(0,0)[r]{\tiny$y_1=0$}}}
\put(21,32){\makebox(0,0)[c]{\tiny$y_0\!+\!a_1y_1\!+\!y_2\!=\!0$}}
\qbezier(4.5,35)(5,5)(31,17.5)
\qbezier(-5,-1)(3,2)(4,-6)
\put(31,20){\makebox(0,0)[r]{\small$C$}}
\drawline(-3,-4)(18,30)
\filltype{white}
\put(5,28.8){\circle*{1.2}}
\put(4,28){\makebox(0,0)[r]{\small$s_-$}}
\put(25.7,15.4){\circle*{1.2}}
\put(26,14){\makebox(0,0)[t]{\small$s_+$}}
\filltype{black}
\put(0,8){\makebox(0,0)[c]{\small$S$}}
\put(10,16.7){\circle*{1.2}}
\put(-0.65,-0.3){\circle*{1.2}}
\drawline(1.6,9.2)(8.5,16)
\drawline(0,6)(-0.6,1)
}
\put(125,0){\makebox(0,0)[b]{\small$b_1=0$}}
\put(125,5){\line(0,1){2}}
\put(110,20){
\dottedline{0.8}(5,-6)(5,35)
\dottedline{0.8}(-5,0)(31,18)
\dottedline{0.8}(0,32)(31,12)
\put(3.4,14){\rotatebox{-90}{\makebox(0,0)[r]{\tiny$y_2=0$}}}
\put(20,10.8){\rotatebox{26}{\makebox(0,0)[r]{\tiny$y_0=0$}}}
\put(24,18.2){\rotatebox{-34}{\makebox(0,0)[r]{\tiny$y_1=0$}}}
\put(21,32){\makebox(0,0)[c]{\tiny$y_0\!+\!a_1y_1\!+\!y_2\!=\!0$}}
\drawline(5,-6)(5,35)
\drawline(-5,0)(31,18)
\put(6,3){\makebox(0,0)[l]{\small$C$}}
\drawline(-3,-4)(18,30)
\filltype{white}
\put(5,28.8){\circle*{1.2}}
\put(4,28){\makebox(0,0)[r]{\small$s_-$}}
\put(25.7,15.4){\circle*{1.2}}
\put(26,14){\makebox(0,0)[t]{\small$s_+$}}
\filltype{black}
\put(1.5,3.2){\circle*{1.2}}
\put(5,8.8){\circle*{1.2}}
\put(2.8,8){\makebox(0,0)[r]{\small$S$}}
}
}
\end{picture}
\bigskip

The morphism $\overline{L}_2\to\overline{\mathcal L}_2$ is 
faithfully flat and finite of degree $2$.
We introduce homogeneous coordinates $z_-,z_+$ of $\overline{L}_2=\P^1$ 
that measure the position of one of the marked points of a $2$-pointed
chain with respect to the other marked point at $(1\!:\!1)$ of its 
component isomorphic to $\P^1$, such that the two points 
$(0\!:\!1),(1\!:\!0)$ correspond to reducible chains (cf.\ \cite{BB11a}). 
Then the point $(1\!:\minus1)$ corresponds to a $2$-pointed curve 
$\P^1$ with marked points $(1\!:\!1)$, $(1\!:\minus1)$ giving 
rise to a degree-$2$-pointed curve with automorphism group $\mu_2$.
The point $(1\!:\!1)$ corresponds to the point of $\overline{\mathcal L}_2$
with nonreduced $S$, the morphism is ramified here and \'etale elsewhere.

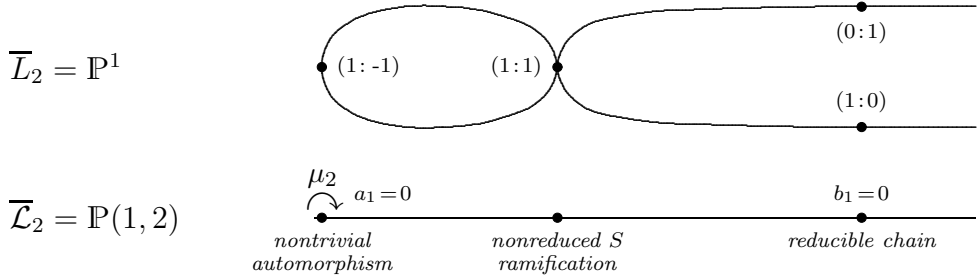
\medskip
\bigskip

\noindent
\begin{picture}(150,35)
\put(8,25){\makebox(0,0)[l]{$\overline{L}_2=\P^1$}}
\put(8,5){\makebox(0,0)[l]{$\overline{\mathcal L}_2=\P(1,2)$}}

\put(25,0){

\curve(
110,33, 90,33, 70,32.5, 65,32, 60,31, 58,30, 56,28,
55,25,
56,22, 58,20, 60,19, 65,18, 70,17.5, 90,17, 110,17)

\curve(
40,33, 45,32.5, 50,31, 52,30, 54,28, 
55,25,
54,22, 52,20, 50,19, 45,17.5, 40,17)

\curve(
40,17, 35,17, 30,18, 27,19.5, 25,21.5,
24,25,
25,28.5,27,30.5,30,32,35,33, 40,33
)

\filltype{black}
\put(55,25){\circle*{1.2}}
\put(53,25){\makebox(0,0)[r]{\tiny$(1\!:\!1)$}}

\put(24,25){\circle*{1.2}}
\put(26,25){\makebox(0,0)[l]{\tiny$(1\!:\minus1)$}}
\put(95,17){\circle*{1.2}}
\put(95,19){\makebox(0,0)[b]{\tiny$(1\!:\!0)$}}
\put(95,33){\circle*{1.2}}
\put(95,31){\makebox(0,0)[t]{\tiny$(0\!:\!1)$}}

\put(23,5){\line(1,0){87}}
\put(24,5){\circle*{1.2}}
\put(32,7){\makebox(0,0)[b]{\tiny$a_1\!=\!0$}}
\put(24,3){\makebox(0,0)[t]{\tiny\it nontrivial}}
\put(24,0){\makebox(0,0)[t]{\tiny\it automorphism}}
\put(24.4,6){\makebox(0,0)[b]{\large$\curvearrowright$}}
\put(24.1,9){\makebox(0,0)[b]{\small$\mu_2$}}

\put(55,5){\circle*{1.2}}
\put(55,3){\makebox(0,0)[t]{\tiny\it nonreduced $S$}}
\put(55,0){\makebox(0,0)[t]{\tiny\it ramification}}

\put(95,5){\circle*{1.2}}
\put(95,7){\makebox(0,0)[b]{\tiny$b_1\!=\!0$}}
\put(95,3){\makebox(0,0)[t]{\tiny\it reducible chain}}
}
\end{picture}

\bigskip
\medskip

\noindent
We will see in section \ref{sec:LM} that the morphism 
$\overline{L}_2\to\overline{\mathcal L}_2$ is given as 
\[\textstyle
(z_-\!+z_+:z_-z_+):\;\P^1\;\to\;\P(1,2).
\]
\end{ex}

\bigskip
\section{The toric orbifolds $\mathcal Y(A_n)$}
\label{sec:toricorbifoldsAn}
\medskip

In this section we will consider a family of toric orbifolds associated 
to the Cartan matrices of root systems of type $A$, but also comment on
some generalities on toric stacks.

\medskip

We use the definitions and notations of \cite{BCS05}. A stacky fan 
$\mathbf\Sigma=(N,\Sigma,\beta)$ defining a toric orbifold has the 
property that the abelian group $N$ is free; it consists of the 
data of a simplicial fan $\Sigma$ in the lattice $N$ and elements 
$n_\rho\in\rho\,\cap\,N$ for the one-dimensional cones $\rho\in\Sigma(1)$. 
Here we assume them to span the ambient space $N_\Q$. 
The homomorphism $\beta\colon\Z^{\Sigma(1)}\to N$ maps the elements 
of the standard basis to the elements $n_\rho$.
Dually we have the exact sequence 
\[
0\longrightarrow M\!=\!\Hom_\Z(N,\Z)
\stackrel{\beta^*}{\longrightarrow}\Z^{\Sigma(1)}
\stackrel{\beta^\vee}{\longrightarrow}\DG(\beta)\longrightarrow 0
\] 
giving rise, as sequence of character groups, to the exact sequence 
of diagonalisable group schemes
\[
1\longrightarrow G\longrightarrow T_{\Sigma(1)}\longrightarrow
T_M\longrightarrow 1.
\]
The toric orbifold $\mathcal X_{\mathbf\Sigma}$ is defined as the 
quotient stack $[U/G]$ with $U\subseteq\A^{\Sigma(1)}$ the open 
subset defined by the information which of the one-dimensional cones 
form higher dimensional cones of $\Sigma$. The constructions make sense
over the integers, however, working with $G$-torsors, in general one 
may have to choose an appropriate Grothendieck topology on the base 
category possibly finer than the \'etale topology 
(see also remark \ref{rem:topology-basecat}).
The resulting algebraic stacks $\mathcal X_{\mathbf\Sigma}$ are 
tame stacks in the sense of \cite{AOV08}. 

\begin{defi}
We define the toric orbifold $\mathcal Y(A_n)$ associated to the 
Cartan matrix of the root system $A_n$ in terms of the stacky fan 
$\mathbf\Upsilon(A_n)=(N,\Upsilon(A_n),\beta)$: let $N=\Z^n$ and let 
the linear map $\beta\colon\Z^{2n}\to N$ be given by the 
$n\times 2n$ matrix
\[
\left(\;
\begin{matrix}
\minus2&1&0&\cdots\quad&\;1\;&\;0\;&\;0\;&\cdots\\
1&\minus2&1&\cdots\quad&0&1&0&\cdots\\
0&1&\minus2&\cdots\quad&0&0&1&\cdots\\
\vdots&\vdots&\vdots&\ddots\quad&\vdots&\vdots&\vdots&\ddots\\
\end{matrix}
\quad\right)
\]
i.e.\ the matrix consisting of two blocks $(\minus C(A_n)\;I_n)$, 
where $C(A_n)$ is the Cartan matrix of the root system $A_n$ 
and $I_n$ the $n\times n$ identity matrix.
The fan $\Upsilon(A_n)$ has the $2n$ one-dimensional cones
$\rho_1,\ldots,\rho_n,\tau_1,\ldots,\tau_n$ generated by the columns 
of the above matrix.
A subset of one-dimensional cones generates a higher dimensional 
cone  of $\Upsilon(A_n)$ if it does not contain one of the sets
$\{\rho_1,\tau_1\}$, $\ldots\:,$ $\{\rho_n,\tau_n\}$.
This defines a fan containing $2^n$ $n$-dimensional 
cones $\sigma_I$ generated by sets 
$\{\rho_i:i\not\in I\}\cup\{\tau_i:i\in I\}$
for subsets $I\subseteq\{1,\ldots,n\}$.
\end{defi}

For the stacky fan $\mathbf\Upsilon(A_n)$ the map 
$\beta\colon\Z^{2n}\to N$ gives rise to the exact sequence of lattices
\[
0\longrightarrow M\!\cong\!\Z^n
\stackrel{\left(\begin{smallmatrix}\minus\,C\\I_n\end{smallmatrix}\right)}
{\longrightarrow}\Z^{2n}
\stackrel{\left(\begin{smallmatrix}I_n\;C\\\end{smallmatrix}\right)}
{\longrightarrow}\DG(\beta)\!\cong\!\Z^{n}\longrightarrow 0
\]
where $C=C(A_n)^\top=C(A_n)$ is (the transpose of) the Cartan matrix, 
and the exact sequence of tori
\[
1\longrightarrow G\!\cong\!(\G_m)^n\longrightarrow(\G_m)^{2n}
\longrightarrow T_N\!\cong\!(\G_m)^n\longrightarrow 1
\]
where $G\!\cong\!(\G_m)^n\longrightarrow(\G_m)^{2n}$,
$(\kappa_1,\ldots,\kappa_n)\mapsto(\kappa_1,\ldots,\kappa_n,
\lambda_1,\ldots,\lambda_n)$ with
$\lambda_i=\kappa_i^2/(\kappa_{i-1}\kappa_{i+1})$
setting $\kappa_0=\kappa_{n+1}=1$ (cf.\ last section).
Note that the toric orbifold $\mathcal Y(A_n)$ arises as 
quotient $[U/G]$ by a torus $G$.

\pagebreak
\begin{ex}\label{ex:stackyfanY(A_1)}
The toric orbifold $\mathcal Y(A_1)$ is isomorphic to the weighted 
projective line $\P(1,2)$: we have the matrix $(\,\minus2\;\:1\,)$ 
and the stacky fan looks as follows:

\noindent
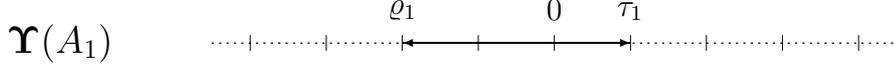
\begin{picture}(150,14)(0,0)
\put(8,6){\makebox(0,0)[l]{\large$\mathbf\Upsilon(A_1)$}}

\put(5,0){
\dottedline{1}(30,6)(120,6)
\put(75,9){\makebox(0,0)[b]{$0$}}
\put(75,6){\vector(1,0){10}}\put(85,9){\makebox(0,0)[b]{$\tau_1$}}
\put(75,6){\vector(-1,0){20}}\put(55,9){\makebox(0,0)[b]{$\rho_1$}}

\drawline(75,5)(75,7)
\drawline(65,5)(65,7)\drawline(55,5)(55,7)\drawline(45,5)(45,7)
\drawline(35,5)(35,7)
\drawline(85,5)(85,7)\drawline(95,5)(95,7)\drawline(105,5)(105,7)
\drawline(115,5)(115,7)
}
\end{picture}
\end{ex}

\begin{ex}\label{ex:stackyfanY(A_2)}
The toric orbifold $\mathcal Y(A_2)$ arises from the matrix
$
\left(\;
\begin{matrix}
\minus2&1&1&0\\
1&\minus2&0&1\:\\
\end{matrix}\;
\right)
$. 
We have the stacky fan 

\noindent
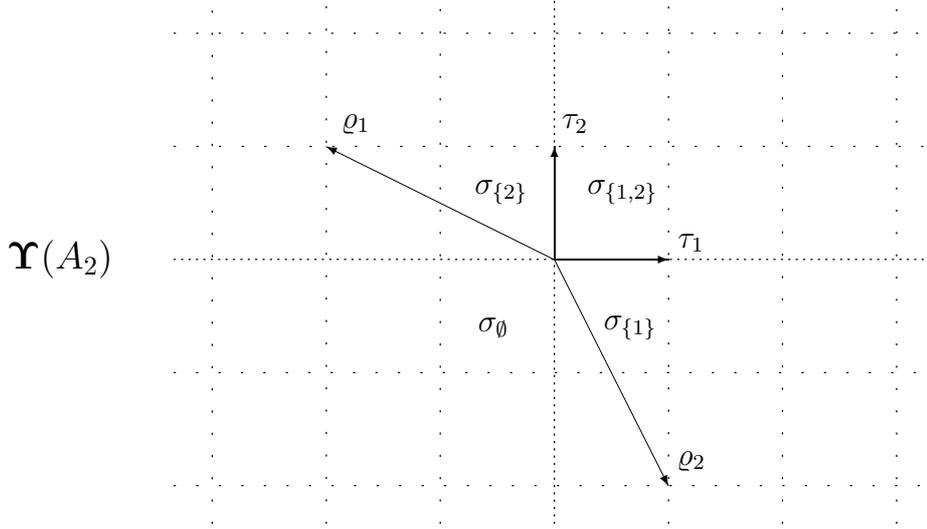
\begin{picture}(150,80)(0,0)
\put(8,40){\makebox(0,0)[l]{\large$\mathbf\Upsilon(A_2)$}}

\put(5,0){
\dottedline{1}(25,40)(125,40)
\dottedline{3}(25,55)(125,55)
\dottedline{3}(25,70)(125,70)
\dottedline{3}(25,25)(125,25)
\dottedline{3}(25,10)(125,10)

\dottedline{1}(75,5)(75,75)
\dottedline{3}(90,5)(90,75)
\dottedline{3}(105,5)(105,75)
\dottedline{3}(120,5)(120,75)
\dottedline{3}(60,5)(60,75)
\dottedline{3}(45,5)(45,75)
\dottedline{3}(30,5)(30,75)

\put(75,40){\vector(1,0){15}}\put(93,41){\makebox(0,0)[b]{$\tau_1$}}
\put(75,40){\vector(0,1){15}}\put(76,58){\makebox(0,0)[l]{$\tau_2$}}
\put(75,40){\vector(-2,1){30}}\put(47,58){\makebox(0,0)[l]{$\rho_1$}}
\put(75,40){\vector(1,-2){15}}\put(93,12){\makebox(0,0)[b]{$\rho_2$}}

\put(84,49){\makebox(0,0)[c]{$\sigma_{\{1,2\}}$}}
\put(68,49){\makebox(0,0)[c]{$\sigma_{\{2\}}$}}
\put(85,31){\makebox(0,0)[c]{$\sigma_{\{1\}}$}}
\put(67,31){\makebox(0,0)[c]{$\sigma_{\emptyset}$}}
}
\end{picture}
\end{ex}

\vspace{-5mm}

The description of the functor of a smooth toric variety given by 
Cox \cite{Co95b} in terms of collections of line bundles with sections
determined by the combinatorial data has been extended to toric
Deligne-Mumford stacks by Iwanari \cite{Iw07} and Perroni 
\cite{Pe08}. For the stacky fan $\mathbf\Upsilon(A_n)$ we have:

\begin{defi}
A $\mathbf\Upsilon(A_n)$-collection on a scheme $Y$ is a collection 
\[
\mathscr L=((\mathscr L_{\rho_i},a_i)_{i=1,\ldots,n},
(\mathscr L_{\tau_i},b_i)_{i=1,\ldots,n},
(c_i)_{i=1,\ldots,n})
\]
where $(\mathscr L_{\rho_i},a_i)$ and $(\mathscr L_{\tau_i},b_i)$ 
are line bundles with a section and 
\[
\begin{array}{l}
c_1\colon\mathscr L_{\tau_1}\!\otimes\!\mathscr L_{\rho_1}^{\otimes\,\minus2}
\!\otimes\!\mathscr L_{\rho_2}\to\O_Y,\;
c_2\colon\mathscr L_{\tau_2}\!\otimes\!\mathscr L_{\rho_1}\!\otimes\!
\mathscr L_{\rho_2}^{\otimes\,\minus2}\!\otimes\!
\mathscr L_{\rho_3}\to\O_Y,\;\ldots\ldots\\
c_{n-1}\colon\mathscr L_{\tau_{n-1}}
\!\!\otimes\!\mathscr L_{\rho_{n-2}}\!\!\otimes\!
\mathscr L_{\rho_{n-1}}^{\otimes\,\minus2}\!\otimes\!
\mathscr L_{\rho_n}\!\to\O_Y,\;
c_n\colon\mathscr L_{\tau_n}\!\!\otimes\!\mathscr L_{\rho_{n-1}}
\!\!\otimes\!\mathscr L_{\rho_n}^{\otimes\,\minus2}\to\O_Y\\
\end{array}
\]
are isomorphisms. These data are subject to the nondegeneracy 
condition that for every point $y\in Y$ and $i=1,\ldots,n$ not both 
$a_i(y)=0$ and $b_i(y)=0$.  

A morphism $\mathscr L'\to\mathscr L$ between two 
$\mathbf\Upsilon(A_n)$-collections $\mathscr L=
((\mathscr L_{\rho_i},a_i)_i,(\mathscr L_{\tau_i},b_i)_i,\linebreak
(c_i)_i)$ on $Y$ and $\mathscr L'=((\mathscr L_{\rho_i}',a_i')_i,
(\mathscr L_{\tau_i}',b_i')_i,(c_i')_i)$ on $Y'$ 
over a morphism of schemes $f\colon Y'\to Y$ is a collection
$((r_i)_{i=1,\ldots,n},(t_i)_{i=1,\ldots,n})$ consisting of 
isomorphisms of line bundles 
$r_i\colon f^*\mathscr L_{\rho_i}\to\mathscr L_{\rho_i}'$, 
$t_i\colon f^*\mathscr L_{\tau_i}\to\mathscr L_{\tau_i}'$ 
such that $r_i(f^*a_i)=a_i'$, $t_i(f^*b_i)=b_i'$ and the diagrams 

\pagebreak

\vspace*{-10mm}

\begin{equation}\label{eq:morphcoll}
\begin{array}{cccc}
f^*\!\mathscr L_{\tau_i}\!\otimes\!f^*\!\mathscr L_{\rho_{i-1}}
\!\otimes\!f^*\!\mathscr L_{\rho_i}^{\otimes\,\minus2}\!\otimes\!
f^*\!\mathscr L_{\rho_{i+1}}&
\stackrel{f^*c_i}{\longrightarrow}&f^*\O_Y\\[1mm]
\downarrow\quad\,&&\downarrow\;\\[-1mm]
\;\:\mathscr L_{\tau_i}'\otimes\mathscr L_{\rho_{i-1}}'
\otimes{\mathscr L_{\rho_i}'}^{\otimes\,\minus2}\otimes
\mathscr L_{\rho_{i+1}}'
&\stackrel{c_i'}{\longrightarrow}&\;\O_{Y'}\\
\end{array}
\end{equation}
($i=1,\ldots,n$; for $i=1,n$ omit the factors indexed by 
$\rho_0,\rho_{n+1}$) commute.

We denote the fibred category of $\mathbf\Upsilon(A_n)$-collections
over the category of schemes by $\mathcal C_{\mathbf\Upsilon(A_n)}$.
It comes with the cleavage given by pull-back of line bundles: for 
$f\colon Y'\to Y$ we have an arrow 
$f^*\mathscr L\to\mathscr L$ in $\mathcal C_{\mathbf\Upsilon(A_n)}$. 
The definition describes a morphism $\mathscr L'\to\mathscr L$ in 
$\mathcal C_{\mathbf\Upsilon(A_n)}$ as composition
of a morphism $\mathscr L'\to f^*\mathscr L$ over $\id_{Y'}$ 
with $f^*\mathscr L\to\mathscr L$ over $f\colon Y'\to Y$.
\end{defi}

\begin{rem}\label{rem:morphcoll}
A morphism of $\mathbf\Upsilon(A_n)$-collections $\mathscr L'\to\mathscr L$ 
over $\id_Y$ Zariski-locally for some open $Y'\subseteq Y$, a after fixing 
isomorphisms of the line bundles with the structure sheaf such that the 
isomorphisms $c_i$ become $\id_{\O_{Y'}}$, corresponds to a collection 
$\kappa_1,\ldots,\kappa_n,\lambda_1,\ldots,\lambda_n\in\O_{Y'}^*(Y')$
such that the isomorphisms 
$\O_{Y'}\!\cong\!\mathscr L_{\rho_i}|_{Y'}\to
\mathscr L_{\rho_i}'|_{Y'}\!\cong\!\O_{Y'}$ 
and 
$\O_{Y'}\!\cong\!\mathscr L_{\tau_i}|_{Y'}
\to\mathscr L_{\tau_i}'|_{Y'}\!\cong\!\O_{Y'}$ 
are given by multiplication by $\kappa_i$ and $\lambda_i$.
The condition expressed in diagram (\ref{eq:morphcoll}) translates 
into the equations $\lambda_i=\kappa_i^2/(\kappa_{i-1}\kappa_{i+1})$, 
putting $\kappa_0=\kappa_{n+1}=1$.
\end{rem}

The category of $\mathbf\Sigma$-collections $\mathcal C_{\mathbf\Sigma}$
for a stacky fan $\mathbf\Sigma$ is a category fibred in groupoids (CFG)  
over the base category of schemes. By descent theory for quasi-coherent 
sheaves the CFG $\mathcal C_{\mathbf\Sigma}$ forms a stack in the fpqc 
topology, see \cite[Thm.\ 4.23]{Vi05}.
By Iwanari \cite[Thm.\ 1.4]{Iw07} (for toric orbifolds) and 
Perroni \cite[Thm.\ 2.6]{Pe08} (for toric Deligne-Mumford stacks) 
over fields of characteristic $0$, working with the \'etale topology, 
there is an isomorphism of stacks 
$\mathcal X_{\mathbf\Sigma}\cong\mathcal C_{\mathbf\Sigma}$.
Also over more general base schemes we have an isomorphism
$\mathcal X_{\mathbf\Sigma}\cong\mathcal C_{\mathbf\Sigma}$; 
we make some comments on this issue.

\begin{con}\label{con:toricstack-coll}
Explicitely, one can construct an isomorphism 
$\mathcal X_{\mathbf\Sigma}\cong\mathcal C_{\mathbf\Sigma}$
as follows, here for simplicity we stick to the orbifold case and 
assume that the one-dimensional cones generate the ambient space $N_\Q$. 

Note that we have a natural 
$G$-equivariant $\mathbf\Sigma$-collection 
$((\O_U\otimes V_\rho,x_\rho)_\rho,(\id)_m)$ on $U$, where 
$V_\rho$ is the one-dimensional representation such that
the coordinate $x_\rho$ of $U\subset\A^{\Sigma(1)}$ is an invariant 
section of $\O_U\otimes V_\rho$ (in the case of smooth toric varieties 
as considered in \cite{Co95a} this $G$-equivariant collection descents 
to the universal collection on the toric variety).

Starting with an object of $\mathcal X_{\mathbf\Sigma}$ over $Y$,
that is a $G$-torsor $p\colon E\to Y$ together with a $G$-equivariant
morphism $t\colon E\to U$, the pull-back
$t^*((\O_U\otimes V_\rho,x_\rho)_\rho,(\id)_m)$ is a $G$-equivariant
$\mathbf\Sigma$-collection on $E$ and gives rise to the 
$\mathbf\Sigma$-collection
$p_*^Gt^*((\O_U\otimes V_\rho,x_\rho)_\rho,(\id)_m)$ on $Y$
(the functor $p_*^G$ takes the $G$-invariant part of the push-forward).

On the other hand, for a given $\mathbf\Sigma$-collection
$((\mathscr L_\rho,u_\rho)_\rho,(c_m)_m)$ on a scheme $Y$ 
we construct a $G$-torsor with a $G$-equivariant morphism to $U$.
Let $\underline{E}$ be the contravariant functor on the category 
of $Y$-schemes
\[
\underline{E}\colon(q\colon Y'\to Y)\:\mapsto\:
\left\{
\begin{array}{l}
\textit{$\mathbf\Sigma$-collections
$((\O_{Y'}\otimes V_\rho,u'_\rho),(\id)_m)$ on $Y'$}\\
\textit{with an isomorphism of $\mathbf\Sigma$-collections}\\
q^*((\mathscr L_\rho,u_\rho)_\rho,(c_m)_m)\cong
((\O_{Y'}\otimes V_\rho,u'_\rho)_\rho,(\id)_m)\\
\end{array}
\right\}
\]
where $V_\rho$, the one-dimensional representation as above,
is used to define an operation of $G$ on this functor.
Then one can show that the functor $\underline{E}$ with this 
$G$-action is represented by a $G$-torsor $p\colon E\to Y$ 
together with a universal isomorphism 
$p^*((\mathscr L_\rho,u_\rho)_\rho,(c_m)_m)\cong
((\O_E\otimes V_\rho,u^E_\rho)_\rho,(\id)_m)$ 
of $G$-equivariant $\mathbf\Sigma$-collections, provided that 
the original $\mathbf\Sigma$-collection is locally trivial in 
the sense that there is a covering $f\colon Y'\to Y$ such that 
$f^*((\mathscr L_\rho,u_\rho)_\rho,(c_m)_m)$ is isomorphic to 
a collection of the form $((\O_{Y'},u'_\rho)_\rho,(\id)_m)$; 
collections of this form correspond to trivial $G$-torsors.
We will assume that the topology on the base category is such that
any $\mathbf\Sigma$-collection has this property, see also 
the following remarks.
The sections $(u^E_\rho)_\rho$ of the universal 
$\mathbf\Sigma$-collection on $E$ then define a 
$G$-equivariant morphism $E\to U\subset\A^{\Sigma(1)}$.

Making use of the fact that for a $G$-torsor $p\colon E\to Y$ 
we have the equivalence $\QCoh(Y)\leftrightarrow\QCoh^G(E)$ 
given by the functors $p^*$ and $p_*^G$, one can show that 
these constructions define functors
$\mathcal X_{\mathbf\Sigma}\leftrightarrow\mathcal C_{\mathbf\Sigma}$
whose compositions are isomorphic to the identity functors.
\end{con}

\begin{rem}
Zariski-locally we can interpret the construction of the $G$-torsor 
as the coboundary homomorphism $d$ in the exact sequence 
(see \cite[Ch.\ III, \S 3]{Gi})
\[
0\longrightarrow H^0(Y,G)\longrightarrow 
H^0(Y,T_{\Sigma(1)})\longrightarrow 
H^0(Y,T_M)\stackrel{d}{\longrightarrow}
H^1(Y,G)
\]
where elements of $H^1(Y,G)$ are isomorphism classes of $G$-torsors
over $Y$:
given a $\mathbf\Sigma$-collection $((\O_Y,u_\rho)_\rho,(c_m)_m)$ on $Y$,
the automorphisms $(c_m)_m$ of the structure sheaf can be interpreted as 
a morphism $Y\to T_M$ or section of $T_M\times Y\to Y$, and fitting in
the cartesian diagram
\[
\begin{array}{ccc}
E&\longrightarrow&T_{\Sigma(1)}\\
\downarrow&&\downarrow\\
Y&\longrightarrow&T_M\\
\end{array}
\]
we obtain a $G$-torsor $E\to Y$ which is trivial if and only if 
$(c_m)_m\in H^0(Y,T_M)$ comes from an element of $H^0(Y,T_{\Sigma(1)})$.
\end{rem}

\begin{rem}\label{rem:topology-basecat}
Working with $G$-torsors, we usually assume that the Grothendieck
topology on the base category is fine enough in the sense that we 
have the same $G$-torsors as we have with respect to the canonical 
topology.
We have seen that this assumption was necessary to derive the isomorphism
$\mathcal X_{\mathbf\Sigma}\cong\mathcal C_{\mathbf\Sigma}$:
whereas the notion of $G$-torsor depends on the topology, this is not 
the case for the notion of $\mathbf\Sigma$-collections.
For $\mathbf\Sigma$-collections we have the corresponding assumption 
that $\mathbf\Sigma$-collections are locally trivial with respect 
to the topology (in the sense of construction \ref{con:toricstack-coll}).
In characteristic $0$ this is always true for the \'etale 
topology. In general we may have to take a finer topology, 
for example the fppf topology.
\end{rem}

In the case of the stacky fan $\mathbf\Upsilon(A_n)$ the lattice $M$ 
is a direct summand of $\Z^{\Sigma(1)}$ and the group scheme $G$ a torus, 
so the following result also holds in weaker topologies like
\'etale or Zariski.

\enlargethispage{4mm}

\begin{cor} 
There is an isomorphism of stacks
$\mathcal Y(A_n)\:\cong\:\mathcal C_{\mathbf\Upsilon(A_n)}$.
\end{cor}

In particular, a $K$-valued point of $\mathcal Y(A_n)$ corresponds 
to $(a_1,\ldots,a_n,b_1,\ldots,b_n)\in K^{2n}$ such that 
for any $i$ not both $a_i=0$ and $b_i=0$, up to the equivalence 
relation given by multiplication by a collection 
$(\kappa_1,\ldots,\kappa_n,\lambda_1,\ldots,\lambda_n)
\in(K^*)^{2n}$ 
as in remark \ref{rem:morphcoll}.

\pagebreak
\section{$\mathcal Y(A_{n-1})$ as moduli stack of degree-$n$-pointed chains}
\label{sec:L_n-toricstack}
\medskip

In this section we will prove the following theorem.

\begin{thm}\label{thm:modulistack-toricstack}
There is an isomorphism of stacks 
$\overline{\mathcal L}_n\cong\mathcal Y(A_{n-1})$.
\end{thm}
 
We will relate families of pointed chains to 
$\mathbf\Upsilon(A_{n-1})$-collections and prove an 
equivalence of fibred categories
$\overline{\mathcal L}_n\cong\mathcal C_{\mathbf\Upsilon(A_{n-1})}$.

\medskip

Let $(C,s_-,s_+,S)$ be a degree-$n$-pointed chain of $\P^1$
over a field $K$. We look at the closed embedding 
$C\to\P_K(H^0(C,\O_C(S)))\cong\P^n_K$ determined by $\O_C(S)$.

\medskip

First assume that $C$ is irreducible, that is $C\!\cong\!\P^1_K$.
The vector space $H^0(C,\O_C(S))$ is $(n+1)$-dimensional and we 
have a basis $y_0,\ldots,y_n$ such that the divisor of $y_i$ 
satisfies $\divis(y_i)=is_-+(n-i)s_+$.
The ideal sheaf $\mathscr I=\O_C(-S)\to\O_C$ defining $S$ is 
a line bundle. Tensored with $\O_C(S)$ we have an inclusion
$\O_C\to\O_C(S)$ with cokernel $\O_S$, and the image of the 
$1$-section of $\O_C$ is a global section 
$\sum_{i=0}^na_iy_i\in H^0(C,\O_C(S))$.
The subscheme $S\subset C$ is given by the equation
$\sum_{i=0}^na_iy_i=0$, where $a_0,a_n\neq0$ as $S$ does not 
meet $s_-,s_+$.
We can choose the basis $y_0,\ldots,y_n$ such that $a_0=a_n=1$.

The embedding defined by $\O_C(S)$, the $n$-fold Veronese embedding 
or $n$-uple embedding, gives an isomorphism of $C$ onto the subscheme 
in $\P^n_K$ determined by the equations
\[
y_iy_{j+1}=b_{i+1}\cdots b_jy_{i+1}y_j
\]
for $0\leq i<j<n$ and certain numbers $b_1,\ldots,b_{n-1}\in K^*$.
These equations express the condition that the rank of the matrix
\[
\left(
\begin{matrix}
y_0&b_1y_1&\ldots&b_1\cdots b_{n-1}y_{n-1}\\
y_1&y_2&\ldots&y_n\\
\end{matrix}
\right)
\]
is less than $2$.
The subscheme $S$ in the embedded curve is given by the additional 
linear equation 
\[y_n\;+\;a_{n-1}y_{n-1}\;+\;\ldots\;+\;a_1y_1\;+\;y_0\;=\;0.\]

Similarly, we have a natural embedding of reducible degree-$n$-pointed 
chains of $\P^1$ into $\P^n_K$.

\begin{prop}\label{prop:embK}
Let $(C,s_-,s_+,S)$ be a degree-$n$-pointed chain of $\P^1$ 
over a field $K$. It decomposes into irreducible components 
$C_1,\ldots,C_m\cong\P^1_K$ with poles $(p_1^-,p_1^+),\ldots,$
$(p_m^-,p_m^+)$ such that $s_-=p_1^-$, $s_+=p_m^+$ and $C_i$ 
intersects $C_{i+1}$ in $p_i^+=p_{i+1}^-$.
Let $n_1,\ldots,n_m$ be the degrees of $S$ on the components 
$C_1,\ldots,C_m$ and $N_k=\sum_{i=1}^kn_i$.
Then there is a basis $y_0,\ldots,y_n$ of $H^0(C,\O_C(S))$
characterised up to nonzero scalars by the following conditions:
$y_i$ is nonzero only on the components $C_k$ satisfying 
$N_{k-1}\leq i\leq N_k$ and in this case 
$\divis(y_i)|_{C_k}=(i-N_{k-1})p_k^-+(N_k-i)p_k^+$. 
We scale $y_0,y_n$ such that the image of the $1$-section 
under the inclusion $\O_C\to\O_C(S)$ is
$\sum_{i=0}^na_iy_i\in H^0(C,\O_C(S))$ with $a_0=a_n=1$,
that is, $S$ is given by an equation
\begin{equation}\label{eq:embS}
y_n\;+\;a_{n-1}y_{n-1}\;+\;\ldots\;+\;a_1y_1\;+\;y_0\;=\;0
\end{equation}
for some $a_1,\ldots,a_{n-1}\in K$.
These sections $y_0,\ldots,y_n$ satisfy the equations
\begin{equation}\label{eq:embC}
y_iy_{j+1}=b_{i+1}\cdots b_jy_{i+1}y_j
\end{equation}
for  $0\leq i<j<n$ and certain numbers $b_1,\ldots,b_{n-1}\in K$ 
such that $b_j=0$ exactly if $j\in\{N_1,\ldots,N_{m-1}\}$. 

The curve $C$ embeds into $\P_K(H^0(C,\O_C(S)))$, the image being 
the subscheme defined by the equations {\rm (\ref{eq:embC})}.
The subscheme $S$ of the embedded curve is given by the additional
equation {\rm (\ref{eq:embS})}. The sections $s_-,s_+$ are  
$(1\!:\!0\!:\ldots:\!0)$, $(0\!:\ldots:\!0\!:\!1)$.

The numbers $a_1,\ldots,a_{n-1}$ and $b_1,\ldots,b_{n-1}$ in
{\rm (\ref{eq:embS})} and {\rm (\ref{eq:embC})} have the property
that not both $a_j=0$ and $b_j=0$.
\end{prop}

We will not work out the proof in detail, but add some remarks.

\begin{rem}\label{rem:embK}
The component $C_k$ is embedded into the projective subspace of 
$\P_K(H^0(C,\O_C(S)))$ spanned by the coordinates
$y_{N_{k-1}},\ldots,y_{N_k}$ via a Veronese embedding, 
the image is given by the equations corresponding to the 
condition that the rank of the matrix
\[
\left(
\begin{matrix}
y_{N_{k-1}}&b_{N_{k-1}+1}y_{N_{k-1}+1}&\ldots&
b_{N_{k-1}+1}\cdots b_{N_k-1}y_{N_k-1}\\
y_{N_{k-1}+1}&y_{N_{k-1}+2}&\ldots&y_{N_k}\\
\end{matrix}
\right)
\]
is less than $2$.
The equation (\ref{eq:embS}) reduces on $C_k$ to 
$a_{N_k}y_{N_k}+\ldots+a_{N_{k-1}}y_{N_{k-1}}=0$
which defines a finite subscheme $S_k$ of degree $n_k$ in 
$C_k\subseteq\P_K^{n_k}$. A subscheme $S_k$ of this form does not 
meet the poles of $C_k$ provided that $a_{N_k},a_{N_{k-1}}\neq 0$.
\end{rem}

We generalise this to degree-$n$-pointed chains over schemes.

\begin{prop}\label{prop:embY}
Let $(\pi\colon C\!\to\!Y,s_-,s_+,S)$ be a degree-$n$-pointed chain 
of $\P^1$ over a scheme $Y$. 
For any $y\in Y$ there is an open affine neighbourhood $Y'\subseteq Y$
such that there is a decomposition 
$\pi_*\O_C(S)|_{Y'}\cong\bigoplus_{i=0}^n\O_{Y'}y_i$ 
characterised on the fibres by the properties of proposition 
\ref{prop:embK}. The generators $y_0,\ldots,y_n\in H^0(Y',\pi_*\O_C(S))$ 
of the individual summands, after possibly rescaling by a global 
section of $\O_{Y'}^*$, satisfy:\\
{\rm (i)} The image of the $1$-section under the inclusion 
$\O_C\to\O_C(S)$ is of the form 
\[y_n\;+\;a_{n-1}y_{n-1}\;+\;\ldots\;+\;a_1y_1\;+\;y_0.\]
over $Y'$ for some $a_1,\ldots,a_{n-1}\in\O_{Y'}(Y')$.\\
{\rm (ii)} The kernel of the homomorphism of algebras
$\Sym\pi_*(\O_C(S))\to\bigoplus_{k=0}^\infty\pi_*\O_C(kS)$
is over $Y'$ generated by the equations
\[
y_iy_{j+1}=b_{i+1}\cdots b_jy_{i+1}y_j
\]
for $0\leq i<j<n$ and some $b_1,\ldots,b_{n-1}\in\O_{Y'}(Y')$. 

The line bundle $\O_C(S)$ determines a closed embedding 
$C\to C'\subset\P_Y(\pi_*(\O_C(S)))$ over $Y$.
Over $Y'$ in the coordinates $y_0,\ldots,y_n$ the embedded curve 
$C'_{Y'}\subset\linebreak \P_{Y'}(\pi_*(\O_C(S))|_{Y'})\cong\P^n_{Y'}$
is defined by the equations in  {\rm (ii)}, the image of $S_{Y'}$ 
in $C'_{Y'}$ by the additional equation in {\rm (i)}, and the sections 
$s_-,s_+$ are  $(1\!:\!0\!:\ldots:\!0)$, $(0\!:\ldots:\!0\!:\!1)$.
\end{prop}

\pagebreak
\begin{proof}
The decomposition of the fibre $C_y$ over a point $y\in Y$ into 
irreducible components $C_y=C_1\cup\ldots\cup C_m$ determines over 
an open affine subscheme $Y'\subseteq Y$ containing $y$ a decomposition 
of $S$ into divisors $S_1,\ldots,S_m$ which are disjoint and such 
that $S_k$ only meets one component on each fibre and the component 
$C_k$ over $y$.
Each $S_k$ determines a morphism  onto a $\P^1$-bundle 
over $Y'$, which on the fibres is an isomorphism on the 
component containing $S_k$ and contracts the other components 
(similar to the contraction morphisms in \cite{Kn83}, cf.\ 
also \cite[3.3]{BB11a}). After possibly shrinking $Y'$,
we have global sections $y_0^{(k)},\ldots,y_n^{(k)}$ of 
$\O_{\P^1_{Y'}}(S_k)$ that satisfy 
$\divis(y_i^{(k)})=(i-N_{k-1})s_-+(N_k-i)s_+$
for $N_{k-1}\leq i\leq N_k$,
$\divis(y_i^{(k)})=(N_k-N_{k-1})s_+$
for $i<N_{k-1}$ and
$\divis(y_i^{(k)})=(N_k-N_{k-1})s_-$
for $i>N_k$
(using the notation $S_k,s_-,s_+$ also for their images in $\P^1_{Y'}$).
Using the pull-backs of these sections to $C_{Y'}$ denoted by the same symbols, 
let $y_i=y_i^{(1)}\cdots y_i^{(m)}\in H^0(C_{Y'},\O_C(S))=
H^0(C_{Y'},\O_{C_{Y'}}(S_1)\otimes\ldots\otimes\O_{C_{Y'}}(S_m))$.
Over the open affine neighbourhood $Y'$ of $y$ the sections $y_0,\ldots,y_n$ 
define a decomposition of $\pi_*\O_C(S)$ with the required properties.

The image of the $1$-section under the inclusion $\O_C\to\O_C(S)$
gives a global section of $\O_C(S)$ which over $Y'$ is of the form 
$\sum_{i=0}^na_iy_i$ with $a_i\in\O_{Y'}(Y')$. 
Since $a_0,a_n\in\O_{Y'}^*(Y')$ we can assume that $a_0,a_n=1$.

Using what is known about the fibres and results from
\cite[III]{EGA} (cf.\ also \cite{Kn83}), we derive that 
$\pi_*(\O_C(kS))$ for $k>0$ is locally free of rank $kn+1$, 
further that the homomorphism $\pi^*\pi_*\O_C(S)\to\O_C(S)$ is 
surjective and defines a closed embedding 
$C\to\P_Y(\pi_*\O_C(S))$.

The embedding $C\to\P^n_Y$ corresponds to the surjection of 
graded algebras\linebreak
$\Sym\pi_*(\O_C(S))\to\bigoplus_{k=0}^\infty\pi_*\O_C(kS)$.
Its kernel $\mathscr I$ is the graded ideal that defines the
embedded curve $C'\subset\P^n_Y$.
Each part $\mathscr I_k$ of $\mathscr I$ is locally free, being  
the kernel of a surjective homomorphism of locally free sheaves. 
The graded ideal $\mathscr I$ is generated in degree $2$ since 
this is the case on the fibres $\mathscr I\otimes\kappa(y)$ for 
each point $y\in Y$. The part $\mathscr I_2$ of 
degree $2$ is a vector bundle of rank  $\frac{1}{2}n(n-1)$. 

Working over $Y'$, for $i<j$ the subsheaves 
$\left<y_{i+1}y_j,y_iy_{j+1}\right>$ and 
$\left<y_{i+1}y_j\right>$ of\linebreak $\pi_*\O_C(S)|_{Y'}$, i.e.\ the 
subsheaves generated by the respective sections, coincide as this
is true on the fibres. Considering the case $j=i+1$, the kernel 
of the surjective homomorphism 
$\Sym^2\pi_*(\O_C(S))|_{Y'}\supset\O_{Y'}y_{i+1}^2\oplus\O_{Y'}y_iy_{i+2}\to
\left<y_{i+1}^2,y_iy_{i+2}\right>\subset\pi_*\O_C(2S)|_{Y'}$
is generated by an element $y_iy_{i+2}-b_{i+1}y_{i+1}^2$ 
for some\linebreak $b_{i+1}\in\O_{Y'}(Y')$.
For general $i<j$ we have as kernel $y_iy_{j+1}-b_{i+1,j}y_{i+1}y_j$
for some $b_{i+1,j}\in\O_{Y'}(Y')$ and from the equation
$(y_{i+2}\cdots y_j)b_{i+1,j}y_{i+1}y_j=(y_{i+2}\cdots y_j)y_iy_{j+1}
=(b_{i+1}\cdots b_j)(y_{i+2}\cdots y_j)y_{i+1}y_j$ in $\pi_*\O_C(S)|_{Y'}$ 
we conclude that $b_{i+1,j}=b_{i+1}\cdots b_j$. 
\end{proof}

We define morphisms of fibred categories 
$\Phi\colon\overline{\mathcal L}_n\to\mathcal C_{\mathbf\Upsilon(A_{n-1})}$ 
and
$\Psi\colon\mathcal C_{\mathbf\Upsilon(A_{n-1})}\to\overline{\mathcal L}_n$. 

\begin{con}\label{con:curve->data}
Let $\mathscr C=(C\to Y,s_-,s_+,S)$ be a degree-$n$-pointed chain 
of $\P^1$ over a scheme $Y$. For any point $y\in Y$ we have an 
open neighbourhood $U\subseteq Y$ over which we have a decomposition 
$\pi_*\O_C(S)|_U\cong\O_U^{\oplus n+1}$ and a basis $y_0,\ldots,y_n$ 
as in proposition \ref{prop:embY}, and we obtain functions 
$a_1,\ldots,a_{n-1},b_1,\ldots,b_{n-1}\in\O_U(U)$.
We define $(\mathscr L_{\rho_i},a_i):=(\O_U,a_i)$,
$(\mathscr L_{\tau_i},b_i):=(\O_U,b_i)$ and have the isomorphisms
$c_i\colon\mathscr L_{\tau_i}\!\otimes\mathscr L_{\rho_{i-1}}
\!\otimes\mathscr L_{\rho_i}^{\otimes\,\minus2}
\otimes\mathscr L_{\rho_{i+1}}\to\O_U$ 
(omit $\mathscr L_{\rho_0}$, $\mathscr L_{\rho_n}$)
given by the identities on $\O_U$. These data form a 
$\mathbf\Upsilon(A_{n-1})$-collection over $U$, 
the nondegeneracy condition that not both $a_i=0$ and $b_i=0$ 
in each point is satisfied by construction and proposition 
{\rm\ref{prop:embK}}.

Different choices of bases $y_0,\ldots,y_n$ and $y'_0,\ldots,y'_n$
over $U$ and $U'$ as above are related over $U''=U\cap U'$ by 
$y_i=\kappa_i'y'_i$ for some $\kappa_i'\in\O_{U''}^*(U'')$, 
where $\kappa_0'=\kappa_n'$. Let $\kappa_i=\kappa_i'/\kappa_0'$.
There is an isomorphism between the corresponding 
$\mathbf\Upsilon(A_{n-1})$-collections over $U''$ that, with 
respect to the given trivialisations, is given by the collection 
$\kappa_1,\ldots,\kappa_{n-1},\lambda_1,\ldots,\lambda_{n-1}
\in\O_{U''}^*(U'')$ as in remark \ref{rem:morphcoll}.

We cover $Y$ by open subschemes $U$ as above, obtain  
$\mathbf\Upsilon(A_{n-1})$-collections on this covering and glue 
them to a $\mathbf\Upsilon(A_{n-1})$-collection 
$\mathscr L=\Phi\mathscr C$ on $Y$.

For a morphism $\mathscr C'\to\mathscr C$ of degree-$n$-pointed
chains over $f\colon Y'\to Y$, i.e.\ a cartesian diagram
consisting of $f$, a morphism $F\colon C'\to C$ that maps 
$s_-',s_+',S'$ to $s_-,s_+,S$ and $\pi\colon C\to Y$, 
$\pi'\colon C'\to Y'$, we have a morphism of 
$\mathbf\Upsilon(A_{n-1})$-collections $\mathscr L'=\Phi\mathscr C'\to
\mathscr L=\Phi\mathscr C$ over $f\colon Y'\to Y$. 
Locally over $U''=f^{-1}(U)\cap U'$, where $U'$ and $U$ are elements 
of the chosen open coverings of $Y'$ and $Y$, we have chosen local bases 
$f^*y_0,\ldots,f^*y_n$ and $y'_0,\ldots,y'_n$ of 
$f^*\pi_*\O_C(S)\cong \pi'_*F^*\O_C(S)\cong\pi'_*\O_{C'}(S')$. 
Comparing these bases gives rise to isomorphisms
$f^*\mathscr L_{\rho_i}|_{U''}\to\mathscr L'_{\rho_i}|_{U''}$, 
$f^*\mathscr L_{\tau_i}|_{U''}\to\mathscr L'_{\tau_i}|_{U''}$
as above, and these can be glued to a morphism 
$\mathscr L'\to\mathscr L$.

One checks that this defines a functor 
$\Phi\colon\overline{\mathcal L}_n\to\mathcal C_{\mathbf\Upsilon(A_{n-1})}$. 
The functor $\Phi$ is base-preserving and sends cartesian
arrows to cartesian arrows.
\end{con}

\begin{con}\label{con:data->curve}
Let $\mathscr L=((\mathscr L_{\rho_i},a_i)_i,
(\mathscr L_{\tau_i},b_i)_i,(c_i)_i)$ be a
$\mathbf\Upsilon(A_{n-1})$-collection on a scheme $Y$. 
For any point $y\in Y$ we have an open neighbourhood 
$U\subseteq Y$ over which we can choose trivialisations 
$\mathscr L_{\rho_i}|_U$, $\mathscr L_{\tau_i}|_U\cong\O_U$ 
such that the isomorphisms 
$c_{i}\colon(\mathscr L_{\tau_i}\!\otimes\mathscr L_{\rho_{i-1}}
\!\otimes\mathscr L_{\rho_i}^{\otimes\,\minus2}\otimes
\mathscr L_{\rho_{i+1}})|_U\to\O_U$ 
for $i=1,\ldots,n-1$ (omit $\mathscr L_{\rho_0},\mathscr L_{\rho_n}$) 
are the identities on $\O_U$.
Let $C$ be the closed subscheme of $\P^n_U$ given by the
equations $y_iy_{j+1}=b_{i+1}\cdots b_jy_{i+1}y_j$
for $0\leq i<j\leq n-1$, where $y_0,\ldots,y_n$ are homogeneous 
coordinates of $\P^n_U$ and $b_1,\ldots,b_{n-1}$ are considered 
as regular functions on $U$ via the isomorphisms 
$\mathscr L_{\tau_i}|_U\cong\O_U$, and let $\pi\colon C\to U$ be
induced by $\P^n_U\to U$. 
By construction, the subscheme $C\subseteq\P^n_U$ is isomorphic 
to  $\Proj_U\mathscr S$ where $\mathscr S$ is the graded algebra
$\O_U[y_0,\ldots,y_n]\,/
\left<y_iy_{j+1}=b_{i+1}\cdots b_jy_{i+1}y_j;i<j\right>$. 
The morphism $\pi\colon C\to U$ is flat since each 
graded piece of $\mathscr S$ is locally free
(\cite[III, (7.9.14)]{EGA}). Indeed, we have
$\mathscr S_k\cong\O_U^{\oplus kn+1}$ with basis $y_n^k$ and
$y_i^{k-l}y_{i+1}^l$ for $i=0,\ldots,n\!-\!1$ and 
$l=0,\ldots,k\!-\!1$.
Let $s_-,s_+$ be the sections $(1\!:\!0:\ldots:\!0)$,
$(0\!:\ldots:\!0\!:\!1)$ with respect to the coordinates 
$y_0,\ldots,y_n$ and let $S\subset C$ be the subscheme 
given by the additional equation
$y_n+a_{n-1}y_{n-1}+\ldots+a_1y_1+y_0=0$, where again  
$a_1,\ldots,a_{n-1}$ are considered as regular functions via 
$\mathscr L_{\rho_i}|_U\cong\O_U$.
This defines a degree-$n$-pointed chain $(C\to U,s_-,s_+,S)$ over $U$.

Different choices of local trivialisations of the line bundles
$\mathscr L_{\rho_i}$ over $U, U'$ are related by multiplication 
by some $\kappa_i\in\O_{U''}^*(U'')$ over $U''=U\cap U'$. The 
corresponding degree-$n$-pointed chains in $\P^n_{U''}$ are connected 
by the automorphism of $\P^n_{U''}$ given by multiplying the 
homogeneous coordinates with $\kappa_0,\ldots,\kappa_n$ 
(set $\kappa_0=\kappa_n=1$).

We cover $Y$ by open subschemes $U$ as above, obtain  
degree-$n$-pointed chains over this covering and glue them 
to a degree-$n$-pointed chain $\Psi\mathscr L$ over $Y$.

For a morphism $\mathscr L'\to\mathscr L$ of 
$\mathbf\Upsilon(A_{n-1})$-collections over $f\colon Y'\to Y$,
i.e.\ a collection of isomorphisms of line bundles with sections
$(f^*\!\mathscr L_{\rho_i},f^*a_i)\to(\mathscr L_{\rho_i}',a_i')$,
$(f^*\!\mathscr L_{\tau_i},f^*b_i)\to(\mathscr L_{\tau_i}',b_i')$,
we have a morphism $\mathscr C'=\Psi\mathscr L'\to\mathscr C=\Psi\mathscr L$ 
of degree-$n$-pointed chains over $f\colon Y'\to Y$: locally on 
$U''=f^{-1}(U)\cap U'$, where $U'$ and $U$ are elements of the chosen 
open coverings of $Y'$ and $Y$, using the given trivialisations
$\mathscr L'_{\rho_i}|_{U'}\cong\O_{U'}$ and
$f^*\!\mathscr L_{\rho_i}|_U\cong f^*\O_U=\O_{f^{-1}(U)}$,
the isomorphisms $f^*\!\mathscr L_{\rho_i}\to\mathscr L_{\rho_i}'$
are given by multiplication with elements $\kappa_i\in\O_{U''}^*(U'')$, 
and the automorphism of $\P^n_{U''}$ such that the coordinates 
$y_i'$ and $f^*y_i$ are related by multiplication by $\kappa_i$ 
(set $\kappa_0=\kappa_n=1$) induces an isomorphism of the embedded 
degree-$n$-pointed chains over $U''$.
By glueing we obtain an isomorphism $\mathscr C'\to f^*\mathscr C$.

This defines a functor $\Psi\colon\mathcal C_{\mathbf\Upsilon(A_{n-1})}
\to\overline{\mathcal L}_n$ which is base-preserving and sends cartesian 
arrows to cartesian arrows.
\end{con}

These two functors give the equivalence of fibred categories stated 
in the theorem.

\begin{proof}[Proof of theorem \ref{thm:modulistack-toricstack}]
We show that the fibred categories 
$\overline{\mathcal L}_n$ and $\mathcal C_{\mathbf\Upsilon(A_{n-1})}$
are equivalent using the functors
$\Psi\colon\mathcal C_{\mathbf\Upsilon(A_{n-1})}\to\overline{\mathcal L}_n$
and
$\Phi\colon\overline{\mathcal L}_n\to\mathcal C_{\mathbf\Upsilon(A_{n-1})}$.

For an object $\mathscr C=(C\to Y,s_-,s_+,S)$ in $\overline{\mathcal L}_n$,
after choice of a suitable open covering, the data $\Phi\mathscr C$ is 
given locally by
$(\mathscr L_{\rho_i}|_U=\O_U,a_i)$, $(\mathscr L_{\tau_i}|_U=\O_U,b_i)$ 
and locally $\mathscr C$ is isomorphic to the embedded object in $\P^n_U$
defined by the functions $a_i,b_i$ with respect to coordinates 
$y_0,\ldots,y_n$ of $\P^n_U$, see construction \ref{con:curve->data}. 
Applying the functor $\Psi$, we choose a covering and 
isomorphisms $\mathscr L_{\rho_i}|_{U'},\mathscr L_{\tau_i}|_{U'}\to\O_{U'}$ 
as in construction \ref{con:data->curve} giving rise to functions 
$\tilde a_i,\tilde b_i$, and these  define the object 
$\Psi\Phi\mathscr C$ locally embedded in $\P^n_{U'}$ with homogeneous 
coordinates $\tilde y_i$. 
Comparing the two isomorphisms $\mathscr L_{\rho_i}|_{U''}\to\O_{U''}$
over open subschemes $U''=U\cap U'$, we obtain an isomorphism 
$\Psi\Phi\mathscr C\to\mathscr C$ locally over $U''$ using the two local
embeddings of $\Psi\Phi\mathscr C$ and $\mathscr C$ in $\P^n_{U''}$.
One checks that these isomorphisms form an isomorphism of functors
$\Psi\circ\Phi\cong\Id$.

Starting with a $\mathbf\Upsilon(A_{n-1})$-collection $\mathscr L$ on $Y$, 
after choice of a covering and isomorphisms 
$\mathscr L|_U\cong((\O_U,a_i)_i,(\O_U,b_i)_i,(\id)_i))$,
we construct an object $\Psi\mathscr L=(\pi\colon C\to Y,s_-,s_+,S)$, 
locally embedded in $\P^n_U$ with homogeneous coordinates $y_0,\ldots,y_n$ 
using the functions $a_i,b_i\in\O_U(U)$.
From $\Psi\mathscr L$ we extract data $\Phi\Psi\mathscr L$ after choice 
of a suitable covering and local bases 
$\tilde y_0,\ldots,\tilde y_n$ of $\pi_*(\O_C(S))$, locally given by 
$(\tilde{\mathscr L}_{\rho_i}|_{U'}=\O_{U'},\tilde a_i)$, 
$(\tilde{\mathscr L}_{\tau_i}|_{U'}=\O_{U'},\tilde b_i)$
for elements $U'$ of the covering.
Comparing the two collections of homogeneous coordinates,
satisfying the conditions of proposition \ref{prop:embY}, of 
$\P^n_U\times_UU''\stackrel{\sim\;}{\to}
\P_{U''}(\bar{\pi}_*\O_{\P^n_U}(\bar S)|_{U''})
\cong\P_{U''}(\pi_*\O_{C}(S)|_{U''})$ where
$\bar S\subset\P^n_U$ is the hyperplane determined by the 
equation $\sum_ia_iy_i=0$ (set $a_0,a_{n+1}=1$) and
$\bar{\pi}\colon\P^n_U\to U$, and thus the two ways 
the object $\mathscr C$ is locally embedded in $\P^n_{U''}$ 
for open sets $U''=U\cap U'$, we obtain an 
isomorphism of $\mathbf\Upsilon(A_{n-1})$-collections 
$\Phi\Psi\mathscr L\to\mathscr L$ locally given by some 
$\kappa_1,\ldots,\kappa_{n-1},\lambda_1,\ldots,\lambda_{n-1}\in\O^*_{U''}(U'')$
as in remark \ref{rem:morphcoll}.
One verifies that this gives an isomorphism of functors 
$\Phi\circ\Psi\cong\Id$.
\end{proof}

\bigskip
\pagebreak

\begin{cor}
The coarse moduli space of $\overline{\mathcal L}_n$, which 
coincides with the quotient $\overline{L}_n/S_n$, is isomorphic 
to the toric variety $Y(A_{n-1})$ corresponding to the simplicial 
fan $\Upsilon(A_{n-1})$ underlying the stacky fan 
$\mathbf\Upsilon(A_{n-1})$.
\end{cor}

\begin{ex}\label{ex:L_2-Y(A_1)}
In the case $n=2$ we have the isomorphism 
$\overline{\mathcal L}_2\cong\mathcal Y(A_1)$. 
The stacky fan of $\mathcal Y(A_1)$ was pictured in 
example \ref{ex:stackyfanY(A_1)}.
We have the following types of pointed chains over 
$\overline{\mathcal L}_2\cong\mathcal Y(A_1)$ 
(cf.\ also example \ref{ex:L_2}):

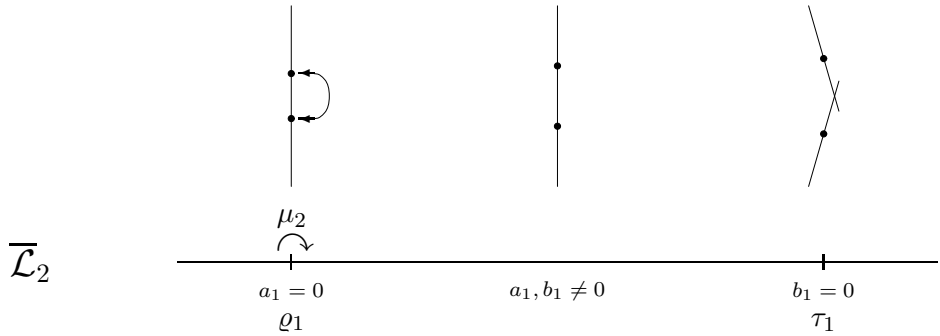
\bigskip  
\noindent
\begin{picture}(150,44)(0,0)
\put(8,10){\makebox(0,0)[l]{\large$\overline{\mathcal L}_2$}}
\put(5,0)
{
\put(25,10){\line(1,0){100}}
\put(40,5){\makebox(0,0)[b]{\tiny$a_1=0$}}
\put(40,2){\makebox(0,0)[c]{\small$\rho_1$}}
\put(40,9){\line(0,1){2}}
\put(40.5,11.5){\makebox(0,0)[b]{\large$\curvearrowright$}}
\put(40.1,14.5){\makebox(0,0)[b]{\small$\mu_2$}}
\put(75,5){\makebox(0,0)[b]{\tiny$a_1,b_1\neq 0$}}
\put(110,5){\makebox(0,0)[b]{\tiny$b_1=0$}}
\put(110,2){\makebox(0,0)[c]{\small$\tau_1$}}
\put(110,9){\line(0,1){2}}
\put(40,32){
\drawline(0,-12)(0,12)
\put(0,3){\circle*{0.8}}
\put(0,-3){\circle*{0.8}}
\put(3,3){\vector(-1,0){2}}
\put(3,-3){\vector(-1,0){2}}
\spline(3,3)(3.8,2.9)(5,1.5)(5,-1.5)(3.8,-2.9)(3,-3)
}
\put(75,32){
\drawline(0,-12)(0,12)
\put(0,4){\circle*{0.8}}
\put(0,-4){\circle*{0.8}}
}
\put(110,32){
\drawline(2,-2)(-2,12)
\drawline(2,2)(-2,-12)
\put(0,5){\circle*{0.8}}
\put(0,-5){\circle*{0.8}}
}
}
\end{picture}
\end{ex}

\begin{ex}
In the case $n=3$ we have the isomorphism 
$\overline{\mathcal L}_3\cong\mathcal Y(A_2)$.
The stacky fan of $\mathcal Y(A_2)$ appeared in 
example \ref{ex:stackyfanY(A_2)}.
Here we picture the types of pointed chains over the torus 
invariant divisors of the moduli stack $\overline{\mathcal L}_3$.

\bigskip
\noindent
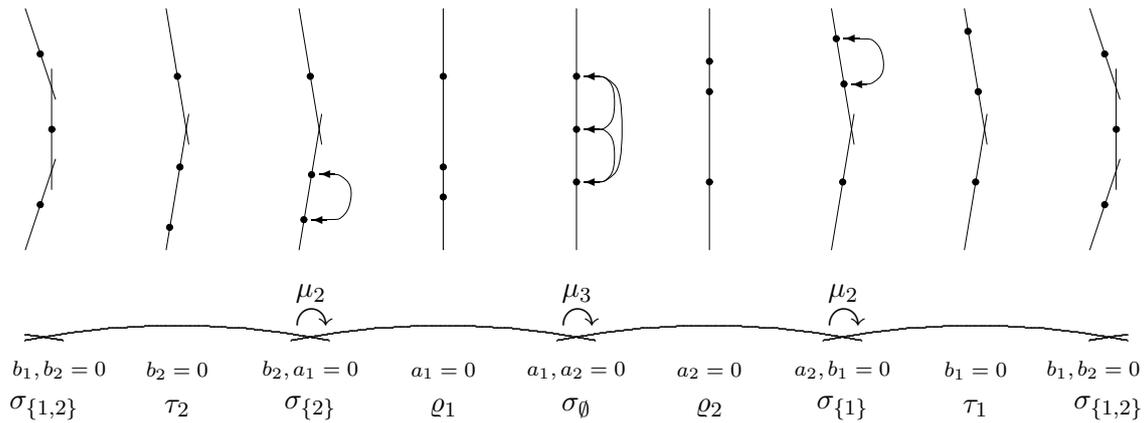
\begin{picture}(150,56)(0,0)
\put(0,10){
\put(2.5,0){
\qbezier(0,0)(20,4)(40,0)
\qbezier(35,0)(55,4)(75,0)
\qbezier(70,0)(90,4)(110,0)
\qbezier(105,0)(125,4)(145,0)
}
\qbezier(2.5,0.8)(7,0.1)(7.5,0)
\qbezier(147.5,0.8)(143,0.1)(142.5,0)
}
\put(7,6){\makebox(0,0)[c]{\tiny$b_1,b_2=0$}}
\put(5,1){\makebox(0,0)[c]{\small$\sigma_{\{1,2\}}$}}
\put(40,6){\makebox(0,0)[c]{\tiny$b_2,a_1=0$}}
\put(40,1){\makebox(0,0)[c]{\small$\sigma_{\{2\}}$}}
\put(75,6){\makebox(0,0)[c]{\tiny$a_1,a_2=0$}}
\put(75,1){\makebox(0,0)[c]{\small$\sigma_{\emptyset}$}}
\put(110,6){\makebox(0,0)[c]{\tiny$a_2,b_1=0$}}
\put(110,1){\makebox(0,0)[c]{\small$\sigma_{\{1\}}$}}
\put(143,6){\makebox(0,0)[c]{\tiny$b_1,b_2=0$}}
\put(145,1){\makebox(0,0)[c]{\small$\sigma_{\{1,2\}}$}}
\put(22.5,6){\makebox(0,0)[c]{\tiny$b_2=0$}}
\put(22.5,1){\makebox(0,0)[c]{\small$\tau_2$}}
\put(57.5,6){\makebox(0,0)[c]{\tiny$a_1=0$}}
\put(57.5,1){\makebox(0,0)[c]{\small$\rho_1$}}
\put(92.5,6){\makebox(0,0)[c]{\tiny$a_2=0$}}
\put(92.5,1){\makebox(0,0)[c]{\small$\rho_2$}}
\put(127.5,6){\makebox(0,0)[c]{\tiny$b_1=0$}}
\put(127.5,1){\makebox(0,0)[c]{\small$\tau_1$}}
\put(40.5,12){\makebox(0,0)[b]{\large$\curvearrowright$}}
\put(40.1,15){\makebox(0,0)[b]{\small$\mu_2$}}
\put(75.5,12){\makebox(0,0)[b]{\large$\curvearrowright$}}
\put(75.1,15){\makebox(0,0)[b]{\small$\mu_3$}}
\put(110.5,12){\makebox(0,0)[b]{\large$\curvearrowright$}}
\put(110.1,15){\makebox(0,0)[b]{\small$\mu_2$}}
\put(5,38){
\drawline(1,-8)(1,8)
\drawline(1.5,4)(-2.5,16)
\drawline(1.5,-4)(-2.5,-16)
\put(1,0){\circle*{0.8}}
\put(-0.5,10){\circle*{0.8}}
\put(-0.5,-10){\circle*{0.8}}
}
\put(22.5,38){
\drawline(1.5,-2)(-1.5,16)
\drawline(1.5,2)(-1.5,-16)
\put(0,7){\circle*{0.8}}
\put(0.3333,-5){\circle*{0.8}}
\put(-1,-13){\circle*{0.8}}
}
\put(40,38){
\drawline(1.5,-2)(-1.5,16)
\drawline(1.5,2)(-1.5,-16)
\put(0,7){\circle*{0.8}}
\put(0.1666,-6){\circle*{0.8}}
\put(-0.8333,-12){\circle*{0.8}}
\put(3.1666,-6){\vector(-1,0){2}}
\put(2.1666,-12){\vector(-1,0){2}}
\spline(3.2,-6)(4,-6.1)(5.5,-7.5)(5.2,-10.5)(4,-11.9)(3.2,-12)(2.2,-12)
}
\put(57.5,38){
\drawline(0,-16)(0,16)
\put(0,7){\circle*{0.8}}
\put(0,-5){\circle*{0.8}}
\put(0,-9){\circle*{0.8}}
}
\put(75,38){
\drawline(0,-16)(0,16)
\put(0,0){\circle*{0.8}}
\put(0,7){\circle*{0.8}}
\put(0,-7){\circle*{0.8}}
\put(3,7){\vector(-1,0){2}}
\put(3,0){\vector(-1,0){2}}
\put(3,-7){\vector(-1,0){2}}
\spline(3,0)(3.8,0.1)(5,1.5)(5,5.5)(3.8,6.9)(3,7)
\spline(3,0)(3.8,-0.1)(5,-1.5)(5,-5.5)(3.8,-6.9)(3,-7)
\spline(3,-7)(3.9,-6.9)(6,-5)(6,5)(3.9,6.9)(3,7)
}
\put(92.5,38){
\drawline(0,-16)(0,16)
\put(0,9){\circle*{0.8}}
\put(0,5){\circle*{0.8}}
\put(0,-7){\circle*{0.8}}
}
\put(110,38){
\drawline(1.5,-2)(-1.5,16)
\drawline(1.5,2)(-1.5,-16)
\put(0,-7){\circle*{0.8}}
\put(0.1666,6){\circle*{0.8}}
\put(-0.8333,12){\circle*{0.8}}
\put(3.1666,6){\vector(-1,0){2}}
\put(2.1666,12){\vector(-1,0){2}}
\spline(3.2,6)(4,6.1)(5.5,7.5)(5.2,10.5)(4,11.9)(3.2,12)(2.2,12)
}
\put(127.5,38){
\drawline(1.5,-2)(-1.5,16)
\drawline(1.5,2)(-1.5,-16)
\put(0,-7){\circle*{0.8}}
\put(0.3333,5){\circle*{0.8}}
\put(-1,13){\circle*{0.8}}
}
\put(145,38){
\drawline(1,-8)(1,8)
\drawline(1.5,4)(-2.5,16)
\drawline(1.5,-4)(-2.5,-16)
\put(1,0){\circle*{0.8}}
\put(-0.5,10){\circle*{0.8}}
\put(-0.5,-10){\circle*{0.8}}
}
\end{picture}
\end{ex}

\bigskip\bigskip

\pagebreak
\section{The functor of $X(A_{n-1})$, Losev-Manin moduli spaces\newline and the morphism to $\overline{\mathcal L}_n$}
\label{sec:LM}
\medskip

We start by comparing three descriptions of the functor of the 
toric variety associated with root systems of type $A$.
We use notations as in \cite[Section 2.1]{BB11a}, in particular we have
the lattice $M(A_{n-1})=\langle u_i-u_j:i,j\in\{1,\ldots,n\}\rangle
\subset\bigoplus_{i=1}^n\Z u_i$, generated by the roots $\beta_{ij}=u_i-u_j$
and forming the character lattice for the toric variety $X(A_{n-1})$. 
Its dual $N(A_{n-1})=\bigoplus_{i=1}^n\Z v_i/\sum_iv_i$,
where $(u_i)_i$ and $(v_i)_i$ are dual bases, is the 
lattice for the fan $\Sigma(A_{n-1})$ of $X(A_{n-1})$.

\medskip

The functor of the toric variety $X(A_{n-1})$ was described
in \cite{BB11a} in terms of $A_{n-1}$-data, i.e.\ 
families $(\mathscr L_{\{\pm\beta_{ij}\}},\{t_{\beta_{ij}},
t_{-\beta_{ij}}\})_{\{\pm\beta_{ij}\}}$ 
of line bundles with two generating sections that satisfy
$t_{\alpha}t_{\beta}t_{-\gamma}=t_{-\alpha}t_{-\beta}t_\gamma$
if $\gamma=\alpha+\beta$, up to isomorphism of line bundles with 
a pair of sections. With pull-back of line bundles and its sections
we have the functor $F_{A_{n-1}}$ of $A_{n-1}$-data, see 
\cite[Def.\ 1.17]{BB11a}.

On $X(A_{n-1})$ we have the universal $A_{n-1}$-data, which can be defined 
using the morphisms $\phi_{\{\pm\beta_{ij}\}}\colon X(A_{n-1})\to\P^1$ 
induced by pairs of opposite roots $\{\pm\beta_{ij}\}$ in $A_{n-1}$ 
(see \cite[Ex.\ 1.5 and 1.13]{BB11a}). We have homogeneous coordinates 
$z_{\beta_{ij}},z_{-\beta_{ij}}\in H^0(\P^1,\O_{\P^1}(1))$ 
such that
$x^{\beta_{ij}}=\phi_{\{\pm\beta_{ij}\}}^*(z_{\beta_{ij}}/z_{-\beta_{ij}})$,
where $x^u$ for $u\in M(A_{n-1})$ is the rational function corresponding
to an element of the root lattice.
Let $\mathscr L_{\{\pm\beta_{ij}\}}=\phi_{\{\pm\beta_{ij}\}}^*
\O_{\P^1}(1)$ and $t_{\beta_{ij}},t_{-\beta_{ij}}$ be the pull-back
of $z_{\beta_{ij}},z_{-\beta_{ij}}$. 

By \cite[Thm.\ 1.20]{BB11a} the toric variety $X(A_{n-1})$ together 
with the universal $A_{n-1}$-data represents the functor $F_{A_{n-1}}$.

\medskip

We can also apply the description of the functor of a smooth toric
variety of Cox \cite{Co95b} to $X(A_{n-1})$. The fan $\Sigma(A_{n-1})$
gives rise to the notion of a $\Sigma(A_{n-1})$-collection 
$((\mathscr L_I,w_I)_I,(c_{ij})_{ij})$ on a scheme $Y$, consisting 
of line bundles on $Y$ with a global section $(\mathscr L_I,w_I)$ 
for $\emptyset\neq I\subsetneq\{1,\ldots,n\}$ and isomorphisms
$c_{ij}\colon\big(\bigotimes_{i\in I,j\not\in I}\mathscr L_I\big)\otimes
\big(\bigotimes_{i\not\in I,j\in I}\mathscr L_I^{\otimes\minus1}\big)\to\O_Y$
for $i,j\in\{1,\ldots,n\}$, $i\neq j$, such that identifications
of the form $c_{ij}\otimes c_{jk}=c_{ik}$ hold.
These data have to satisfy the nondegeneracy condition that
for any point $y\in Y$ there are sets
$I_1\subset\ldots\subset I_{n-1}\subset\{1,\ldots,n\}$ with
$|I_i|=i$ such that $w_I(y)\neq 0$ if 
$I\neq I_1,\ldots,I_{n-1}$.
We denote the functor of $\Sigma(A_{n-1})$-collections by 
$C_{\Sigma(A_{n-1})}$.

On $X(A_{n-1})$ we have the universal $\Sigma(A_{n-1})$-collection 
given by the line bundles $\mathscr L_I=\O_{X(A_{n-1})}(D_I)$, 
where $D_I$ is the torus invariant prime divisor corresponding 
to the ray generated by $\sum_{i\in I}v_i$, with the section 
$w_I$ arising as the image of the $1$-section under the natural 
inclusion $\O_{X(A_{n-1})}\to\O_{X(A_{n-1})}(D_I)$ and the isomorphisms
$c_{ij}\colon\O_{X(A_{n-1})}(\sum_{i\in I,j\not\in I}D_I
-\sum_{i\not\in I,j\in I}D_I)\to\O_{X(A_{n-1})}$ induced by multiplication
with the rational functions $x^{\beta_{ij}}$ on $X(A_{n-1})$. 

By \cite{Co95b} the toric variety $X(A_{n-1})$ together with the 
universal $\Sigma(A_{n-1})$-collec-tion 
$((\O_{X(A_{n-1})}(D_I),w_I)_I,(c_{ij})_{ij})$ 
represents the functor $C_{\Sigma(A_{n-1})}$.

\medskip

As both functors $C_{\Sigma(A_{n-1})}$ and $F_{A_{n-1}}$ are
isomorphic to the functor of the toric variety $X(A_{n-1})$,
we have an isomorphism of functors $C_{\Sigma(A_{n-1})}\to F_{A_{n-1}}$, 
which we describe explicitely.

\begin{prop}\label{prop:SigmaAdata-Adata}
By the following procedure we can construct $A_{n-1}$-data\linebreak
$(\mathscr L_{\{\pm\beta_{ij}\}},\{t_{\beta_{ij}},
t_{-\beta_{ij}}\})_{\{\pm\beta_{ij}\}}$
out of a $\Sigma(A_{n-1})$-collection 
$((\mathscr L_I,w_I)_I,(c_{ij})_{ij})$ over\linebreak 
a scheme $Y$: for a pair of opposite roots $\pm\beta_{ij}$ in 
$A_{n-1}$ we have isomorphisms
$\bigotimes_{i\in I\!,\,j\not\in I}\mathscr L_I
\;\stackrel{\sim}{\longleftrightarrow}\;
\bigotimes_{i\not\in I\!,\,j\in I}\mathscr L_I$ 
of line bundles on $Y$ defined by $c_{ij},c_{ji}$ inverse to 
each other, and we let $\mathscr L_{\{\pm\beta_{ij}\}}$ be a 
line bundles in the same isomorphism class, with the sections 
$t_{\beta_{ij}},t_{-\beta_{ij}}$ defined as the images of 
$\prod_{i\in I,j\not\in I}w_I,\prod_{i\not\in I,j\in I}w_I$ 
in $\mathscr L_{\{\pm\beta_{ij}\}}$ under isomorphisms 
compatible with the above.
This construction defines an isomorphism of functors 
$C_{\Sigma(A_{n-1})}\to F_{A_{n-1}}$, mapping the universal 
$\Sigma(A_{n-1})$-collection to the universal $A_{n-1}$-data. 
\end{prop}
\begin{proof}
This construction defines a morphism of functors
$C_{\Sigma(A_{n-1})}\to F_{A_{n-1}}$, in particular the requirement 
that the two sections $t_{\pm\beta_{ij}}$ as defined in the 
construction generate the line bundle $\mathscr L_{\{\pm\beta_{ij}\}}$
follows from the nondegeneracy condition of $\Sigma(A_{n-1})$-data.  
One can show that this morphism of functors is an isomorphism by showing 
that it coincides with the composition of isomorphisms
$C_{\Sigma(A_{n-1})}\to\Mor(\,\cdot\,,X(A_{n-1}))\to F_{A_{n-1}}$.
This follows from the fact that the universal 
$\Sigma(A_{n-1})$-collection is mapped to the universal 
$A_{n-1}$-data, which is easy to verify.
\end{proof}

\medskip

Considering the universal data on $X(A_{n-1})$, we have 
isomorphisms
$\bigotimes_{i\in I}\mathscr L_I\stackrel{\sim\:}{\to}
\bigotimes_{j\in I}\mathscr L_I$ 
via multiplication by the rational function $x^{\beta_{ij}}$.
For any chosen $j\in\{1,\ldots,n\}$ we may define 
$x_1,\ldots,x_n\in H^0(X(A_{n-1}),\bigotimes_{j\in I}\mathscr L_I)$ 
as images of the sections $\prod_{1\in I}w_I,\ldots,\prod_{n\in I}w_I$ 
under these isomorphisms. We then have $x^{\beta_{ij}}=x_i/x_j$.

\begin{defi}\label{def:linebundlesL_j}
Given an ordering $i_1,\ldots,i_n$ of the set $\{1,\ldots,n\}$,
we define line bundles $\mathscr L_1,\ldots,\mathscr L_{n-1}$ on 
$X(A_{n-1})$ and sections $x_J$ of $\mathscr L_{|J|}$. Let 
\[\textstyle
\mathscr L_j=\bigotimes_I\mathscr L_I^{
\otimes(|\{i_1,\ldots,i_j\}\cap I|-\max\{0,|I|+j-n\})}
\]
be defined in terms of the universal $\Sigma(A_{n-1})$-collection.
The line bundles\linebreak $\bigotimes_I\mathscr L_I^{
\otimes(|J\cap I|-\max\{0,|I|+j-n\})}$ for any $J\subset\{1,\ldots,n\}$
of cardinality $j$ are isomorphic to $\mathscr L_j$ via multiplication 
by $\prod_{i\in J}x_i/\prod_{k=1}^jx_{i_k}$ (can also be 
expressed in terms of the isomorphisms $c_{ij}$ being part of the
universal $\Sigma(A_{n-1})$-collection).
We define $x_J\in H^0(X(A_{n-1}),\mathscr L_j)$ as the image of 
$\prod_Iw_I^{|J\cap I|-\max\{0,|I|+j-n\}}$ under this isomorphism.
\end{defi}

For these sections $x_J\in H^0(X(A_{n-1}),\mathscr L_{|J|})$
we have equations of rational functions\\[-4mm]
\begin{center}
$\prod_{j\in J}x_j/\prod_{j\in J'}x_j=x_J/x_{J'}
=\prod_{j\in J'}x_{\{1,\ldots,n\}\setminus\{j\}}/
\prod_{j\in J}x_{\{1,\ldots,n\}\setminus\{j\}}$.
\end{center}

\begin{rem}\label{rem:L_j-Delta_j}
The line bundle $\mathscr L_j$ was defined as $\O_{X(A_{n-1})}(D)$ 
in terms of the divisor $D=\sum_Id_ID_i$, where 
$d_I=|\{i_1,\ldots,i_j\}\cap I|-\max\{0,|I|+j-n\}$, 
which corresponds to the lattice polytope 
\[\textstyle
\begin{array}{rcl}
\Delta_j(A_{n-1})
&=&\bigcap_{I}\{\,u\in M(A_{n-1})_\Q\::\:\sum_{i\in I}v_i(u)\geq-d_I\,\}\\
&=&\conv\,\{\,\sum_{i\in J}u_i-\sum_{k=1}^ju_{i_k}:|J|=j\,\}\\
\end{array}
\] 
in $M(A_{n-1})_\Q$; the elements $x_J\in H^0(X(A_{n-1}),\mathscr L_j)$
for $|J|=j$ form a basis of global sections.
Different choices of the ordering fixed in the definition give rise to 
translated polytopes. 
\end{rem}

The line bundle $\mathscr L_1$ with its basis of global sections
$x_1,\ldots,x_n$ defines a morphism $X(A_{n-1})\to\P^{n-1}$.
This morphism is a composition of toric blow-ups as described
in \cite[(4.3.13)]{Ka93}. 
It maps the divisors $D_{\{i\}}$ of $X(A_{n-1})$ to the torus 
invariant prime divisor $D_i=\{x_i=0\}$ of $\P^{n-1}$, and more 
generally $D_I$ to $\bigcap_{i\in I}D_i$.
Similarly, we have a morphism $X(A_{n-1})\to\P^{n-1}$ defined
by the line bundle $\mathscr L_{n-1}$ mapping the divisor
$D_{\{1,\ldots,n\}\setminus\{i\}}$ of $X(A_{n-1})$ to the torus 
invariant prime divisor $D_i^*=\{x_{\{1,\ldots,n\}\setminus\{i\}}=0\}$ 
of $\P^{n-1}$, and more generally $D_{\{1,\ldots,n\}\setminus I}$ 
to $\bigcap_{i\in I}D_i^*$, see also \cite[(4.3.14)]{Ka93}.
We also consider the morphisms defined by the other line 
bundles $\mathscr L_j$:

\begin{prop}\label{prop:embX(A)prodP}
The line bundle $\mathscr L_j$ is generated by the basis of global 
sections $(x_J)_{|J|=j}$, it determines a projective toric morphism 
\[\textstyle
X(A_{n-1})\;\to\;\P(\langle x_J:|J|=j\rangle)\cong\P^{\binom{n}{j}-1}
\]
which is birational onto its image.
Together, these morphisms form a closed embedding
\begin{equation}\label{eq:embX(A)prodP}
\textstyle
X(A_{n-1})\;\to\;\prod_{j=1}^{n-1}\P(\langle x_J:|J|=j\rangle)
\cong\prod_{j=1}^{n-1}\P^{\binom{n}{j}-1}.
\end{equation}
The subscheme $X(A_{n-1})$ in this product is defined by
homogeneous equations
\begin{equation}\label{eq:X(A)subprodP}
\textstyle
\prod_{i=1}^lx_{J_i}\;=\;\prod_{i=1}^lx_{J'_i}
\end{equation}
where $\emptyset\neq J_i,J'_i\subsetneq\{1,\ldots,n\}$ such that
$|J_i|=|J'_i|$ and the equation for the characteristic functions
$\sum_i\chi_{J_i}=\sum_i\chi_{J'_i}$ is satisfied.
\end{prop}
\begin{proof}
That $\mathscr L_j$ is generated by global sections $(x_J)_{|J|=j}$ 
and determines a projective toric morphism follows from the 
fact that $\mathscr L_j$ can be reconstructed from the polytope 
$\Delta_j(A_{n-1})$, see remark \ref{rem:L_j-Delta_j}. This morphism
is birational onto its image since the polytope is full-dimensional.
We describe these morphisms in terms of the corresponding maps of fans.

For the toric variety $\P(\langle x_J:|J|=j\rangle)$ we have the
character lattice $M(A_{n-1})_j\subset\bigoplus_{|J|=j}\Z u_J$ 
generated by differences $u_J-u_{J'}$ and the dual lattice
$N(A_{n-1})_j=(\bigoplus_{|J|=j}\Z v_J)/(\sum_Jv_J)$.
The fan of $\P(\langle x_J:|J|=j\rangle)$ has the one-dimensional
cones generated by the $v_J$.
The morphism $X(A_{n-1})\to\P(\langle x_J:|J|=j\rangle)$ is 
determined by the map of lattices $M(A_{n-1})_j\to M(A_{n-1})$,
$u_J\mapsto\sum_{i\in J}u_i$, or dually
$N(A_{n-1})\to N(A_{n-1})_j$, $v_i\mapsto\sum_{i\in J}v_J$,
which defines a map of fans.

The product of these morphisms is given by the map of lattices
$\bigoplus_{j=1}^nM(A_{n-1})_j\to M(A_{n-1})$ with kernel
generated by elements of the form $\sum_in_{J_i}-\sum_in_{J'_i}$ 
such that $|J_i|=|J'_i|$ and $\sum_i\chi_{J_i}=\sum_i\chi_{J'_i}$. 
This gives rise to the homogeneous equations.

The maximal cones of the fan of $\P(\langle x_J:|J|=j\rangle)$
are the cones $\sigma_J$ generated by $\{v_{J'}\:|\:J'\neq J,|J'|=j\}$.
For an ordering $i_1,\ldots,i_n$ of $\{1,\ldots,n\}$ the preimage 
of the maximal cone 
$\sigma_{\{i_n\}}\times\sigma_{\{i_n,i_{n-1}\}}\times\ldots\times
\sigma_{\{i_n,\ldots,i_2\}}$ of the product fan is the
maximal cone of $\Sigma(A_{n-1})$ generated by
$v_{i_1},v_{i_1}\!+v_{i_2},\ldots,v_{i_1}\!+\ldots+v_{i_{n-1}}$,
thus the open sets corresponding to maximal cones of this form 
cover the image of $X(A_{n-1})$.
The corresponding maps of coordinate algebras are surjective,
so the morphism is a closed embedding.
\end{proof}

\begin{rem}
The $(n-1)$-dimensional permutohedron, usually defined in 
an $n$-dimensional vector space as convex hull of the orbit of 
$(n,n\!-\!1,\ldots,1)$ under the action of the symmetric group $S_n$ 
permuting the given basis (cf.\ for example \cite[(4.3.10)]{Ka93}), 
can be considered as a lattice polytope in $M(A_{n-1})_\Q\subset\Q^n$ 
after a translation moving one of its vertices, specified by 
fixing an ordering $i_1,\ldots,i_n$ of the set $\{1,\ldots,n\}$, 
to the origin:
\[\textstyle
\Delta(A_{n-1})=\conv\big\{\sum_{k=1}^{n-1}(n-k)u_{\sigma(k)}
-\sum_{k=1}^{n-1}(n-k)u_{i_k}:\sigma\in S_n\big\}.
\]
We have Minkowski sum decompositions of the permutohedron, first 
\[\textstyle
\Delta(A_{n-1})=\sum_{k<j}l_{i_ji_k}
\]
into line segments $l_{ij}=\{r\cdot\beta_{ij}\:|\:0\leq r\leq 1\}$ 
corresponding to the line bundles $\mathscr L_{\{\pm\beta_{ij}\}}$
forming the universal $A_{n-1}$-data 
(choosing $\O_{X(A_{n-1})}(\sum_{i_k\in I,i_j\not\in I}D_I)$ 
in the isomorphism class of $\mathscr L_{\{\pm\beta_{i_ji_k}\}}$
if $k<j$), and second
\[\Delta(A_{n-1})=\Delta_1(A_{n-1})+\ldots+\Delta_{n-1}(A_{n-1})\]
into the polytopes corresponding to the line bundles $\mathscr L_j$.
\end{rem}

\begin{rem}
The closed embedding (\ref{eq:embX(A)prodP}) together with the 
functor of projective spaces gives another description of 
the functor of the toric variety $X(A_{n-1})$.
We have a contravariant functor on the category of schemes: 
its data on a scheme $Y$ are line bundles with generating sections 
$(\mathscr L_j,(x_J)_{|J|=j})_{j=1,\ldots,n-1}$ 
up to isomorphism such that the sections satisfy the relations  
{\rm(\ref{eq:X(A)subprodP})}, and for morphisms of schemes we 
have the pull-back of line bundles with sections.
We call the data on $X(A_{n-1})$ introduced in definition
\ref{def:linebundlesL_j} the universal data on $X(A_{n-1})$.
Then, the toric variety $X(A_{n-1})$ together with the universal
data represents this functor. 
Further, the method of definition \ref{def:linebundlesL_j}
applied to $\Sigma(A_{n-1})$-data over arbitrary schemes
gives a morphism from $C_{\Sigma(A_{n-1})}$ to this functor,
mapping the universal $\Sigma(A_{n-1})$-collection to the universal 
data. As in the proof of proposition \ref{prop:SigmaAdata-Adata} 
this implies that we have an isomorphism of functors.
\end{rem}

The following observation can be directly calculated from the 
definition of the line bundles $\mathscr L_j$.

\begin{lemma}
We have isomorphisms
\begin{equation}\label{eq:isomL_jL_J}
\textstyle
\mathscr L_{j-1}^{\otimes\,\minus1}\otimes
\mathscr L_j^{\otimes 2}\otimes\mathscr L_{j+1}^{\otimes\,\minus1}
\;\cong\;\bigotimes_{|J|=n-j}\mathscr L_J
\end{equation}
where we set $\mathscr L_0=\mathscr L_n=\O_{X(A_{n-1})}$.
\end{lemma}

\begin{defi}
We define the divisors $C_1,\ldots,C_{n-1}$ and $D_1,\ldots,D_{n-1}$
on $X(A_{n-1})$.
Let $D_j=\sum_{|J|=n-j}D_J$ and let $C_j$ be the zero divisor of the 
section $\sum_{|J|=j}x_J$ of the line bundle $\mathscr L_j$.
\end{defi}

\begin{rem}
We may write the isomorphism (\ref{eq:isomL_jL_J}) as 
linear equivalence of divisors
\[2C_j-C_{j-1}-C_{j+1}\;\sim\;D_j\,.\]
The rational function
\begin{equation}\label{eq:ratfunctCD}
\frac{(\sum_{|I|=j-1}x_I)(\sum_{|I|=j+1}x_I)}{(\sum_{|I|=j}x_I)^2}
\end{equation}
has divisor $D_j+C_{j-1}+C_{j+1}-2C_j$.
\end{rem}

\begin{lemma}\label{le:CDempty}
For $j=1,\ldots,n\!-\!1$ we have $C_j\cap D_j=\emptyset$.
\end{lemma}
\begin{proof}
Can easily be checked locally using the covering of the following remark.
\end{proof}

\begin{rem}\label{rem:coveringWsigma}
Given a $\Sigma(A_{n-1})$-collection 
$((\mathscr L_I,w_I)_I,(c_{ij})_{ij})$ on a scheme $Y$, 
by nondegeneracy we have the following covering of $Y$ 
by open subschemes: for a permutation $\sigma\in S_n$ set 
$\mathcal I_\sigma\!=\!\big\{\{\sigma(n)\},
\{\sigma(n),\sigma(n\!-\!1)\},
\ldots,\{\sigma(n),\ldots,\sigma(2)\}\big\}$
and let $W_\sigma$ be the open subscheme of $Y$ where $w_I\neq0$ 
for $I\not\in\mathcal I_\sigma$.

In the case of the universal $\Sigma(A_{n-1})$-collection on 
$X(A_{n-1})$ the subscheme $W_\sigma\subset X(A_{n-1})$ 
corresponds to the maximal cone 
$\langle v_{\sigma(n)},\ldots,
v_{\sigma(n)}+\ldots+v_{\sigma(2)}\rangle 
\subset N(A_{n-1})_\Q$ dual to the cone generated by the simple roots 
$u_{\sigma(n)}-u_{\sigma(n-1)},\ldots,u_{\sigma(2)}-u_{\sigma(1)}$
and has as coordinate algebra the polynomial ring generated by
$\frac{x_{\sigma(n)}}{x_{\sigma(n-1)}},\ldots,
\frac{x_{\sigma(2)}}{x_{\sigma(1)}}$.
\end{rem}

By \cite[Thm.\ 3.19]{BB11a} there is an isomorphism between the
functor $F_{A_{n-1}}$ and the moduli functor of 
$n$-pointed chains of $\P^1$ $\overline{L}_n$  
mapping the universal $A_{n-1}$-data to the universal 
$n$-pointed chain
$(X(A_n)\to X(A_{n-1}),s_-,s_+,s_1,\ldots,s_n)$ defined in 
\cite[Con.\ 3.6]{BB11a}. This means that the toric variety 
$X(A_{n-1})$ coincides with the Losev-Manin moduli space 
$\overline{L}_n$ (we use the same symbol for the functor and
the moduli space). The construction uses an embedding of 
$n$-pointed chains into $(\P^1)^n$.

\medskip

This also implies that there is an isomorphism between the functor 
$C_{\Sigma(A_{n-1})}$ and the moduli functor $\overline{L}_n$ compatible 
with the other isomorphisms of functors. We make this isomorphism  
explicit using an embedding of $n$-pointed chains into $\P^n$.

\begin{con}\label{con:SigmaAdata-curve}
Let $((\mathscr L_I,w_I)_I,(c_{ij})_{ij})$ be a 
$\Sigma(A_{n-1})$-collection over a scheme $Y$. We construct an 
$n$-pointed chain of $\P^1$ $(C\to Y,s_-,s_+,s_1,\ldots,s_n)$
using the covering of $Y$ by $(W_\sigma)_{\sigma\in S_n}$ 
(see remark \ref{rem:coveringWsigma}).

For $\sigma\!\in\!S_n\!$ the restricted $\Sigma(A_{n-1})$-collection 
$((\mathscr L_I|_{W_\sigma},w_I|_{W_\sigma}\!)_I,
(c_{ij}|_{W_\sigma}\!)_{ij})$ 
is isomorphic to a $\Sigma(A_{n-1})$-collection
$((\mathscr L_I^\sigma,w_I^\sigma)_I,(c_{ij}^\sigma)_{ij})$ on $W_\sigma$
with the property 
$(\mathscr L_I^\sigma,w_I^\sigma)=(\O_{W_\sigma},1)$
for $I\not\in\mathcal I_\sigma$,
and for $i=1,\ldots,n-1$ we have isomorphisms 
$c^\sigma_{\sigma(i+1),\sigma(i)}:\linebreak
\mathscr L_{\{\sigma(n),\ldots,\sigma(i+1)\}}^\sigma\to\O_{W_\sigma}$.
Let $w_i^\sigma\in\O_{W_\sigma}(W_\sigma)$ be the image of
$w_{\{\sigma(n),\ldots,\sigma(i+1)\}}$.
Equivalently, we can use the restricted original data
$((\mathscr L_I|_{W_\sigma},w_I|_{W_\sigma})_I,(c_{ij}|_{W_\sigma})_{ij})$
and the image of the respective product of the restricted 
$w_I$'s under $c_{\sigma(i+1),\sigma(i)}|_{W_\sigma}$.

From these functions $w_1^\sigma,\ldots,w_{n-1}^\sigma$ we construct
an $n$-pointed chain over $W_\sigma$ embedded in the projective space
$\P^n_{W_\sigma}$ with homogeneous coordinates $y_0,\ldots,y_n$.
Let $C_\sigma$ be the subscheme of $\P^n_{W_\sigma}$ defined by 
the equations $y_iy_{j+1}=w_{i+1}^\sigma\cdots w_j^\sigma y_{i+1}y_j$
for $0\leq i<j<n$ (cf.\ construction \ref{con:data->curve}), the 
sections $s_{\sigma(i)}^\sigma$ defined by the additional equation 
$y_{i-1}=y_i\neq 0$, and let $s_-^\sigma,s_+^\sigma$ be the sections 
$(1\!:\!0:\ldots:\!0),(0\!:\ldots:\!0\!:\!1)$. 

These $n$-pointed chains $(C_\sigma\to W_\sigma,s_-^\sigma,s_+^\sigma,
s_1^\sigma,\ldots,s_n^\sigma)$ can be glued to an $n$-pointed chain
$(C\to Y,s_-,s_+,s_1,\ldots,s_n)$ over $Y$.
\end{con}

\begin{prop}\label{prop:SigmaAdata-curve}
Construction \ref{con:SigmaAdata-curve} is valid and defines an 
isomorphism between the functor $C_{\Sigma(A_{n-1})}$ and the 
moduli functor of $n$-pointed chains of $\P^1$ mapping 
the universal $\Sigma(A_{n-1})$-collection
$((\O_{X(A_{n-1})}(D_I),w_I)_I,(c_{ij})_{ij})$ 
to the universal $n$-pointed chain 
$(X(A_n)\to X(A_{n-1}),s_-,s_+,s_1,\ldots,s_n)$.
\end{prop}
\begin{proof}
Given $\Sigma(A_{n-1})$-data over a scheme $Y$, it is easy to show 
that construction \ref{con:SigmaAdata-curve} locally over 
the open subschemes $W_\sigma\subseteq Y$ defines $n$-pointed chains of 
$\P^1$ (compare also to construction \ref{con:data->curve}).

We show that, applying the isomorphism of functors $\overline{L}_n
\to F_{A_{n-1}}$ to these objects over $W_\sigma$, we obtain
$A_{n-1}$-data which coincide with the data we get by applying the 
functor $C_{\Sigma(A_{n-1})}\to F_{A_{n-1}}$ to the restricted data. 
According to \cite[Section 3.3]{BB11a} we extract $A_{n-1}$-data from 
an $n$-pointed chain 
$(C_\sigma\to W_\sigma,s_-^\sigma,s_+^\sigma,s_1^\sigma,\ldots,s_n^\sigma)$
via projections to $\P^1_{W_\sigma}$ such that $s_-^\sigma,s_+^\sigma$ 
become the $(1\!:\!0),(0\!:\!1)$-section and a given section $s_i^\sigma$ 
becomes the section $(1\!:\!1)$. 
In the present case for $i=1,\ldots,n$ the morphism determined by the 
rational functions $1,y_i/y_{i-1}$ restricted to the component of
$C_\sigma$ containing $s^\sigma_{\sigma(i)}$ after contracting
the other components transforms the sections 
$s_-^\sigma,s_+^\sigma,s_{\sigma(i)}^\sigma$ into the 
$(1\!:\!0),(0\!:\!1),(1\!:\!1)$-sections. 
For $n=1,\ldots,n\!-\!1$ the section $s^\sigma_{\sigma(i+1)}$ becomes the 
section $(w_i^\sigma\!:\!1)$, and this gives 
$(t_{-\beta_{\sigma(i),\sigma(i+1)}}\!:t_{\beta_{\sigma(i),\sigma(i+1)}})
=(w_i^\sigma\!:\!1)$ which coincides with the data obtained via 
proposition \ref{prop:SigmaAdata-Adata}.

Thus, the chains over the open subschemes $W_\sigma$ can be glued 
to an $n$-pointed chain over the scheme $Y$ and construction
\ref{con:SigmaAdata-curve} defines a morphism of functors 
$C_{\Sigma(A_{n-1})}\to\overline{L}_n$, such that its composition
with $\overline{L}_n\to F_{A_{n-1}}$ coincides with the isomorphism
of functors $C_{\Sigma(A_{n-1})}\to F_{A_{n-1}}$ defined in 
proposition \ref{prop:SigmaAdata-Adata}.
Since the other morphisms of functors are isomorphisms and map the 
given universal objects to the given universal objects, this is also 
true for $C_{\Sigma(A_{n-1})}\to\overline{L}_n$.
\end{proof}

\begin{thm}\label{thm:datamorphLM}
The morphism $\overline{L}_n\to\overline{\mathcal L}_n$ 
that arises by forgetting the labels of the $n$ sections
is given by the following $\mathbf\Upsilon(A_{n-1})$-collection
on $\overline{L}_n=X(A_{n-1})$: for $i=1,\ldots,n\!-\!1$ let 
$\mathscr L_{\rho_i}=\O_{\overline{L}_n}(C_i)$ and 
$\mathscr L_{\tau_i}=\O_{\overline{L}_n}(D_i)$,
let\\[-5mm] 
\begin{center}
$c_i:\;\;\mathscr L_{\tau_i}\mathscr L_{\rho_{i-1}}
\mathscr L_{\rho_i}^{\otimes\,\minus2}\,\mathscr L_{\rho_{i+1}}\!
=\O_{\overline{L}_n}(D_i+C_{i-1}\!-2C_i+C_{i+1})
\;\;\to\;\;\O_{\overline{L}_n}$
\end{center}\ \\[-4mm]
be given by multiplication by the rational function 
{\rm(\ref{eq:ratfunctCD})}, and let the sections $a_i,b_i$ 
be defined as the images of the $1$-sections under the inclusions 
$\O_{\overline{L}_n}\to\mathscr L_{\rho_i}$,
$\O_{\overline{L}_n}\to\mathscr L_{\tau_i}$.
\end{thm}
\begin{proof}
The data defined form a $\mathbf\Upsilon(A_{n-1})$-collection,
nondegeneracy follows from $C_i\cap D_i=\emptyset$, see 
lemma \ref{le:CDempty}.

We use the covering by $W_\sigma$, $\sigma\in S_n$ (see remark 
\ref{rem:coveringWsigma}). We have an isomorphism 
$((\mathscr L_{\rho_i},a_i)_i,(\mathscr L_{\tau_i},b_i)_i,
(c_i)_i)|_{W_\sigma}\to((\O_{W_\sigma},a^\sigma_i)_i,
(\O_{W_\sigma},b^\sigma_i)_i,(\id)_i)$ 
of $\mathbf\Upsilon(A_{n-1})$-collections on $W_\sigma$ consisting 
of isomorphisms $\mathscr L_{\tau_i}|_{W_\sigma}\to\O_{W_\sigma}$ 
given by multiplication with $x_{\sigma(i+1)}/x_{\sigma(i)}$ 
(compare to construction \ref{con:SigmaAdata-curve})
and $\mathscr L_{\rho_i}|_{W_\sigma}\to\O_{W_\sigma}$ 
by multiplication with
$(\sum_{|I|=i}x_I)/x_{\{\sigma(1),\ldots,\sigma(i)\}}$.

We show that the degree-$n$-pointed chain constructed from these
data coincides with the degree-$n$-pointed chain that arises by
forgetting the labels of the universal $n$-pointed chain
coming from the universal $\Sigma(A_{n-1})$-collection
by proposition \ref{prop:SigmaAdata-curve}.

Applying construction \ref{con:data->curve} to these data, we get
a chain of $\P^1$ $C\subset\P^n_{W_\sigma}$ defined by 
the functions $b^\sigma_i=x_{\sigma(i+1)}/x_{\sigma(i)}$ and 
a subscheme $S\subset C$ finite of degree $n$ over $W_\sigma$
defined by the functions
$a^\sigma_i=(\sum_{|I|=i}x_I)/x_{\{\sigma(1),\ldots,\sigma(i)\}}$.

Applying construction \ref{con:SigmaAdata-curve} to the universal 
$\Sigma(A_{n-1})$-collection on $X(A_{n-1})$, locally over 
$W_\sigma\subset X(A_{n-1})$ again we get the chain of $\P^1$ 
$C\subset\P^n_{W_\sigma}$ defined by the functions 
$b^\sigma_i=x_{\sigma(i+1)}/x_{\sigma(i)}$. The $n$ sections are
\[\textstyle
s_{\sigma(i)}
=(\;\ldots\::\frac{x_{\sigma(i)}^2}{x_{\sigma(i-2)}x_{\sigma(i-1)}}:
\frac{x_{\sigma(i)}}{x_{\sigma(i-1)}}:1:1:
\frac{x_{\sigma(i+1)}}{x_{\sigma(i)}}:
\frac{x_{\sigma(i+1)}x_{\sigma(i+2)}}{x_{\sigma(i)}^2}:\:\ldots\;)
\] 
in terms the coordinates $y_0,\ldots,y_n$ of $\P^n_{W_\sigma}$,
that is, we have $y_{i-1}(s_{\sigma(i)})=y_i(s_{\sigma(i)})$
which we may set to $1$, and then 
$y_k(s_{\sigma(i)})=x_{\sigma(i)}^{i-k}\,
\frac{x_{\sigma(1)}\cdots x_{\sigma(k)}}
{x_{\sigma(1)}\cdots x_{\sigma(i)}}$.
The sections are contained in the hyperplane 
defined by $\sum_{k=0}^n(-)^ka^\sigma_ky_k=0$ 
(set $a^\sigma_0=a^\sigma_n=1$):\\

$\sum_{k=0}^n(-)^ka_k^\sigma y_k(s_{\sigma(i)})\;\;=\;\;
\frac{x_{\sigma(i)}^i}{x_{\sigma(1)}\cdots x_{\sigma(i)}}
\sum_{k=0}^n(-)^k\sum\limits_{|I|=k}\frac{x_I\,x_{\sigma(1)}
\cdots x_{\sigma(k)}}
{x_{\{\sigma(1),\ldots,\sigma(k)\}}x_{\sigma(i)}^k}$\\

$=\;\frac{x_{\sigma(i)}^i}{x_{\sigma(1)}\cdots x_{\sigma(i)}}
\sum_{k=0}^n(-)^k\Big(
\!\sum\limits_{\substack{\scriptscriptstyle|I|=k\\
\scriptscriptstyle\sigma(i)\not\in I}}
\frac{x_I\,x_{\sigma(1)}\cdots x_{\sigma(k)}}
{x_{\{\sigma(1),\ldots,\sigma(k)\}}x_{\sigma(i)}^k}
\!+\!\!\sum\limits_{\substack{\scriptscriptstyle|I|=k\\
\scriptscriptstyle\sigma(i)\in I}}
\frac{x_I\,x_{\sigma(1)}\cdots x_{\sigma(k)}}
{x_{\{\sigma(1),\ldots,\sigma(k)\}}x_{\sigma(i)}^k}
\Big)$\\

$=\frac{x_{\sigma(i)}^i}{x_{\sigma(1)}\cdots x_{\sigma(i)}}
\sum_{k=0}^n\Big(
(-)^k\sum\limits_{\substack{\scriptscriptstyle|I|=k\\
\scriptscriptstyle\sigma(i)\not\in I}}
\frac{x_I\,x_{\sigma(1)}\cdots x_{\sigma(k)}}
{x_{\{\sigma(1),\ldots,\sigma(k)\}}x_{\sigma(i)}^k}
\;\;-$\\[-2mm]

$\hspace{5.8cm}(-)^{k-1}
\sum\limits_{\substack{\scriptscriptstyle|I|=k-1\\
\scriptscriptstyle\sigma(i)\not\in I}}
\frac{x_I\,x_{\sigma(1)}\cdots x_{\sigma(k-1)}}
{x_{\{\sigma(1),\ldots,\sigma(k-1)\}}x_{\sigma(i)}^{k-1}}
\Big)$\\[-4mm]

$=\;0$

\bigskip

The relative effective divisors $\sum_is_i$ and $S$ in $C$ over 
$W_\sigma$ coincide since they coincide over the open dense subscheme
of $\overline{L}_n$ parametrising chains with distinct sections.
\end{proof}

\begin{rem}
The results of this section imply a construction of a morphism
of fibred categories from the functor of $\Sigma(A_{n-1})$-collections, 
considering the $S_n$-operation on this functor, to the category of 
$\mathbf\Upsilon(A_{n-1})$-collections such that the diagram
\[
\begin{array}{ccc}
C_{\Sigma(A_{n-1})}&\stackrel{\sim}{\longleftrightarrow}
&\overline{L}_n\\[1mm]
\downarrow\qquad\;&&\downarrow\;\\[1mm]
\mathcal C_{\mathbf\Upsilon(A_{n-1})}&\stackrel{\sim}{\longleftrightarrow}
&\overline{\mathcal L}_n\\
\end{array}
\]
commutes. 
\end{rem}

\begin{ex}\label{ex:L_2-L_2}
In the case $n=2$ the functor $F_{A_1}\cong C_{\Sigma(A_1)}
\to\mathcal C_{\mathbf\Upsilon(A_1)}$ maps an object 
$(\mathscr L_{\{\pm\beta_{12}\}},\{t_{\beta_{12}},t_{-\beta_{12}}\})$ 
of $F_{A_1}$ over a scheme $Y$ to the object
\[
\begin{array}{l}
((\O_Y(C_1),a_1),(\O_Y(D_1),b_1),c_1\colon\O_Y(D_1)\otimes
\O_Y(C_1)^{\otimes\,\minus2}\stackrel{\sim\;}{\to}\O_Y)\cong\\
((\mathscr L_{\{\pm\beta_{12}\}},t_{\beta_{12}}+t_{-\beta_{12}}),
(\mathscr L_{\{\pm\beta_{12}\}}^{\otimes2},t_{\beta_{12}}t_{-\beta_{12}}),
\mathscr L_{\{\pm\beta_{12}\}}^{\otimes2}\otimes
\mathscr L_{\{\pm\beta_{12}\}}^{\otimes\,\minus2}
\stackrel{\sim\;}{\to}\O_Y)\\
\end{array}
\] 
of $\mathcal C_{\mathbf\Upsilon(A_1)}$, cf.\ example \ref{ex:L_2}.
\end{ex}

\medskip
\section{Pointed chains with involution and Cartan matrices\newline of type $B$ and $C$}
\label{sec:BC}
\medskip

As a natural variation of $\overline{\mathcal L}_n$ we consider
moduli stacks $\overline{\mathcal L}_n^{\pm}$ of stable 
degree-$2n$-pointed chains of $\P^1$ with an involution.

\begin{defi}\label{def:L_n+-}
We define the fibred category $\overline{\mathcal L}_n^{\pm}$ of 
stable degree-$2n$-pointed chains of $\P^1$ with involution.
An object over a scheme $Y$ is a collection $(C\to Y,I,s_-,s_+,S)$,
where $(C\to Y,s_-,s_+,S)$ is a stable degree-$2n$-pointed chain of 
$\P^1$ over $Y$ (definition \ref{def:L_n}), $I$ an 
automorphism of $C$ over $Y$ such that $I^2=\id_C$ and $I(s_-)=s_+$, 
and $S$ is invariant under $I$.
Morphisms between objects are morphisms of degree-$2n$-pointed 
chains which commute with the involution $I$.
\end{defi}

As in the case of $\overline{\mathcal L}_n$, see proposition 
\ref{prop:L_n-stack}, the fibred category 
$\overline{\mathcal L}_n^{\pm}$ is a stack in the fpqc topology
with representable finite diagonal.

\medskip

Considering degree-$2n$-pointed chains of $\P^1$ with 
involution as degree-$2n$-pointed chains defines a morphism of
stacks $\overline{\mathcal L}_n^{\pm}\to\overline{\mathcal L}_{2n}$
which makes $\overline{\mathcal L}_n^{\pm}$ a subcategory
of $\overline{\mathcal L}_{2n}$ but in general not a substack, 
because a stable degree-$2n$-pointed chain may have
automorphisms not commuting with an additional involution.

\medskip

The moduli stack  $\overline{\mathcal L}_n^{\pm}$ decomposes,
unless we are working in characteristic $2$, into two components 
$\overline{\mathcal L}_n^{\pm}=\overline{\mathcal L}_{n,+}^{\pm}
\cup\overline{\mathcal L}_{n,-}^{\pm}$, where the component 
$\overline{\mathcal L}_{n,+}^{\pm}$ parametrises isomorphism 
classes of stable degree-$2n$-pointed chains with involution
$(C,I,s_-,s_+,S)$ such that the degree of $S$ in each of the
fixed points under the involution is even.
We first consider this main component $\overline{\mathcal L}_{n,+}^{\pm}$.
 
\medskip

The component $\overline{\mathcal L}_{n,+}^{\pm}$ is related to 
the moduli space $\overline{L}_n^{\pm}\cong X(C_n)$ of $2n$-pointed
chains with involution defined in \cite[Section 6]{BB11b}. There is 
a morphism $\overline{L}_n^{\pm}\to\overline{\mathcal L}_{n,+}^{\pm}$
forgetting the labels of the sections.
This morphism is equivariant with respect to the natural
action of the Weyl group $W(C_n)=(\Z/2\Z)^n\rtimes S_n$
on $\overline{L}_n^{\pm}$, the coarse moduli space of 
$\overline{\mathcal L}_{n,+}^{\pm}$ is $\overline{L}_n^{\pm}/W(C_n)$. 
Similar as in proposition \ref{prop:L_n-L_n} one can show 
that the morphism
$\overline{L}_n^{\pm}\to\overline{\mathcal L}_{n,+}^{\pm}$ 
is faithfully flat and finite of degree $|W(C_n)|=2^nn!$. 

\medskip

These morphisms together with the morphisms  
$\overline{L}_{2n}\to\overline{\mathcal L}_{2n}$  
(see section \ref{sec:modulistacks}) and 
$\overline{L}_n^{\pm}\to\overline{L}_{2n}$ 
(see \cite[Rem.\ 6.16]{BB11b}) form a commutative diagram
\[
\begin{array}{ccc}
\overline{L}_n^{\pm}&\longrightarrow&\overline{L}_{2n}\\
\downarrow&&\downarrow\\
\overline{\mathcal L}_{n,+}^{\pm}&\longrightarrow
&\overline{\mathcal L}_{2n}\\
\end{array}
\]
where $\overline{L}_n^{\pm}$ is a component of the fibred product.

\medskip

The stack $\overline{\mathcal L}_{n,+}^{\pm}$ compactifies the space 
of finite subschemes of degree $2n$ in
$\P^1\setminus\{0,\infty\}$ which are invariant under the involution 
and of even degree in each of the fixed points of the involution. 
Equivalently, this is the space of polynomials $\sum_{i=0}^{2n}a'_iy^i$ 
of degree $2n$ with the symmetry $a'_{2n-i}=a'_i$ in the coefficients,
up to change of the variable by multiplication by $-1$. 
These polynomials can contain $y-1$ and $y+1$ only with
even multiplicity.
After dividing by the coefficient $a'_{2n}=a'_0$, we have a
polynomial of the form 
\[
y^{-n}+a_{n-1}y^{-n+1}+\ldots+a_1y^{-1}+a_0+a_1y+\ldots
+a_{n-1}y^{n-1}+y^n
\]
determined by the isomorphism class up to multiplication 
of $y$ with $-1$ (together with multiplication of the whole 
expression by $(-1)^n$).

\medskip

In general, embedding a chain $(C,I,s_-,s_+,S)$ into the projective 
space $\P^{2n}=\P(H^0(C,\O_C(S)))$, the image of $C$ is given 
by equations arising from the $2\times 2$ minors of a matrix 
of the form (decompose into several matrices if some of the $b_i$
are zero, cf.\ remark \ref{rem:embK}; symbol $\sqrt{b_0}$ introduced 
for symmetry reasons)
\[
\left(
\begin{matrix}
\cdots&y_{\minus2}&y_{\minus1}&\sqrt{b_0}y_0&\sqrt{b_0}b_1y_1&\cdots\\
\cdots&\sqrt{b_0}b_1y_{\minus1}&\sqrt{b_0}y_0&y_1&y_2&\cdots\\
\end{matrix}
\right)
\]
where $y_{\minus n},\ldots,y_0,\ldots,y_n$ is a basis of $H^0(C,\O_C(S))$
defined similar as in proposition \ref{prop:embK}, \ref{prop:embY} 
and such that the involution maps $y_{\minus i}\leftrightarrow y_i$.
The sections $s_-,s_+$ become the sections 
$(1\!:\!0\!:\ldots:\!0),(0\!:\ldots:\!0\!:\!1)$ and the subscheme 
$S\subset C\subset\P^{2n}$ is determined by an equation
\[
y_{\minus n}+a_{n-1}y_{\minus(n-1)}+\ldots+a_1y_{\minus1}
+a_0y_0+a_1y_1+\ldots+a_{n-1}y_{n-1}+y_n=0.
\]

\medskip

For an algebraically closed field $K$ a $K$-valued point of 
$\overline{\mathcal L}_{n,+}^{\pm}$ corresponds to a collection 
$(a_{n-1},\ldots,a_0,b_{n-1},\ldots,b_0)\in K^{2n}$ up to the 
equivalence
\[
(a_{n-1},\ldots,a_0,b_{n-1},\ldots,b_0)
\sim(\kappa_{n-1}a_{n-1},\ldots,\kappa_0a_0,
\lambda_{n-1}b_{n-1},\ldots,\lambda_0b_0)
\]
with $(\kappa_{n-1},\ldots,\kappa_0,\lambda_{n-1},\ldots,\lambda_0)
\in(K^*)^{2n}$ satisfying
$\lambda_{n-1}=\kappa_{n-1}^2/\kappa_{n-2}$,
$\lambda_{n-2}=\kappa_{n-2}^2/(\kappa_{n-3}\kappa_{n-1}),\;
\ldots\,,\;
\lambda_1=\kappa_1^2/(\kappa_0\kappa_2),\;
\lambda_0=\kappa_0^2/\kappa_1^2.$
This gives rise to a toric orbifold whose exact sequence
of tori 
\[
1\longrightarrow G\!\cong\!(\G_m)^n\longrightarrow(\G_m)^{2n}
\longrightarrow T_N\!\cong\!(\G_m)^n\longrightarrow 1
\]
corresponds to the exact sequence of lattices
\[
0\longrightarrow M\!\cong\!\Z^n
\stackrel{\left(\begin{smallmatrix}\minus\,C\\I_n\end{smallmatrix}\right)}
{\longrightarrow}\Z^{2n}
\stackrel{\left(\begin{smallmatrix}I_n\;C\\\end{smallmatrix}\right)}
{\longrightarrow}\Z^{n}\longrightarrow 0
\]
where $C=C(C_n)^\top$ is the transpose of the Cartan matrix
\[
C(C_n)\;=\;
\left(\quad
\begin{matrix}
2&\minus1&0&\cdots&0\\[-1mm]
\minus1&2&\ddots&\ddots&\vdots\\[-1mm]
0&\ddots&\ddots&\minus1&0\\[-1mm]
\vdots&\ddots&\minus1&2&\minus1\\[1mm]
0&\cdots&0&\minus2&2\\
\end{matrix}
\quad\right)
\] 
of the root system $C_n$.

\medskip

\begin{defi}\label{def:Y(C_n)}
We define the toric orbifold $\mathcal Y(C_n)$ associated to the 
Cartan matrix of the root system $C_n$ in terms of the stacky fan 
$\mathbf\Upsilon(C_n)=(N,\Upsilon(C_n),\beta)$, where $N=\Z^n$ and 
the linear map $\beta\colon\Z^{2n}\to N$ is given by the 
$n\!\times\!2n$ matrix $(\minus C(C_n)\;I_n)$.
The fan $\Upsilon(C_n)$ has the $2n$ one-dimensional cones
$\rho_{n-1},\ldots,\rho_0,\tau_{n-1},\ldots,\tau_0$ generated by 
the columns of the matrix $(\minus C(C_n)\;I_n)$. A subset of 
one-dimensional cones generates a higher dimensional cone 
of $\Upsilon(C_n)$ if it does not contain one of the sets
$\{\rho_0,\tau_0\},\ldots,\{\rho_{n-1},\tau_{n-1}\}$.
This defines a fan containing $2^n$ $n$-dimensional cones $\sigma_I$
generated by sets $\{\rho_i\}_{i\not\in I}\cup\{\tau_i\}_{i\in I}$
for subsets $I\subseteq\{0,\ldots,n-1\}$.
\end{defi}

The functor of $\mathbf\Upsilon(C_n)$-collections
$\mathcal C_{\mathbf\Upsilon(C_n)}\cong\mathcal Y(C_n)$
has objects of the form\linebreak
$((\mathscr L_{\rho_i},a_i)_{i=0,\ldots,n-1},
(\mathscr L_{\tau_i},b_i)_{i=0,\ldots,n-1},
(c_i)_{i=0,\ldots,n-1})$ over a scheme $Y$,
where the $c_i$ are isomorphisms of line bundles on $Y$

\medskip

$c_{n-1}\colon\mathscr L_{\tau_{n-1}}\!\otimes\!
\mathscr L_{\rho_{n-1}}^{\otimes\,\minus2}
\!\otimes\!\mathscr L_{\rho_{n-2}}\to\O_Y,$

$c_{n-2}\colon\mathscr L_{\tau_{n-2}}\!\otimes\!
\mathscr L_{\rho_{n-1}}\!\otimes\!
\mathscr L_{\rho_{n-2}}^{\otimes\,\minus2}\!\otimes\!
\mathscr L_{\rho_{n-3}}\to\O_Y,$

$\;\vdots$

$c_1\colon\mathscr L_{\tau_1}\!\otimes\!
\mathscr L_{\rho_2}\!\otimes\!
\mathscr L_{\rho_1}^{\otimes\,\minus2}\!\otimes\!
\mathscr L_{\rho_0}\to\O_Y,$

$c_0\colon\mathscr L_{\tau_0}\!\otimes\!
\mathscr L_{\rho_1}^{\otimes 2}\!\otimes\!
\mathscr L_{\rho_0}^{\otimes\,\minus2}\to\O_Y.$

\medskip
\noindent
We have a morphism of stacks $\mathcal C_{\mathbf\Upsilon(C_n)}\to
\mathcal C_{\mathbf\Upsilon(A_{2n-1})}$
by considering the collection
\begin{equation}\label{eq:C_ncoll}
\begin{array}{c}
((\mathscr L_{\rho_{n-1}},a_{n-1}),\ldots,
(\mathscr L_{\rho_0},a_0),\ldots,(\mathscr L_{\rho_{n-1}},a_{n-1}),
(\mathscr L_{\tau_{n-1}},b_{n-1}),\\
\ldots,(\mathscr L_{\tau_0},b_0),\ldots,
(\mathscr L_{\tau_{n-1}},b_{n-1}),c_{n-1},\ldots,c_0,
\ldots,c_{n-1}),\\
\end{array}
\end{equation}
built out of a $\mathbf\Upsilon(C_n)$-collection, as a 
$\mathbf\Upsilon(A_{2n-1})$-collection.
This morphism can be described by the map of fans
$\mathbf\Upsilon(C_{n})\to\mathbf\Upsilon(A_{2n-1})$
mapping $e'_{n-1}\mapsto e_{2n-1}+e_1,\ldots,e'_1\mapsto e_{n+1}+e_{n-1},
e'_0\mapsto e_n$, where $e'_{n-1},\ldots,e'_0$ are the generators 
of $\tau_{n-1},\ldots,\tau_0$ of $\mathbf\Upsilon(C_n)$ and 
$e_{2n-1},\ldots,e_1$ are those of $\tau_{2n-1},\ldots,\tau_1$ of 
$\mathbf\Upsilon(A_{2n-1})$. It corresponds to a toric morphism 
$\mathcal Y(C_n)\to\mathcal Y(A_{2n-1})$ making $\mathcal Y(C_n)$ 
a subcategory of $\mathcal Y(A_{2n-1})$.

\begin{thm}\label{thm:Y(Cn)}
There is an isomorphism of stacks 
$\;\overline{\mathcal L}_{n,+}^{\pm}\cong\,\mathcal Y(C_n)$.
\end{thm}
\begin{proof}
Applying construction \ref{con:curve->data} to a 
degree-$2n$-pointed chain with involution, it is 
possible to choose $y_0,\ldots,y_{2n}$ such that 
the involution maps $y_i\leftrightarrow y_{2n-i}$,
and we obtain a $\mathbf\Upsilon(A_{2n-1})$-collection 
of the form (\ref{eq:C_ncoll}).
Applying construction \ref{con:data->curve} to a 
$\mathbf\Upsilon(A_{2n-1})$-collection of the form 
(\ref{eq:C_ncoll}), making symmetric choices, we can 
introduce an involution on the resulting degree-$2n$-pointed 
chain by $y_i\leftrightarrow y_{2n-i}$.
A $\mathbf\Upsilon(A_{2n-1})$-collection of the form (\ref{eq:C_ncoll}) 
is equivalent to a $\mathbf\Upsilon(C_n)$-collection,
and further, morphisms of the corresponding degree-$2n$-chains 
with involution that commute with the involution are equivalent
to morphisms of $\mathbf\Upsilon(C_n)$-collections.
\end{proof}

\smallskip
\enlargethispage{2mm}

The case of the other component $\overline{\mathcal L}_{n,-}^{\pm}$ 
is very similar.
The stack $\overline{\mathcal L}_{n,-}^{\pm}$ parametrises 
isomorphism classes of stable degree-$2n$-pointed chains with 
involution
$(C,I,s_-,s_+,\linebreak S)$ such that the degree of $S$ in each of the 
fixed points of the involution is odd if there are two fixed points, 
and positive if there is only one fixed point.
It is related to the moduli stack $\mathcal X(C_{n-1})$ defined in 
\cite[Section 6]{BB11b}: there is a morphism 
$\mathcal X(C_{n-1})\to\overline{\mathcal L}_{n,-}^{\pm}$ 
determined by forgetting the labels of the sections and adding 
the fixed point subscheme of the involution as a subscheme of 
degree $2$ to the $2n-2$ sections.

\medskip

The stack $\overline{\mathcal L}_{n,-}^{\pm}$ compactifies the stack 
of finite subschemes $S$ of degree $2n$ in $\P^1\setminus\{0,\infty\}$ 
invariant under the involution such that $S$ has odd degree in
$(1\!:\!1)$ and $(1\!:\!\minus1)$ (positive degree in 
$(1\!:\!1)=(1\!:\!\minus1)$ in characteristic $2$).
Equivalently we may consider polynomials $\sum_{i=0}^{2n}a'_iy^i$ 
of degree $2n$ with the symmetry $a'_{2n-i}=-a'_i$ in the 
coefficients and $a'_n=0$. We may represent each isomorphism class 
by an expression of the form
\[
-y^{-n}-a_{n-1}y^{-n+1}-\ldots-a_1y^{-1}
+0+a_1y+\ldots+a_{n-1}y^{n-1}+y^n
\]
determined up to multiplication of $y$ with $\minus1$ 
(together with multiplication of the whole expression by $(\minus1)^n$).
It has a factor $(y-1)(y^{-1}+1)$, occurring with odd multiplicity 
in characteristic $\neq 2$.

\medskip

In general, similar as in the case of $\overline{\mathcal L}_{n,+}^{\pm}$, 
a not necessarily irreducible chain can be naturally embedded into
$\P^{2n}=\P(H^0(C,\O_C(S)))$ and described by equations
arising from the $2\times 2$ minors of a matrix of the form 
\[
\left(
\begin{matrix}
\cdots&y_{\minus2}&y_{\minus1}&y_0&b_1y_1&\cdots\\
\cdots&b_1y_{\minus1}&y_0&y_1&y_2&\cdots\\
\end{matrix}
\right)
\]
and the subscheme $S\subset C\subset\P^{2n}$ is given by an equation
\[
-y_{\minus n}-a_{n-1}y_{\minus(n-1)}-\ldots-a_1y_{\minus1}+a_1y_1+\ldots
+a_{n-1}y_{n-1}+y_n\;=\;0.
\]
Over an algebraically closed field $K$ a $K$-valued point of 
$\overline{\mathcal L}_{n,-}^{\pm}$ corresponds to a collection 
$(a_{n-1},\ldots,a_1,b_{n-1},\ldots,b_1)\in K^{2n-2}$ up to the 
equivalence
\[
(a_{n-1},\ldots,a_1,b_{n-1},\ldots,b_1)
\sim(\kappa_{n-1}a_{n-1},\ldots,\kappa_1a_1,
\lambda_{n-1}b_{n-1},\ldots,\lambda_1b_1)
\]
with $(\kappa_{n-1},\ldots,\kappa_1,\lambda_{n-1},\ldots,\lambda_1)
\in(K^*)^{2n-2}$ satisfying
$\lambda_{n-1}\!=\!\kappa_{n-1}^2/\kappa_{n-2},$
$\lambda_{n-2}=\kappa_{n-2}^2/(\kappa_{n-3}\kappa_{n-1}),\;
\ldots\,,\;
\lambda_2=\kappa_2^2/(\kappa_1\kappa_3),\;
\lambda_1^2=\kappa_1^2/\kappa_2^2.$

We will see that this stack can be described by a toric stack 
that differs from $\mathcal Y(C_{n-1})$ by replacing the matrix 
$(\minus C(C_{n-1})\;I_{n-1})$ defining the map $\beta$ of the 
stacky fan $\mathbf\Upsilon(C_{n-1})$ by the matrix 
\[
\left(
\qquad\minus C(C_{n-1})\qquad
\begin{matrix}
1&0&\\
0&\ddots&\ddots\\
&\ddots&1&0\\
&&0&2\\
\end{matrix}\;\;
\right).
\]
In the case $n=1$ we define it to be $B\mu_2$.
This toric stack corresponds to the category of collections 
of the form
$((\mathscr L_{\rho_i},a_i)_{i=1,\ldots,n-1},
(\mathscr L_{\tau_i},b_i)_{i=1,\ldots,n-1},
(c_i)_{i=1,\ldots,n-1})$ over a scheme $Y$,
where the $c_i$ are isomorphisms of line bundles

\smallskip

$c_{n-1}\colon\mathscr L_{\tau_{n-1}}\!\otimes\!
\mathscr L_{\rho_{n-1}}^{\otimes\,\minus2}
\!\otimes\!\mathscr L_{\rho_{n-2}}\to\O_Y,$

$c_{n-2}\colon\mathscr L_{\tau_{n-2}}\!\otimes\!
\mathscr L_{\rho_{n-1}}\!\otimes\!
\mathscr L_{\rho_{n-2}}^{\otimes\,\minus2}\!\otimes\!
\mathscr L_{\rho_{n-3}}\to\O_Y,$

$\;\vdots$

$c_2\colon\mathscr L_{\tau_2}\!\otimes\!
\mathscr L_{\rho_3}\!\otimes\!
\mathscr L_{\rho_2}^{\otimes\,\minus2}\!\otimes\!
\mathscr L_{\rho_1}\to\O_Y,$

$c_1\colon\mathscr L_{\tau_1}^{\otimes 2}\!\otimes\!
\mathscr L_{\rho_2}^{\otimes2}\!\otimes\!
\mathscr L_{\rho_1}^{\otimes\,\minus2}\to\O_Y.$

\begin{prop}
The stack $\overline{\mathcal L}_{n,-}^{\pm}$ is isomorphic to 
the above toric stack. It can be embedded into 
$\overline{\mathcal L}_{n,+}^{\pm}$ as the divisor $D_{\rho_0}$
corresponding to the cone $\rho_0$ and defined by $a_0=0$.
\end{prop}
\begin{proof}
Similar as in the proof of theorem \ref{thm:Y(Cn)} one can show 
that $\overline{\mathcal L}_{n,-}^{\pm}$ is isomorphic to the 
above toric stack by embedding this toric stack as a subcategory 
into $\mathcal Y(A_{2n-1})$.

An embedding of $\overline{\mathcal L}_{n,-}^{\pm}$ into
$\overline{\mathcal L}_{n,+}^{\pm}$ as divisor $D_{\rho_0}$
is given by mapping a collection 
$((\mathscr L_{\rho_i},a_i)_{i=1,\ldots,n-1},
(\mathscr L_{\tau_i},b_i)_{i=1,\ldots,n-1},
(c_i)_{i=1,\ldots,n-1})$
over a scheme $Y$ to the collection
$((\mathscr L_{\rho_i},a_i)_{i=0,\ldots,n-1},
(\mathscr L_{\tau_i},b_i)_{i=0,\ldots,n-1},
(c_i)_{i=0,\ldots,n-1})$
where $\mathscr L_{\rho_0}=\mathscr L_{\tau_1}\otimes\mathscr L_{\rho_2}$,
$a_0=0$ and $\mathscr L_{\tau_0}=\O_Y$, $b_0=1$ and
$c_0$ is defined using $c_1$.
This corresponds to mapping $(C\to Y,I,s_-,s_+,S_-)$ to
$(C\to Y,I,s_-,s_+,S_+)$ such that $S_+$ is  given by
\[
y_{\minus n}+a_{n-1}y_{\minus(n-1)}+\ldots+a_1y_{\minus1}+a_1y_1+\ldots
+a_{n-1}y_{n-1}+y_n
\]
if $S_-$ is given by
\[
-y_{\minus n}-a_{n-1}y_{\minus(n-1)}-\ldots-a_1y_{\minus1}+a_1y_1+\ldots
+a_{n-1}y_{n-1}+y_n.
\]
\end{proof}

\begin{ex} The toric orbifold $\overline{\mathcal L}_{1,+}^{\pm}$ 
is isomorphic to the weighted projective line 
$\P(1,2)\cong\mathcal Y(C_1)$. Here the inclusion as subcategory
$\overline{\mathcal L}_{1,+}^{\pm}\to\overline{\mathcal L}_2$ is
an isomorphism of stacks as any degree-$2$-pointed chain is isomorphic 
to a symmetric object under an involution whose isomorphisms commute
with the involution.
So we have the same situation as in examples \ref{ex:L_2},
\ref{ex:stackyfanY(A_1)}, \ref{ex:L_2-Y(A_1)}.
The component $\overline{\mathcal L}_{1,-}^{\pm}$ is isomorphic
to $B\mu_2$.
\end{ex}

\begin{ex}
The stacky fan of the toric orbifold 
$\overline{\mathcal L}_{2,+}^{\pm}\cong\mathcal Y(C_2)$ 
is given by the matrix
$
\left(\;
\begin{matrix}
\minus2&1&1&0\\
2&\minus2&0&1\\
\end{matrix}\;
\right)
$: 

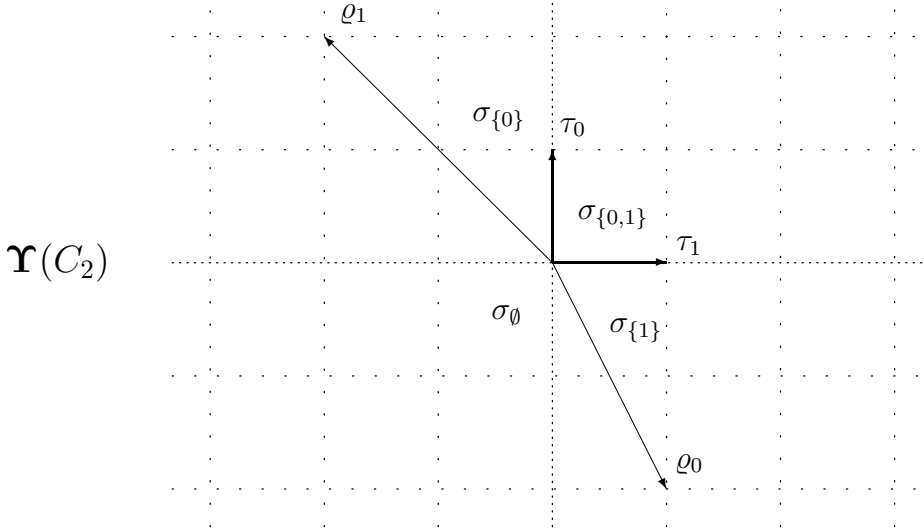
\bigskip
\noindent
\begin{picture}(150,70)(0,0)
\put(8,35){\makebox(0,0)[l]{\large$\mathbf\Upsilon(C_2)$}}

\put(5,-5){
\dottedline{1}(25,40)(125,40)
\dottedline{3}(25,55)(125,55)
\dottedline{3}(25,70)(125,70)
\dottedline{3}(25,25)(125,25)
\dottedline{3}(25,10)(125,10)

\dottedline{1}(75,5)(75,75)
\dottedline{3}(90,5)(90,75)
\dottedline{3}(105,5)(105,75)
\dottedline{3}(120,5)(120,75)
\dottedline{3}(60,5)(60,75)
\dottedline{3}(45,5)(45,75)
\dottedline{3}(30,5)(30,75)

\put(75,40){\vector(1,0){15}}\put(93,41){\makebox(0,0)[b]{$\tau_1$}}
\put(75,40){\vector(0,1){15}}\put(76,58){\makebox(0,0)[l]{$\tau_0$}}
\put(75,40){\vector(-1,1){30}}\put(47,73){\makebox(0,0)[l]{$\rho_1$}}
\put(75,40){\vector(1,-2){15}}\put(93,12){\makebox(0,0)[b]{$\rho_0$}}

\put(83,46){\makebox(0,0)[c]{$\sigma_{\{0,1\}}$}}
\put(68,59){\makebox(0,0)[c]{$\sigma_{\{0\}}$}}
\put(86,31){\makebox(0,0)[c]{$\sigma_{\{1\}}$}}
\put(69,33){\makebox(0,0)[c]{$\sigma_{\emptyset}$}}
}
\end{picture}

\bigskip
\noindent
We picture the types of pointed chains over the torus invariant 
divisors of the moduli stack $\overline{\mathcal L}_{2,+}^{\pm}$.

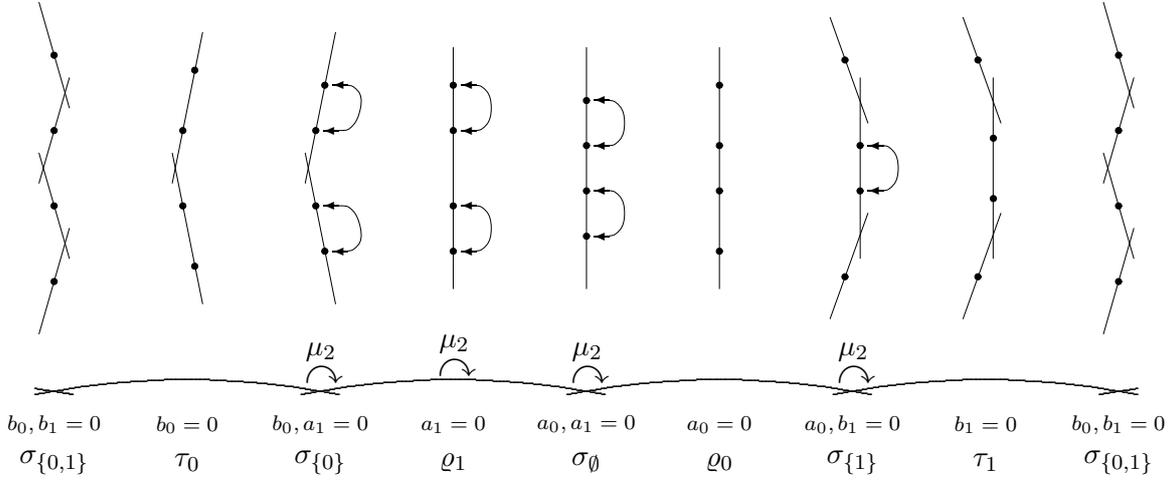
\medskip
\noindent
\begin{picture}(150,62)(0,0)
\put(0,10){
\put(2.5,0){
\qbezier(0,0)(20,4)(40,0)
\qbezier(35,0)(55,4)(75,0)
\qbezier(70,0)(90,4)(110,0)
\qbezier(105,0)(125,4)(145,0)
}
\qbezier(2.5,0.8)(7,0.1)(7.5,0)
\qbezier(147.5,0.8)(143,0.1)(142.5,0)
}
\put(5,6){\makebox(0,0)[c]{\tiny$b_0,b_1=0$}}
\put(5,1){\makebox(0,0)[c]{\small$\sigma_{\{0,1\}}$}}
\put(40,6){\makebox(0,0)[c]{\tiny$b_0,a_1=0$}}
\put(40,1){\makebox(0,0)[c]{\small$\sigma_{\{0\}}$}}
\put(40.5,11.5){\makebox(0,0)[b]{\large$\curvearrowright$}}
\put(40.1,14.5){\makebox(0,0)[b]{\small$\mu_2$}}
\put(75,6){\makebox(0,0)[c]{\tiny$a_0,a_1=0$}}
\put(75,1){\makebox(0,0)[c]{\small$\sigma_{\emptyset}$}}
\put(75.5,11.5){\makebox(0,0)[b]{\large$\curvearrowright$}}
\put(75.1,14.5){\makebox(0,0)[b]{\small$\mu_2$}}
\put(110,6){\makebox(0,0)[c]{\tiny$a_0,b_1=0$}}
\put(110,1){\makebox(0,0)[c]{\small$\sigma_{\{1\}}$}}
\put(110.5,11.5){\makebox(0,0)[b]{\large$\curvearrowright$}}
\put(110.1,14.5){\makebox(0,0)[b]{\small$\mu_2$}}
\put(145,6){\makebox(0,0)[c]{\tiny$b_0,b_1=0$}}
\put(145,1){\makebox(0,0)[c]{\small$\sigma_{\{0,1\}}$}}
\put(22.5,6){\makebox(0,0)[c]{\tiny$b_0=0$}}
\put(22.5,1){\makebox(0,0)[c]{\small$\tau_0$}}
\put(57.5,6){\makebox(0,0)[c]{\tiny$a_1=0$}}
\put(57.5,1){\makebox(0,0)[c]{\small$\rho_1$}}
\put(58,12.5){\makebox(0,0)[b]{\large$\curvearrowright$}}
\put(57.6,15.5){\makebox(0,0)[b]{\small$\mu_2$}}
\put(92.5,6){\makebox(0,0)[c]{\tiny$a_0=0$}}
\put(92.5,1){\makebox(0,0)[c]{\small$\rho_0$}}
\put(127.5,6){\makebox(0,0)[c]{\tiny$b_1=0$}}
\put(127.5,1){\makebox(0,0)[c]{\small$\tau_1$}}
\put(5,40){
\drawline(-2,22)(2,8)
\drawline(2,12)(-2,-2)
\drawline(-2,2)(2,-12)
\drawline(2,-8)(-2,-22)
\put(0,15){\circle*{0.8}}
\put(0,5){\circle*{0.8}}
\put(0,-5){\circle*{0.8}}
\put(0,-15){\circle*{0.8}}
}
\put(22.5,40){
\drawline(-2,-2)(2,18)
\drawline(-2,2)(2,-18)
\put(1,13){\circle*{0.8}}
\put(-0.6,5){\circle*{0.8}}
\put(1,-13){\circle*{0.8}}
\put(-0.6,-5){\circle*{0.8}}
}
\put(40,40){
\drawline(-2,-2)(2,18)
\drawline(-2,2)(2,-18)
\put(0.6,11){\circle*{0.8}}
\put(-0.6,5){\circle*{0.8}}
\put(0.6,-11){\circle*{0.8}}
\put(-0.6,-5){\circle*{0.8}}
\put(3.6,11){\vector(-1,0){2}}
\put(2.4,5){\vector(-1,0){2}}
\spline(2.4,5)(3,5)(3.8,5.1)(5,6.5)(5.6,9.5)(4.4,10.9)(3.6,11)
\put(3.6,-11){\vector(-1,0){2}}
\put(2.4,-5){\vector(-1,0){2}}
\spline(2.4,-5)(3,-5)(3.8,-5.1)(5,-6.5)(5.6,-9.5)(4.4,-10.9)(3.6,-11)
}
\put(57.5,40){
\drawline(0,-16)(0,16)
\put(0,11){\circle*{0.8}}
\put(0,5){\circle*{0.8}}
\put(0,-5){\circle*{0.8}}
\put(0,-11){\circle*{0.8}}
\put(3,11){\vector(-1,0){2}}
\put(3,5){\vector(-1,0){2}}
\spline(3,5)(3.8,5.1)(5,6.5)(5,9.5)(3.8,10.9)(3,11)
\put(3,-11){\vector(-1,0){2}}
\put(3,-5){\vector(-1,0){2}}
\spline(3,-5)(3.8,-5.1)(5,-6.5)(5,-9.5)(3.8,-10.9)(3,-11)
}
\put(75,40){
\drawline(0,-16)(0,16)
\put(0,9){\circle*{0.8}}
\put(0,3){\circle*{0.8}}
\put(0,-3){\circle*{0.8}}
\put(0,-9){\circle*{0.8}}
\put(3,9){\vector(-1,0){2}}
\put(3,3){\vector(-1,0){2}}
\spline(3,3)(3.8,3.1)(5,4.5)(5,7.5)(3.8,8.9)(3,9)
\put(3,-9){\vector(-1,0){2}}
\put(3,-3){\vector(-1,0){2}}
\spline(3,-3)(3.8,-3.1)(5,-4.5)(5,-7.5)(3.8,-8.9)(3,-9)
}
\put(92.5,40){
\drawline(0,-16)(0,16)
\put(0,11){\circle*{0.8}}
\put(0,3){\circle*{0.8}}
\put(0,-3){\circle*{0.8}}
\put(0,-11){\circle*{0.8}}
}
\put(110,40){
\drawline(1,-12)(1,12)
\drawline(2,6)(-3,20)
\drawline(2,-6)(-3,-20)
\put(1,3){\circle*{0.8}}
\put(1,-3){\circle*{0.8}}
\put(4,3){\vector(-1,0){2}}
\put(4,-3){\vector(-1,0){2}}
\spline(4,-3)(4.8,-2.9)(6,-1.5)(6,1.5)(4.8,2.9)(4,3)
\put(-1,14.375){\circle*{0.8}}
\put(-1,-14.375){\circle*{0.8}}
}
\put(127.5,40){
\drawline(1,-12)(1,12)
\drawline(2,6)(-3,20)
\drawline(2,-6)(-3,-20)
\put(1,4){\circle*{0.8}}
\put(1,-4){\circle*{0.8}}
\put(-1,14.375){\circle*{0.8}}
\put(-1,-14.375){\circle*{0.8}}
}
\put(145,40){
\drawline(-2,22)(2,8)
\drawline(2,12)(-2,-2)
\drawline(-2,2)(2,-12)
\drawline(2,-8)(-2,-22)
\put(0,15){\circle*{0.8}}
\put(0,5){\circle*{0.8}}
\put(0,-5){\circle*{0.8}}
\put(0,-15){\circle*{0.8}}
}
\end{picture}

\medskip

\noindent
The toric orbifold $\overline{\mathcal L}_{2,-}^{\pm}\!\cong\P^1/\mu_2$ 
corresponds to the stacky fan given by the matrix $(\;\minus2\;\;2\;)$.
We have the following types of pointed chains over 
$\overline{\mathcal L}_{2,-}^{\pm}$: 

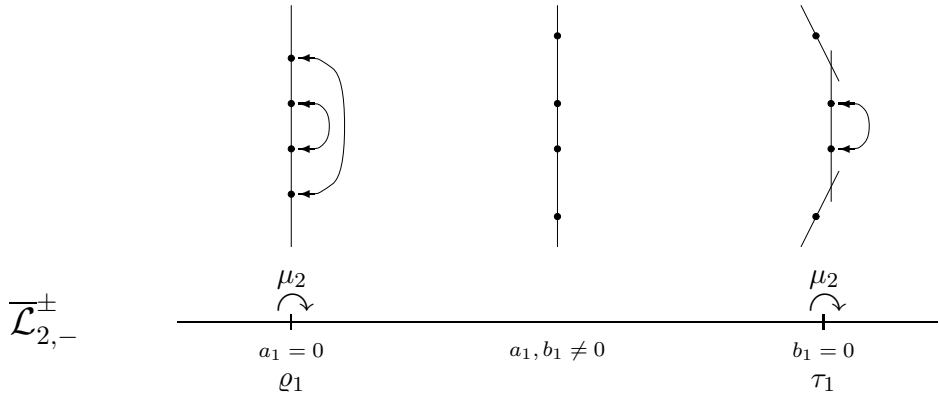
\medskip
\noindent
\begin{picture}(150,52)(0,0)
\put(8,10){\makebox(0,0)[l]{\large$\overline{\mathcal L}_{2,-}^\pm$}}
\put(5,0){
\put(25,10){\line(1,0){100}}
\put(40,5){\makebox(0,0)[b]{\tiny$a_1=0$}}
\put(40,2){\makebox(0,0)[c]{\small$\rho_1$}}
\put(40,9){\line(0,1){2}}
\put(40.5,11.5){\makebox(0,0)[b]{\large$\curvearrowright$}}
\put(40.1,14.5){\makebox(0,0)[b]{\small$\mu_2$}}
\put(75,5){\makebox(0,0)[b]{\tiny$a_1,b_1\neq 0$}}
\put(110,5){\makebox(0,0)[b]{\tiny$b_1=0$}}
\put(110,2){\makebox(0,0)[c]{\small$\tau_1$}}
\put(110,9){\line(0,1){2}}
\put(110.5,11.5){\makebox(0,0)[b]{\large$\curvearrowright$}}
\put(110.1,14.5){\makebox(0,0)[b]{\small$\mu_2$}}
\put(40,36){
\drawline(0,-16)(0,16)
\put(0,9){\circle*{0.8}}
\put(0,3){\circle*{0.8}}
\put(0,-3){\circle*{0.8}}
\put(0,-9){\circle*{0.8}}
\put(3,3){\vector(-1,0){2}}
\put(3,-3){\vector(-1,0){2}}
\spline(3,-3)(3.8,-2.9)(5,-1.5)(5,1.5)(3.8,2.9)(3,3)
\put(3,9){\vector(-1,0){2}}
\put(3,-9){\vector(-1,0){2}}
\spline(3,-9)(3.8,-8.9)(7,-6.5)(7,6.5)(3.8,8.9)(3,9)
}
\put(75,36){
\drawline(0,-16)(0,16)
\put(0,12){\circle*{0.8}}
\put(0,3){\circle*{0.8}}
\put(0,-3){\circle*{0.8}}
\put(0,-12){\circle*{0.8}}
}
\put(110,36){
\drawline(1,-10)(1,10)
\drawline(2,6)(-3,16)
\drawline(2,-6)(-3,-16)
\put(1,3){\circle*{0.8}}
\put(1,-3){\circle*{0.8}}
\put(4,3){\vector(-1,0){2}}
\put(4,-3){\vector(-1,0){2}}
\spline(4,-3)(4.8,-2.9)(6,-1.5)(6,1.5)(4.8,2.9)(4,3)
\put(-1,12){\circle*{0.8}}
\put(-1,-12){\circle*{0.8}}
}
}
\end{picture}
\end{ex}

\medskip

One also may consider chains with involution
and a subscheme of odd degree.

\begin{defi}
Let the fibred category $\overline{\mathcal L}_n^{0,\pm}$ of 
stable degree-$(2n\!+\!1)$-pointed chains of $\P^1$ with 
involution be defined analogously to definition \ref{def:L_n+-}.
\end{defi}

The fibred category $\overline{\mathcal L}_n^{0,\pm}$ is a 
stack in the fpqc topology with representable finite diagonal.

\medskip

The moduli stack $\overline{\mathcal L}_n^{0,\pm}$ forms a subcategory 
of $\overline{\mathcal L}_{2n+1}$.
It is related to the moduli space $\overline{L}_n^{0,\pm}\cong X(B_n)$ 
of $(2n\!+\!1)$-pointed chains with involution defined in 
\cite[Section 1]{BB11b}. We have a morphism 
$\overline{L}_n^{0,\pm}\to\overline{\mathcal L}_n^{0,\pm}$
forgetting the labels of the sections, which is equivariant with 
respect to the action of the Weyl group $W(B_n)=(\Z/2\Z)^n\rtimes S_n$
on $\overline{L}_n^{0,\pm}$. The coarse moduli space of 
$\overline{\mathcal L}_n^{0,\pm}$ is $\overline{L}_n^{0,\pm}/W(B_n)$.
As in the $C$-case the morphism
$\overline{L}_n^{0,\pm}\to\overline{\mathcal L}_n^{0,\pm}$ 
is faithfully flat and finite of degree $|W(B_n)|=2^nn!$,
and we have a commutative diagram
\[
\begin{array}{ccc}
\overline{L}_n^{0,\pm}&\longrightarrow&\overline{L}_{2n+1}\\
\downarrow&&\downarrow\\
\overline{\mathcal L}_n^{0,\pm}&\longrightarrow
&\overline{\mathcal L}_{2n+1}\\
\end{array}
\]

\medskip

Embedding a degree-$(2n+1)$-pointed chain with involution 
$(C,I,s_-,s_+,S)$ into the projective space 
$\P^{2n+1}=\P(H^0(C,\O_C(S)))$, the image of $C$ is given by 
equations arising from the $2\times 2$ minors of a matrix 
of the form
\[
\left(
\begin{matrix}
\cdots&y_{\minus5/2}&y_{\minus3/2}&y_{\minus1/2}
&b_1y_{1/2}
&b_1b_2y_{3/2}
&\cdots\\
\cdots&b_1b_2y_{\minus3/2}
&b_1y_{\minus1/2}&y_{1/2}
&y_{3/2}&y_{5/2}
&\cdots\\
\end{matrix}
\right)
\]
and the subscheme $S$ by 
\[
y_{\minus(2n+1)/2}+a_ny_{\minus(2n-1)/2}+\ldots+a_1y_{\minus1/2}
+a_1y_{1/2}+\ldots+a_ny_{(2n-1)/2}+y_{(2n+1)/2}.
\]
where $y_{\minus(2n+1)/2},\ldots,y_{\minus3/2},y_{\minus1/2},
y_{1/2},y_{3/2},\ldots,y_{(2n+1)/2}$ is a basis of $H^0(C,\O_C(S))$ 
defined similar as in proposition \ref{prop:embK}, \ref{prop:embY} 
and such that the involution maps $y_{\minus i/2}\leftrightarrow y_{i/2}$.

\begin{defi}
We define the toric orbifold $\mathcal Y(B_n)$  
in terms of the stacky fan $\mathbf\Upsilon(B_n)$
as in definition \ref{def:Y(C_n)} replacing
the Cartan matrix $C(C_n)$ of the root system $C_n$ 
by the Cartan matrix $C(B_n)$ of the root system $B_n$.
\end{defi}

It turns out that $\overline{\mathcal L}_n^{0,\pm}$ is not quite 
$\mathcal Y(B_n)$, but coincides with the underlying canonical toric 
stack $\mathcal Y(B_n)^{\textrm{can}}$ (as defined in \cite{FMN10}).
So instead of the Cartan matrix of the root system $B_n$
we have the matrix
\[
\left(\quad
\begin{matrix}
2&\minus1&0&\cdots&\cdots&0\\
\minus1&2&\ddots&\ddots&&\vdots\\
0&\ddots&\ddots&\ddots&\ddots&\vdots\\
\vdots&\ddots&\ddots&2&\minus1&0\\
\vdots&&\ddots&\minus1&2&\minus1\\[2mm]
0&\cdots&\cdots&0&\minus1&1\\
\end{matrix}
\quad\right)
\] 
where the rightmost column is half of the column of the
Cartan matrix. The
functor of $\mathbf\Upsilon(B_n)^{\textrm{can}}$-collections
$\mathcal C_{\mathbf\Upsilon(B_n)^{\textrm{can}}}\cong
\mathcal Y(B_n)^{\textrm{can}}$
has objects of the form
$((\mathscr L_{\rho_i},a_i)_{i=1,\ldots,n},\linebreak
(\mathscr L_{\tau_i},b_i)_{i=1,\ldots,n},
(c_i)_{i=1,\ldots,n})$ over a scheme $Y$,
where the $c_i$ are isomorphisms of line bundles
\[
\begin{array}{l}
c_n\colon\!\mathscr L_{\tau_n}\!\!\otimes\!
\mathscr L_{\rho_n}^{\otimes\minus2}
\!\otimes\!\mathscr L_{\rho_{n-1}}\!\!\to\O_Y,\;
c_{n-1}\colon\!\mathscr L_{\tau_{n-1}}\!\otimes\!
\mathscr L_{\rho_n}\!\!\otimes\!
\mathscr L_{\rho_{n-1}}^{\otimes\minus2}\!\!\otimes\!
\mathscr L_{\rho_{n-2}}\!\!\to\O_Y,\\
\ldots\;,\;
c_2\colon\mathscr L_{\tau_2}\!\otimes\!
\mathscr L_{\rho_3}\!\otimes\!
\mathscr L_{\rho_2}^{\otimes\minus2}\!\otimes\!
\mathscr L_{\rho_1}\to\O_Y,\;
c_1\colon\mathscr L_{\tau_1}\!\otimes\!
\mathscr L_{\rho_2}\!\otimes\!
\mathscr L_{\rho_1}^{\otimes\minus1}\to\O_Y.\\
\end{array}
\]
The inclusion as subcategory $\mathcal Y(B_n)^{\textrm{can}}\to
\mathcal Y(A_{2n})$ can be described as 
$\mathcal C_{\mathbf\Upsilon(B_n)^{\textrm{can}}}\to
\mathcal C_{\mathbf\Upsilon(A_{2n})}$ 
by considering the collection
\[
\begin{array}{c}
((\mathscr L_{\rho_n},a_n),\ldots,(\mathscr L_{\rho_1},a_1),
(\mathscr L_{\rho_1},a_1),\ldots,(\mathscr L_{\rho_n},a_n),\\
(\mathscr L_{\tau_n},b_n),\ldots,(\mathscr L_{\tau_1},b_1),
(\mathscr L_{\tau_1},b_1),\ldots,(\mathscr L_{\tau_n},b_n),
c_n,\ldots,c_1,c_1,\ldots,c_n),\\
\end{array}
\]
formed out of a $\mathbf\Upsilon(B_n)^{\textrm{can}}$-collection, as a 
$\mathbf\Upsilon(A_{2n})$-collection.

\pagebreak

As in the case of degree-$2n$-pointed chains with involution one can prove:

\begin{thm}
There is an isomorphism of stacks 
$\;\overline{\mathcal L}_n^{0,\pm}\cong\,\mathcal Y(B_n)^{\textrm{can}}$.
\end{thm}

\begin{ex} In the case $n=1$ we have a scheme 
$\overline{\mathcal L}_1^{0,\pm}\cong\,\mathcal Y(B_1)^{\textrm{can}}$ 
isomorphic to $\P^1$.
\end{ex}

\begin{ex}
The toric orbifold $\overline{\mathcal L}_2^{0,\pm}\cong\,
\mathcal Y(B_2)^{\textrm{can}}$ is given by
$
\left(\;
\begin{matrix}
\minus2&1&1&0\\
1&\minus1&0&1\\
\end{matrix}\;
\right)
$.
In the picture of the stacky fan $\mathbf\Upsilon(B_2)^{\textrm{can}}$
the dotted arrow corresponds to the generator of the ray
$\rho_1$ determined by the stacky fan $\mathbf\Upsilon(B_2)$.

\bigskip
\noindent
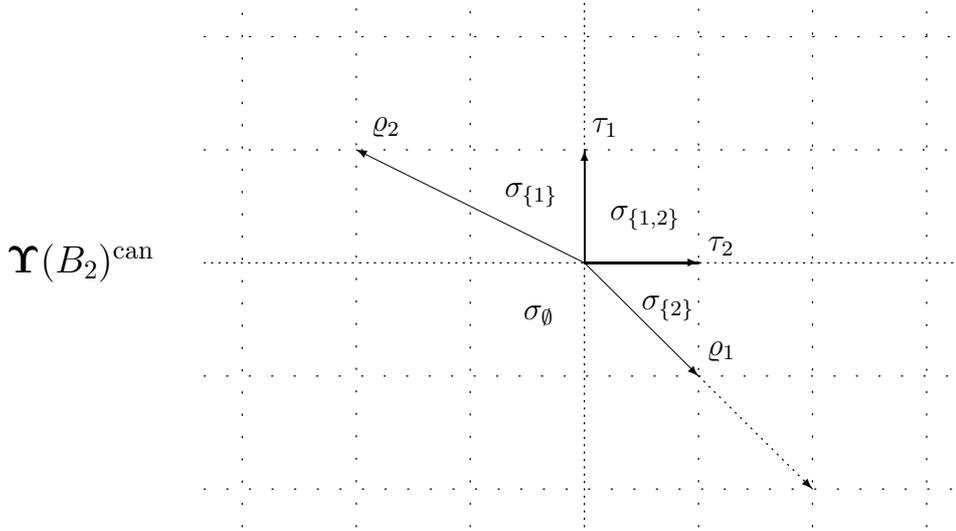
\begin{picture}(150,70)(0,0)
\put(4,35){\makebox(0,0)[l]{\large$\mathbf\Upsilon(B_2)^{\textrm{can}}$}}

\put(5,-5){
\dottedline{1}(25,40)(125,40)
\dottedline{3}(25,55)(125,55)
\dottedline{3}(25,70)(125,70)
\dottedline{3}(25,25)(125,25)
\dottedline{3}(25,10)(125,10)

\dottedline{1}(75,5)(75,75)
\dottedline{3}(90,5)(90,75)
\dottedline{3}(105,5)(105,75)
\dottedline{3}(120,5)(120,75)
\dottedline{3}(60,5)(60,75)
\dottedline{3}(45,5)(45,75)
\dottedline{3}(30,5)(30,75)

\put(75,40){\vector(1,0){15}}\put(93,41){\makebox(0,0)[b]{$\tau_2$}}
\put(75,40){\vector(0,1){15}}\put(76,58){\makebox(0,0)[l]{$\tau_1$}}
\put(75,40){\vector(-2,1){30}}\put(47,58){\makebox(0,0)[l]{$\rho_2$}}
\put(75,40){\vector(1,-1){15}}\put(93,27){\makebox(0,0)[b]{$\rho_1$}}
\dottedline{1}(90,25)(105,10)
\put(104,11){\vector(1,-1){1}}

\put(83,46){\makebox(0,0)[c]{$\sigma_{\{1,2\}}$}}
\put(68,49){\makebox(0,0)[c]{$\sigma_{\{1\}}$}}
\put(86,34){\makebox(0,0)[c]{$\sigma_{\{2\}}$}}
\put(69,33){\makebox(0,0)[c]{$\sigma_{\emptyset}$}}
}
\end{picture}

\medskip
\bigskip
\noindent
We picture the types of pointed chains over the torus invariant 
divisors of the moduli stack $\overline{\mathcal L}_2^{0,\pm}$.

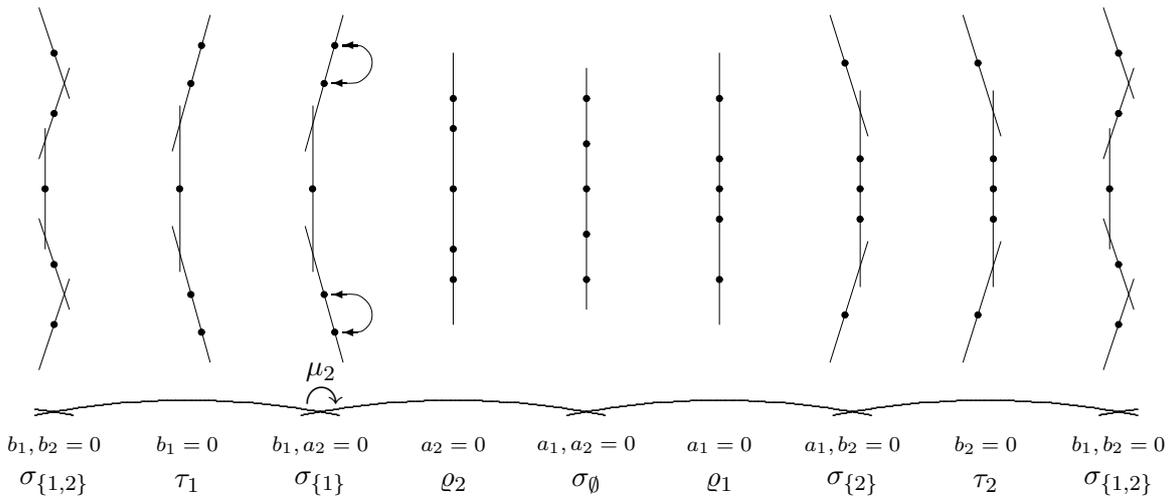
\medskip
\noindent
\begin{picture}(150,70)(0,0)
\put(0,10){
\put(2.5,0){
\qbezier(0,0)(20,4)(40,0)
\qbezier(35,0)(55,4)(75,0)
\qbezier(70,0)(90,4)(110,0)
\qbezier(105,0)(125,4)(145,0)
}
\qbezier(2.5,0.8)(7,0.1)(7.5,0)
\qbezier(147.5,0.8)(143,0.1)(142.5,0)
}
\put(5,6){\makebox(0,0)[c]{\tiny$b_1,b_2=0$}}
\put(5,1){\makebox(0,0)[c]{\small$\sigma_{\{1,2\}}$}}
\put(40,6){\makebox(0,0)[c]{\tiny$b_1,a_2=0$}}
\put(40,1){\makebox(0,0)[c]{\small$\sigma_{\{1\}}$}}
\put(40.5,11.5){\makebox(0,0)[b]{\large$\curvearrowright$}}
\put(40.1,14.5){\makebox(0,0)[b]{\small$\mu_2$}}
\put(75,6){\makebox(0,0)[c]{\tiny$a_1,a_2=0$}}
\put(75,1){\makebox(0,0)[c]{\small$\sigma_{\emptyset}$}}
\put(110,6){\makebox(0,0)[c]{\tiny$a_1,b_2=0$}}
\put(110,1){\makebox(0,0)[c]{\small$\sigma_{\{2\}}$}}
\put(145,6){\makebox(0,0)[c]{\tiny$b_1,b_2=0$}}
\put(145,1){\makebox(0,0)[c]{\small$\sigma_{\{1,2\}}$}}
\put(22.5,6){\makebox(0,0)[c]{\tiny$b_1=0$}}
\put(22.5,1){\makebox(0,0)[c]{\small$\tau_1$}}
\put(57.5,6){\makebox(0,0)[c]{\tiny$a_2=0$}}
\put(57.5,1){\makebox(0,0)[c]{\small$\rho_2$}}
\put(92.5,6){\makebox(0,0)[c]{\tiny$a_1=0$}}
\put(92.5,1){\makebox(0,0)[c]{\small$\rho_1$}}
\put(127.5,6){\makebox(0,0)[c]{\tiny$b_2=0$}}
\put(127.5,1){\makebox(0,0)[c]{\small$\tau_2$}}
\put(5,40){
\drawline(-2,24)(2,12)
\drawline(2,16)(-2,4)
\drawline(-1.2,8)(-1.2,-8)
\drawline(-2,-4)(2,-16)
\drawline(2,-12)(-2,-24)
\put(0,18){\circle*{0.8}}
\put(0,10){\circle*{0.8}}
\put(-1.2,0){\circle*{0.8}}
\put(0,-10){\circle*{0.8}}
\put(0,-18){\circle*{0.8}}
}
\put(22.5,40){
\drawline(-1,-11)(-1,11)
\drawline(-2,5)(3,23)
\drawline(-2,-5)(3,-23)
\put(-1,0){\circle*{0.8}}
\put(0.5,14){\circle*{0.8}}
\put(0.5,-14){\circle*{0.8}}
\put(1.9,19){\circle*{0.8}}
\put(1.9,-19){\circle*{0.8}}
}
\put(40,40){
\drawline(-1,-11)(-1,11)
\drawline(-2,5)(3,23)
\drawline(-2,-5)(3,-23)
\put(-1,0){\circle*{0.8}}
\put(0.5,14){\circle*{0.8}}
\put(0.5,-14){\circle*{0.8}}
\put(1.9,19){\circle*{0.8}}
\put(1.9,-19){\circle*{0.8}}
\put(4.9,19){\vector(-1,0){2}}
\put(3.5,14){\vector(-1,0){2}}
\spline(3.5,14)(4.1,14)(5.2,14.1)(6.7,15.5)
(6.9,17.5)(5.7,18.9)(4.9,19)
\put(4.9,-19){\vector(-1,0){2}}
\put(3.5,-14){\vector(-1,0){2}}
\spline(3.5,-14)(4.1,-14)(5.2,-14.1)(6.7,-15.5)
(6.9,-17.5)(5.7,-18.9)(4.9,-19)
}
\put(57.5,40){
\drawline(0,-18)(0,18)
\put(0,12){\circle*{0.8}}
\put(0,8){\circle*{0.8}}
\put(0,0){\circle*{0.8}}
\put(0,-8){\circle*{0.8}}
\put(0,-12){\circle*{0.8}}
}
\put(75,40){
\drawline(0,-16)(0,16)
\put(0,12){\circle*{0.8}}
\put(0,6){\circle*{0.8}}
\put(0,0){\circle*{0.8}}
\put(0,-6){\circle*{0.8}}
\put(0,-12){\circle*{0.8}}
}
\put(92.5,40){
\drawline(0,-18)(0,18)
\put(0,12){\circle*{0.8}}
\put(0,4){\circle*{0.8}}
\put(0,0){\circle*{0.8}}
\put(0,-4){\circle*{0.8}}
\put(0,-12){\circle*{0.8}}
}
\put(110,40){
\drawline(1,-13)(1,13)
\drawline(2,7)(-3,23)
\drawline(2,-7)(-3,-23)
\put(1,4){\circle*{0.8}}
\put(1,0){\circle*{0.8}}
\put(1,-4){\circle*{0.8}}
\put(-1,16.7){\circle*{0.8}}
\put(-1,-16.7){\circle*{0.8}}
}
\put(127.5,40){
\drawline(1,-13)(1,13)
\drawline(2,7)(-3,23)
\drawline(2,-7)(-3,-23)
\put(1,4){\circle*{0.8}}
\put(1,0){\circle*{0.8}}
\put(1,-4){\circle*{0.8}}
\put(-1,16.7){\circle*{0.8}}
\put(-1,-16.7){\circle*{0.8}}
}
\put(145,40){
\drawline(-2,24)(2,12)
\drawline(2,16)(-2,4)
\drawline(-1.2,8)(-1.2,-8)
\drawline(-2,-4)(2,-16)
\drawline(2,-12)(-2,-24)
\put(0,18){\circle*{0.8}}
\put(0,10){\circle*{0.8}}
\put(-1.2,0){\circle*{0.8}}
\put(0,-10){\circle*{0.8}}
\put(0,-18){\circle*{0.8}}
}
\end{picture}
\end{ex}

\bigskip\bigskip

\pagebreak


\begin{thebibliography}{...............}

\setlength{\parskip}{1mm}

\bibitem[AOV08]{AOV08} {\sc D.\ Abramovich, M.\ Olsson, A.\ Vistoli}, 
{\it Tame stacks in positive characteristic},
Ann.\ Inst.\ Fourier 58 (2008), 1057-–1091,
arXiv:math/0703310.

\bibitem[BB11a]{BB11a} {\sc V.\ Batyrev, M.\ Blume}, 
{\it The functor of toric varieties associated with Weyl chambers 
and Losev-Manin moduli spaces}, Tohoku Math.\ J.\ 63 (2011),
581--604, arXiv:0911.3607.

\bibitem[BB11b]{BB11b} {\sc V.\ Batyrev, M.\ Blume}, 
{\it On generalisations of Losev-Manin moduli spaces 
for classical root systems},
Pure and Applied Mathematics Quarterly 7 (2011)
(Special Issue: In memory of Eckart Viehweg), 1053--1084,
arXiv:0912.2898.

\bibitem[BCS05]{BCS05} {\sc L.\ Borisov, L.\ Chen, G.\ Smith},
{\it The orbifold Chow ring of toric Deligne-Mumford stacks},
J.\ Amer.\ Math.\ Soc.\ 18 (2005), 193--215, 
arXiv:math/0309229.

\bibitem[Co95a]{Co95a} {\sc D.\ Cox}, {\it The homogeneous coordinate 
ring of a toric variety}, J.\ Algebraic Geom.\ 4 (1995), 17-–50,
arXiv:alg-geom/9210008.

\bibitem[Co95b]{Co95b} {\sc D.\ Cox}, {\it The functor of a smooth 
toric variety}, Tohoku Math.\ J.\ 47 (1995), 251--262,
arXiv:alg-geom/9312001.

\bibitem[EGA]{EGA} {\sc A.\ Grothendieck, J.\ Dieudonn\'e},
{\it \'El\'ements de G\'eom\'etrie Alg\'ebrique}, 
Publ.\ Math.\ IHES 4,8,11,17,20,24,28,32 (1960-1967).

\bibitem[FMN10]{FMN10} {\sc B.\ Fantechi, E.\ Mann, F.\ Nironi},
{\it Smooth toric Deligne-Mumford stacks}, 
J.\ reine angew.\ Math.\ 648 (2010), 201--244,
arXiv:0708.1254.

\bibitem[Gi]{Gi} {\sc J.\ Giraud}, 
{\it Cohomologie non ab\'elienne}, Grundlehren math.\ Wiss.\ 179, 
Springer-Verlag, Berlin - Heidelberg - New York, 1971.

\bibitem[Iw07]{Iw07} {\sc I.\ Iwanari},
{\it Integral Chow rings of toric stacks}, 
arXiv:0705.3524.

\bibitem[Ka93]{Ka93} {\sc M.\ Kapranov}, {\it Chow quotients of
Grassmannians I}, Adv.\ Soviet Math.\ 16 (1993), 29--110,
arXiv:alg-geom/9210002.

\bibitem[Kl05]{Kl05} {\sc S.\ Kleiman}, {\it The Picard scheme},
in {\it Fundamental Algebraic Geometry}, Math.\ Surveys 
Monogr.\ 123, Amer.\ Math.\ Soc., Providence, RI, 2005, 
237–-321, arXiv:math/0504020.

\bibitem[Kn83]{Kn83} {\sc F.\ Knudsen}, {\it The projectivity of the
moduli space of stable curves II: The stacks $M_{g,n}$},
Math.\ Scand.\ 52 (1983), 161--199.

\bibitem[LM00]{LM00} {\sc A.\ Losev, Yu.\ Manin}, 
{\it New Moduli Spaces of Pointed Curves and Pencils of Flat 
Connections}, Michigan Math.\ J.\ 48 (2000), 443--472, 
arXiv:math/0001003.

\bibitem[LMB]{LMB} {\sc G.\ Laumon, L.\ Moret-Bailly}, 
{\it Champs alg\'ebriques}, 
Springer-Verlag, Berlin - Heidelberg - New York, 2000.

\bibitem[Pe08]{Pe08} {\sc F.\ Perroni}, {\it A note on toric 
Deligne-Mumford stacks}, Tohoku Math.\ J.\ 60 (2008), 441--458,  
arXiv:0705.3823.

\bibitem[SGA4(3)]{SGA4(3)} {\sc A.\ Grothendieck et al.}, 
{\it S\'eminaire de g\'eom\'etrie alg\'ebrique, 
Th\'eorie des topos et cohomologie \'etale des sch\'emas}, Tome 3,
Lecture Notes in Mathematics 305, Springer-Verlag, 
Berlin - Heidelberg - New York, 1973.

\bibitem[Vi05]{Vi05} {\sc A.\ Vistoli}, 
{\it Grothendieck topologies, fibered categories and descent theory}, 
in {\it Fundamental Algebraic Geometry}, Math.\ Surveys Monogr.\ 123, 
Amer.\ Math.\ Soc., Providence, RI, 2005, 1–-104, 
arXiv:math/0412512.

\bigskip\bigskip

\end{thebibliography}
\end{document}